\documentclass{article}
\usepackage[utf8]{inputenc}
\usepackage{xcolor}
\usepackage{geometry}
\usepackage{amsmath, amssymb}
\usepackage{amsthm}
\usepackage{amssymb}
\usepackage{graphicx}
\usepackage{ mathrsfs }
\newtheorem{theorem}{Theorem}
\newtheorem{definition}{Definition}
\newtheorem{corollary}{Corollary}
\newtheorem{proposition}{Proposition}
\newtheorem{lemma}{Lemma}
\newtheorem{remark}{Remark}
\newtheorem{example}{Example}
\usepackage[colorlinks=true, allcolors=blue]{hyperref}
\usepackage{cleveref}
\usepackage{algorithm}
\usepackage{algorithmicx}
\usepackage{algcompatible}
\usepackage{float}

\newtheorem{assumptionenv}{Assumption}
\newcounter{assumption}

\Crefname{assumptionenv}{assumption}{assumptions}
\Crefname{assumptionenv}{Assumption}{Assumptions}

\newenvironment{assumption}
  {\refstepcounter{assumption}\begin{assumptionenv}\label{assumption:\theassumption}}
  {\end{assumptionenv}}

\usepackage{tocloft}

\usepackage{titlesec}

\titleformat{\subsection}
  {\normalfont\large\bfseries}
  {\thesubsection}
  {1em}
  {}

\usepackage{lipsum}

\usepackage{caption}
\usepackage{graphicx}
\usepackage{tikz}
\usepackage{float}
\usepackage{subcaption}
\usepackage{appendix}
\usepackage{comment}
\usepackage{enumitem}
\usepackage[colorinlistoftodos]{todonotes}
\usetikzlibrary{arrows.meta, shapes.geometric}

\newcommand{\R}{\mathbb{R}}

\newcommand{\rank}{\textrm{rank }}

\def\R{\mathbb{R}}

\newcommand{\wassspace}{{(\mathcal P_2(\mathbb R^n),W_2)}} 

\newcommand{\Nalpha}{N_{\alpha}}
\newcommand{\Nbeta}{N_{\beta}}

\newcommand{\FNa}{F_{\Nalpha}}                   
\newcommand{\FNb}{F_{\Nbeta}}                    
\newcommand{\FdagNa}{F^{\dagger}_{\Nalpha}}      
\newcommand{\FdagNb}{F^{\dagger}_{\Nbeta}}       
\newcommand{\dwm}{d_{\textrm{WM}}}
\newcommand{\dwmhat}{\hat d_{\textrm{WM}}}

\newcommand{\Pmeas}{\mathcal{P}}                 

\begin{document}

\title{Wasserstein Mahalanobis Distances for Recovering Latent Geometry}
\author{
Chuxiangbo Wang\footnote{Department of Mathematics, University of North Carolina at Chapel Hill, email: \{chuxianw, cmoosm\}$@$unc.edu}, Shiying Li\footnote{Department of Mathematics, University of Nebraska at Lincoln, email: sli82$@$unl.edu}, and Caroline Moosm\"uller$^{\ast}$
}
\date{}

\maketitle

\begin{abstract}

The Mahalanobis distance is a fundamental covariance-adapted metric for multivariate data and plays a central role in recovering latent geometry from nonlinear observations. We extend this principle from vector-valued data to probability measures by introducing a Wasserstein Mahalanobis distance. Our construction replaces Euclidean displacement vectors with optimal transport displacement fields and local covariance matrices with covariance operators defined on Wasserstein tangent spaces.

We show that this construction inherits the geometry-recovery property underlying nonlinear independent component analysis. In particular, for Gaussian measures with common covariance transformed by a smooth nonlinear pushforward, the proposed Wasserstein Mahalanobis distance approximates the classical Mahalanobis distance between the transformed latent means. The correspondence is exact for affine transformations and holds up to controlled higher-order error terms for general smooth transformations. These results establish a distribution-valued analog of classical Mahalanobis geometry and provide theoretical support for covariance-adapted learning directly in Wasserstein space. Numerical experiments confirm the theoretical predictions and demonstrate accurate recovery of latent geometric structure.

\emph{Keywords: Mahalanobis distance, Wasserstein distance, latent geometry recovery}

\emph{MSC: 49Q22, 62H20, 68T09}
\end{abstract}

\tableofcontents

\section{Introduction}
The Mahalanobis distance, introduced by P. C. Mahalanobis in 1936 \cite{mahalanobis1936generalized}, is a covariance-adapted notion of distance for multivariate data. Unlike Euclidean distance, it accounts for the local variability of the data by rescaling directions according to the covariance structure. When the covariance matrix is symmetric positive definite, the Mahalanobis distance is equivalent to Euclidean distance under an invertible linear transformation, commonly referred to as the whitening transformation \cite{de2000mahalanobis}. Such distances have been widely used in multivariate data analysis and machine learning, including classification \cite{de2000mahalanobis}, outlier detection \cite{rousseeuw1990unmasking,leys2018detecting}, and distance metric learning \cite{xing2002distance,weinberger2009distance}. More generally, local Mahalanobis-type distances \cite{Roberto2013LocalMahalanobis} allow the covariance structure to vary with the base point and have proven useful for recovering intrinsic geometric structure from data. This local viewpoint has been used in nonlinear data analysis, for example to recover intrinsic variables from nonlinear observations \cite{singer08}, and in data-driven modeling of high-dimensional time-series data generated by lower-dimensional intrinsic variables \cite{talmon2013empirical}.

A particularly influential example is the nonlinear independent component analysis (ICA) framework of \cite{singer08}. There, observations are assumed to arise from an unknown nonlinear transformation of latent variables. Although Euclidean distances between observations are distorted by the transformation, local covariance information can be used to construct a Mahalanobis-type distance that recovers the intrinsic geometry of the latent space. From this perspective, the  Mahalanobis distance provides a mechanism for recovering latent geometric structure from transformed observations.

Many modern datasets, however, are naturally represented not as vectors but as probability measures or point clouds. Examples include cell populations in single-cell genomics \cite{mathews18}, collections of word embeddings in text analysis \cite{zhang2010understanding}, and image distributions \cite{rubner2000earth}. In such settings, the basic data object is itself a distribution, and the geometric information of interest is encoded in how mass is distributed rather than in the location of a single point. This raises a natural question:

\emph{Can the geometry-recovery principles underlying local Mahalanobis distance be extended from point-valued data to distribution-valued data?}

Optimal transport provides a natural framework for addressing this question. The quadratic Wasserstein distance measures similarity between probability measures through the cost of transporting mass, and its associated Riemannian-type structure endows the space of probability measures with well-defined tangent spaces \cite{otto2000geometry,Villani2009}. In particular, optimal transport maps generate displacement functions that play the role of tangent vectors in Wasserstein space. This suggests a natural analog of the local Mahalanobis construction: replace Euclidean displacement vectors by optimal transport displacement functions and replace covariance matrices by covariance operators acting on Wasserstein tangent vectors.

This viewpoint leads to a Wasserstein analog of the local Mahalanobis framework. Given a collection of probability measures, we construct local covariance operators from optimal transport displacement fields between neighboring measures. These covariance operators define a Wasserstein Mahalanobis distance that adapts to the local geometry of the measure space in the same spirit that classical Mahalanobis distance adapts to local covariance structure in Euclidean space.

Our motivation comes directly from the geometry-recovery setting of \cite{singer08}. Suppose that probability measures are generated from latent parameters and subsequently transformed by an unknown nonlinear pushforward map. Can a distance constructed solely from the transformed measures (and their neighborhoods) recover the geometry of the latent parameter space? We investigate this question for Gaussian families, where optimal transport maps admit explicit descriptions and the resulting geometry can be analyzed theoretically. 

Our main result shows that, for Gaussian measures with common covariance, the proposed Wasserstein Mahalanobis distance recovers the classical Mahalanobis geometry associated with the transformed mean manifold. More precisely, if
$$
\mu_i=f_{\sharp}\mathcal{N}(m_i,\Sigma),
$$
then the Wasserstein Mahalanobis distance between $\mu_1$ and $\mu_2$ agrees with the corresponding Mahalanobis distance between $f(m_1)$ and $f(m_2)$ up to a controlled error. 
Informally, our main result is the following
\begin{equation}\label{intro:main_result}
\dwm^2(f_{\sharp}\mathcal{N}(m_i,\Sigma),f_{\sharp}\mathcal{N}(m_j,\Sigma)) = d_M^2(f(m_i),f(m_j))+O(\|m_i-m_j\|^2), 
\end{equation}
where $d_M$ is the Mahalanobis distance of \cite{singer08} and $\dwm$ is the Wasserstein Mahalanobis that we define in \Cref{sec:Mahalanobis distance in Wasserstein spaces}. { Note that, when $f$ is an orthogonal  transformation, $W_2\bigl(f_{\sharp}\mathcal{N}(m_i,\Sigma),f_{\sharp}\mathcal{N}(m_j,\Sigma)\bigr)
=\lVert f(m_i)-f(m_j)\rVert =
\lVert m_i-m_j\rVert$. Under this assumption, our earlier work \cite{LMW2025linear} used linear ICA based on the Wasserstein distance $W_2$ to recover statistically independent latent variables. More generally,  we investigated settings in which the Wasserstein distance between observed measures is controlled, up to a small error, by the Euclidean distances between their associated latent parameters. The present work considers more general transformations and instead studies the Wasserstein Mahalanobis and related approximation results. }

The covariance operators built from Wasserstein tangent vectors recover the same local geometry that classical Mahalanobis distances recover in the point-valued setting. This provides a distribution-level analog of the geometric consistency results underlying nonlinear ICA and local Mahalanobis learning.

Related work on Mahalanobis Wasserstein distances has largely focused on combining Mahalanobis-type weighting with one-dimensional Wasserstein geometry. Verde and Irpino \cite{verde2008comparing}, for example, introduced a Mahalanobis--Wasserstein distance for histogram-valued data using quantile representations. Similar quantile-based constructions have been used in applications such as voltage-based phase mapping \cite{lima2025robust}. Other works incorporate Mahalanobis-type structure directly into optimal transport costs \cite{kerdoncuff2020metric, zhang2024mahalanobis, lin2024small} or into related sliced-Wasserstein constructions \cite{bonet2025sliced}. In contrast, our objective is not to modify the ground transport costs, but to extend the covariance-based geometric principles of local Mahalanobis distance to Wasserstein space through covariance operators of optimal transport displacement fields. 

The remainder of the paper is organized as follows. In \Cref{sec:preliminary}, we review background on optimal transport and Mahalanobis distance. In \Cref{sec:Mahalanobis distance in Wasserstein spaces}, we introduce the Wasserstein Mahalanobis distance, while \Cref{Special approximation results} analyzes its geometric and approximation properties, with particular emphasis on Gaussian measures. In \Cref{sec:numerics}, we present numerical experiments illustrating the proposed construction and its geometry-recovery behavior.

\section{Preliminaries} \label{sec:preliminary}

\subsection{Mahalanobis distances in Euclidean space} \label{sec:preliminaries:mahalanobis}

Given two points $x_i,x_j\in\mathbb{R}^n$, the Euclidean distance $\|x_i-x_j\|$ measures displacement using the same scale in every coordinate direction. In other words, its balls  $B(x_j, \delta) = \{x:\|x-x_j\|\le \delta\}$ are round. However, the local spread of a dataset may have very different scales in different directions. For example, nearby points may form an elongated cloud rather than a round one, meaning that the data vary substantially in some directions but only slightly in others.   In such cases, it is natural to measure the displacement $x_i-x_j$ relative to the local variability of the data around $x_j$, which leads to Mahalanobis-type distances based on a local covariance matrix. For each point  $x_j \in \mathbb{R}^n$, let
\begin{equation}\label{euclidean neiboorhood}
     N_{x_j} = \{ x_{j_1}, x_{j_2}, \ldots, x_{j_K} \} 
\end{equation}
be the set of $K$ nearby points around $x_j$, which is viewed  as a local neighborhood set describing how the nearby data distribute around $x_j$. The empirical local covariance matrix at $x_j$ with respect to $N_{x_j}$ is 
\begin{equation}\label{local coavriance matrix}
   C_j = \frac{1}{K-1} \sum_{m=1}^K (x_{j_m} - x_j)(x_{j_m} - x_j)^T.
\end{equation}
The locally centered Mahalanobis distance \cite{Roberto2013LocalMahalanobis} is defined as
\begin{equation*}
        d_{LM}({x_i}, {x_j}) = \sqrt{({x_i} - {x_j})^T C_{j}^\dagger ({x_i} - {x_j})},
    \end{equation*}
where $C_j^\dagger$ denotes the pseudoinverse matrix 
 when $C_j$ is singular (for example,  when $N_{x_j}$ lie in a lower dimensional subspace).  
Note that $d_{LM}$ is generally not symmetric as it only uses the covariance at $x_j$. To remove the asymmetry in $d_{LM}(x_i,x_j)$, we follow \cite{singer08}'s construction by combining the local covariance information at both $x_i$ and $x_j$. 
Let $N_{x_i}$ and $N_{x_j}$ be neighborhood sets of $x_i$ and $x_j$, with associated local covariance matrices
$C_i$ and $C_j$. 
We define the symmetric locally centered Mahalanobis distance by

\begin{align}\label{symetric locally centered Mahalanobis distance}
    d^2_M(x_i, x_j) &= \frac{1}{2} \left[ (x_i - x_j)^T C_j^{\dagger} (x_i - x_j) + (x_j - x_i)^T C_i^{\dagger} (x_j - x_i) \right] \\ \nonumber
    &= \frac{1}{2} (x_i - x_j)^T (C_j^{\dagger} + C_i^{\dagger}) (x_i - x_j).
\end{align}
By construction, $d_M(x_i,x_j)=d_M(x_j,x_i)$ and $d_M(x_i,x_j)\ge 0$. However, $d_M$ is in general only a semimetric and may fail the triangle inequality. Since the Mahalanobis distance depends on a choice of neighborhoods $N_{x_i}$ and $N_{x_j}$, in what follows, we always mention which neighborhoods have been used to compute it.

A key motivation for the symmetric locally centered Mahalanobis distance $d_M$ is that it can recover meaningful geometric information from data that have been transformed. Let $\mathcal{M}_X, \mathcal{M}_Y $ be two smooth manifolds, suppose we can observe data points $y_i\in \mathcal{M}_Y\subset\mathbb{R}^m$ that are obtained from unknown latent points $x_i\in\mathcal{M}_X\subset\mathbb{R}^n$ through an unknown smooth map

\begin{equation*}
    y_i=f(x_i) \qquad f:\mathcal{M}_X\to\mathcal{M}_Y.
\end{equation*}
Although the Euclidean distance $\|y_i-y_j\|$ may be distorted from $\|x_i-x_j\|$ by the mapping $f$, the symmetric locally centered Mahalanobis distance $d_M(y_i,y_j)$ computed using neighborhoods of $y_i$ and $y_j$ can, under certain local assumptions, approximate the Euclidean distance $\|x_i-x_j\|$ in the latent space:

\begin{theorem}[\cite{singer08}]\label{singer's result}
		Let $f: \mathcal{M}_X \to \mathcal{M}_Y $ be a smooth mapping between smooth manifolds, where $\mathcal{M}_X \subset \mathbb{R}^n$ and $\mathcal{M}_Y \subset \mathbb{R}^m$. Assume $d = \dim(\mathcal{M}_X) = \dim(\mathcal{M}_Y)$. Let $ x_i, x_j \in \mathcal{M}_X$ and let $\delta > 0$. Define the $\delta$-ball $B(x,\delta)$ by
        \begin{equation*}
            B(x,\delta)=\{z\in\mathbb{R}^n:\|z-x\|\le \delta\}.
        \end{equation*}
Let $B_{x_i,\delta}=\{x_{i,1},\dots,x_{i,k}\}$, where the points $\{x_{i,\ell}\}_{\ell=1}^k$ are sampled uniformly from $B(x_i,\delta)$; similarly define $B_{x_j,\delta}$. The observations $\{y_{i,\ell}=f(x_{i,\ell}):x_{i,\ell} \in B_{x_i,\delta}\}$ lie approximately inside a small ellipsoid $\mathcal{Q}_{y_i, \delta} = f(B_{x_i,\delta} \cap \mathcal{M}_X)$ centered at $y_i = f(x_i)$; a similar statement holds for $\mathcal{Q}_{y_j, \delta}$. Then, for $\delta$ sufficiently small, 
		
\begin{equation*}
\|x_j - x_i\|^2 = \frac {\delta^2}{d + 2} d_M^2(y_j, y_i)  + O(\|x_j - x_i\|^4),
\end{equation*}
where the Mahalanobis distance is computed with respect to $\mathcal{Q}_{y_i, \delta}$ and $ \mathcal{Q}_{y_j, \delta}$.

\end{theorem}

\begin{remark}\label{remark:singer_special}
In the special case where $f$ is an affine transformation, the approximation in \Cref{singer's result} becomes exact. The proof in \cite{singer08} shows that the error term in the approximation arises from higher-order Taylor terms of the inverse map. When $f$ is affine, these higher-order terms vanish. Hence
\begin{equation*}
    \|x_j-x_i\|^2
    =
    \frac{\delta^2}{d+2} d_M^2(y_j,y_i).
\end{equation*}
\end{remark}

\subsection{Optimal transport and Wasserstein space}
In this paper, we work with probability measures and aim to generalize the symmetric locally centered Mahalanobis distance to Wasserstein space using the optimal transport (OT) framework. We introduce some basic definitions needed for the Wasserstein space and refer the reader to standard references \cite{peyre19,Villani2009} for further background.

Let $\Pmeas(\mathbb{R}^n)$ denote the set of Borel probability measures on $\mathbb{R}^n$. We consider the space of probability measures with finite second moment,
\begin{equation*}
\Pmeas_2(\mathbb{R}^n)
=
\left\{
\mu\in\Pmeas(\mathbb{R}^n)
:
\int_{\mathbb{R}^n}\|x\|^2 d\mu(x)<\infty
\right\}.
\end{equation*}
Let $\alpha,\beta\in\Pmeas_2(\mathbb{R}^n)$. The quadratic Monge formulation defines the 2-Wasserstein distance by
\begin{equation}\label{eq:W2_monge}
W_2^2(\alpha,\beta)
=
\min_{T: T_\sharp\alpha=\beta}
\int_{\mathbb{R}^n}\|T(x)-x\|^2 d\alpha(x),
\end{equation}
where $T:\mathbb{R}^n\to\mathbb{R}^n$ is measurable and $T_\sharp\alpha$ denotes the push-forward of $\alpha$ by $T$, defined by
\begin{equation*}
    T_\sharp\alpha(A)=\alpha(T^{-1}(A)),
\end{equation*}
for every measurable set $A\subseteq\mathbb{R}^n$. More general cost function, for example,  $\|\cdot\|^p$,  can be used in \eqref{eq:W2_monge} to define a Wasserstein-type distance.  In  this paper, we focus on the \emph{Wasserstein space} $\wassspace$. When $\alpha$ is absolutely continuous, \eqref{eq:W2_monge} admits a unique minimizer $T_\alpha^\beta$ (up to an additive constant) \cite[Theorem~2.12]{villani2003topics}. This minimizer is called the optimal transport map from $\alpha$ to $\beta$, and in this case $T_\alpha^\beta\in L^2(\alpha)$ (here and throughout, all $L^2$ spaces are understood to be over $\mathbb R^n$). In this paper, we work with the Monge formulation only (as opposed to the \emph{Kantorovich formulation} \cite{Villani2009}), meaning that  we assume enough regularity of the measures such that optimal transport maps exist.

We define the associated OT displacement function from $\alpha$ to $\beta$ by
\begin{equation}\label{eq:displacement_def}
g_\alpha^\beta = T_\alpha^\beta-\mathrm{id},
\end{equation}
where $\mathrm{id}$ denotes the identity map on $\mathbb{R}^n$.
The Wasserstein space $\wassspace$ has more structure than a general metric space: it has a formal Riemannian structure, following Otto's interpretation \cite{otto2000geometry}. In this geometric view, displacement functions \eqref{eq:displacement_def} represent directions of motion along Wasserstein geodesics. The OT displacement $g_\alpha^\beta$ is interpreted as the tangent vector at $\alpha$ pointing toward $\beta$ (\Cref{fig:tangent vector}).
\begin{figure}[H]
    \centering
    \includegraphics[width=0.3\linewidth]{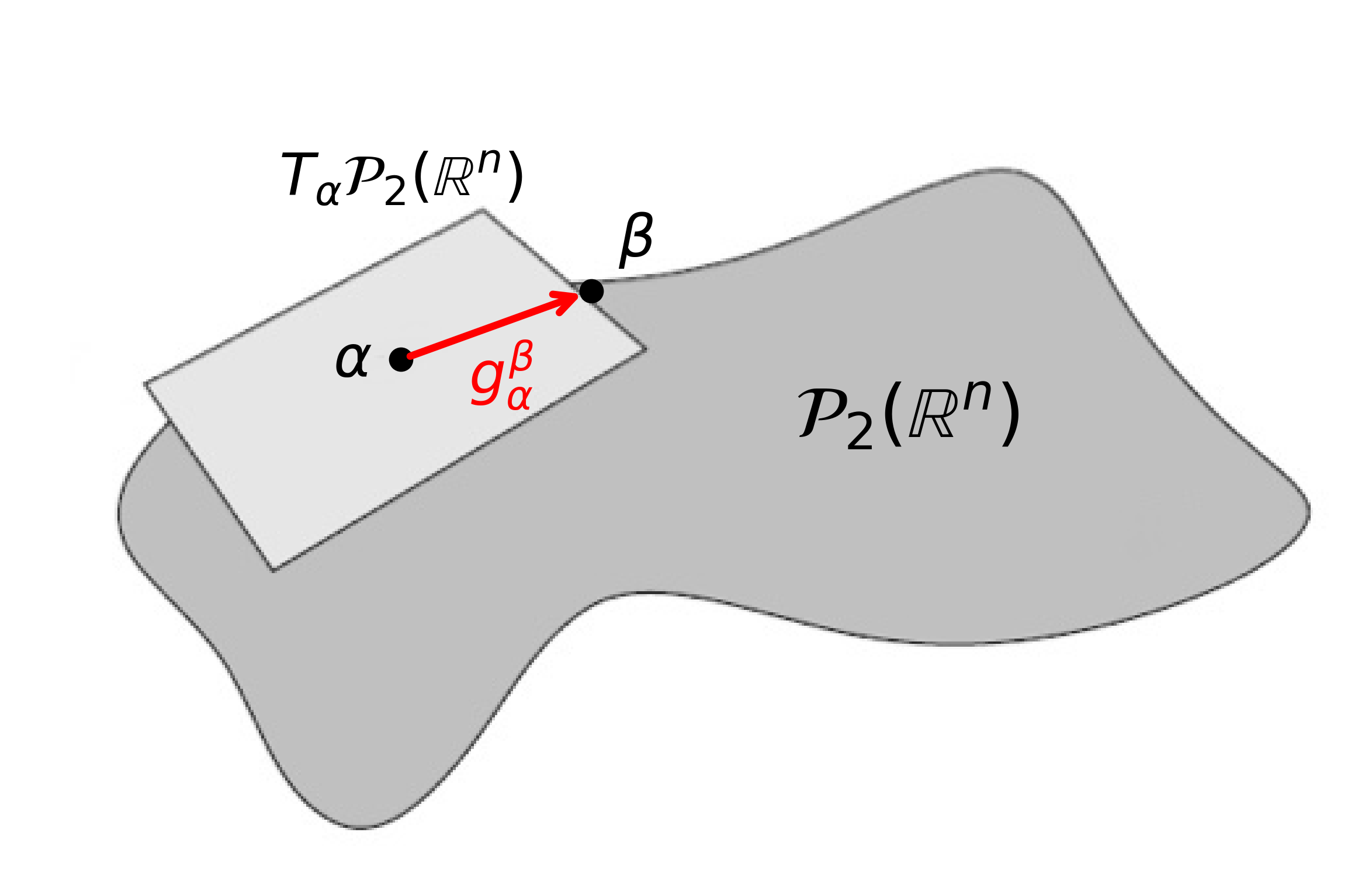}
    \caption{OT displacement function $g_\alpha^\beta = T_\alpha^\beta-\mathrm{id}$ viewed as a tangent vector in $T_\alpha \mathcal P_2(\mathbb R^n)$, the tangent space at $\alpha$.}
    \label{fig:tangent vector}
\end{figure}

\section{Mahalanobis distance in Wasserstein spaces}\label{sec:Mahalanobis distance in Wasserstein spaces}

In \Cref{sec:preliminaries:mahalanobis}, the Mahalanobis distance $d_M$ \eqref{symetric locally centered Mahalanobis distance} in the Euclidean space,
where the distance $x_i-x_j$ is measured relative to the local covariance matrices
$C_i$ and $C_j$ built from neighborhoods $N_{x_i}$ and $N_{x_j}$. In the Wasserstein setting, ``points'' become measures in $\mathcal P_2(\mathbb R^n)$, and the natural analog
of a displacement vector is the OT displacement function \eqref{eq:displacement_def}. Consequently, the form of an empirical covariance matrix will become a finite rank operator
acting on the displacement function in the appropriate $L^2$-space.

Throughout this section, we work in the Wasserstein space $\wassspace$ and introduce a Wasserstein version of Mahalanobis distance, denoted by $\dwm$. Let $\alpha\in\mathcal P_2(\mathbb R^n)$, and $\{\alpha_i\}_{i=1}^K\subset\mathcal P_2(\mathbb R^n)$
be a set of measures in the small neighborhood of $\alpha$, where neighborhoods are understood with respect to the topology of the 2-Wasserstein distance $W_2$. For each neighbor measure $\alpha_i$, denote the OT displacement from $\alpha$ to $\alpha_i$ by 
\begin{equation}\label{eq:displacement_mu_mui}
g_\alpha^{\alpha_i}=T_\alpha^{\alpha_i}-\mathrm{id}\in L^2(\alpha),
\qquad i=1,\dots,K.
\end{equation}
We collect these local displacement functions 
\begin{equation}\label{eq:def:displacement set}
\Nalpha = \{g_\alpha^{\alpha_i}\}_{i=1}^K \subset L^2(\alpha).
\end{equation}
Similarly to $N_{x_j}$ \eqref{euclidean neiboorhood} is used to construct the empirical covariance matrix \eqref{local coavriance matrix},
the set $\Nalpha$ will be used to describe the local structure around $\alpha$ and to construct an empirical covariance operator.
To define such an operator, we use the $L^2(\alpha)$-inner product,
\begin{equation*}
\langle \psi,\varphi\rangle_\alpha = \int_{\mathbb R^n} \psi(x)\cdot \varphi(x)  d\alpha(x), \qquad \psi,\varphi \in L^2(\alpha),
\end{equation*}
where $\cdot$ denotes the Euclidean inner product on $\mathbb R^n$. Similar to the construction in \cite{Hamm2025ManifoldLearningW2}, we define the empirical covariance operator $\FNa:L^2(\alpha)\to L^2(\alpha)$ by
\begin{equation}\label{eq:FValpha_def}
\FNa
=
\frac{1}{K}\sum_{i=1}^K g_\alpha^{\alpha_i}  \langle g_\alpha^{\alpha_i},\cdot\rangle_\alpha.
\end{equation}
\begin{remark}
$\FNa$ is finite-rank. Moreover, 
   every self-adjoint finite-rank linear operator in a Hilbert space can be expressed in the form of \eqref{eq:FValpha_def}. 
\end{remark}
As in the Euclidean case,
$\FNa$ may not be invertible (for example, when the displacements in $\Nalpha$ lie in a lower-dimensional subspace). Therefore, we work with its pseudo-inverse.
To make the pseudo-inverse explicit, form the symmetric positive semi-definite matrix $B\in\mathbb R^{K\times K}$ with entries
\begin{equation}\label{eq:B_def}
B_{ij}=\big\langle \FNa g_\alpha^{\alpha_j},  g_\alpha^{\alpha_i}\big\rangle_\alpha,
\qquad i,j=1,\dots,K,
\end{equation}
and denote by $B^\dagger\in\mathbb R^{K\times K}$ the  pseudo-inverse of $B$.
Then the pseudo-inverse operator $\FdagNa$ can be defined as
\begin{equation}\label{eq:Fdagger_Nalpha}
\FdagNa
=
\sum_{i=1}^K\sum_{j=1}^K B^\dagger_{ij}  g_\alpha^{\alpha_i} \langle g_\alpha^{\alpha_j},\cdot\rangle_\alpha,
\end{equation}
where $B^\dagger_{ij}$ denotes the $(i,j)$ entry of $B^\dagger$.

\begin{remark}

By choosing an orthonormal basis $U = \{u_i\}_{i=1}^p \subseteq L^2(\alpha)$, $p\leq K$, for the subspace spanned by $\Nalpha$, $\FdagNa$ \eqref{eq:Fdagger_Nalpha} can equivalently be defined by \begin{equation}\label{eq:Finv_U}
    F^\dagger_{U} = \sum_{i=1}^p \sum_{j=1}^p A^{-1}_{ij} u_i \langle u_j, \cdot \rangle_\alpha,
\end{equation}
where 
\begin{equation}\label{eq:Def_Aij}
    A_{ij} = \langle \FNa u_j, u_i \rangle_\alpha,
\end{equation}
is symmetric positive definite.
\end{remark}

\begin{remark}
The finite-rank operator $\FdagNa$ defined on $L^2(\alpha)$ has properties analogous to its matrix counterparts on
$\mathbb R^n$. In particular $\FdagNa$ behaves like the Moore--Penrose pseudo-inverse matrix appearing in the Euclidean Mahalanobis distance \eqref{symetric locally centered Mahalanobis distance}. By standard arguments, it is invariant under change of basis, satisfies the Moore pseudo-inverse properties, and is also the solution to the least squares problem; see, for example, \cite[Section~2]{Groetsch1977inverse}.
\end{remark}

We now define the Mahalanobis distance between two measures.
\begin{definition}[Wasserstein Mahalanobis distance]\label{def:Wass-Maha}
 Let $\alpha,\beta\in\mathcal P_2(\mathbb R^n)$, and let $\{\alpha_i\}_{i=1}^{N}$ and $\{\beta_j\}_{j=1}^{K}$
be neighborhood measures near $\alpha$ and $\beta$, respectively. Construct $\Nalpha$ \eqref{eq:def:displacement set}, $\FNa$ \eqref{eq:FValpha_def}, and $\FdagNa$ \eqref{eq:Fdagger_Nalpha}
as above (and similarly, do so for $\Nbeta$, $\FNb$, and $\FdagNb$). Let $g_\alpha^\beta=T_\alpha^\beta-\mathrm{id}\in L^2(\alpha)$. Then the \emph{Wasserstein Mahalanobis distance} $\dwm$ is defined by

\begin{equation}\label{Mahalanobis in Wasserstein}
\dwm^2(\alpha, \beta)
=
\frac{1}{2}\left(
\langle g_\alpha^\beta, \FdagNa g_\alpha^\beta \rangle_\alpha
+
\langle g_\beta^\alpha, \FdagNb g_\beta^\alpha \rangle_\beta
\right).
\end{equation}   
\end{definition}

\Cref{def:Wass-Maha} is designed to mirror the Euclidean Mahalanobis distance $d_M$ \eqref{symetric locally centered Mahalanobis distance}, with local covariance matrices replaced by
finite-rank covariance operators built from OT displacement functions. We refer to $\dwm$ as a Wasserstein Mahalanobis distance, although, for the same reason as $d_M$, it should be understood as a semimetric rather than a genuine metric. A key motivation for introducing the Wasserstein Mahalanobis distance $\dwm$
\eqref{Mahalanobis in Wasserstein} is the same one that motivates $d_M$ in Euclidean space, which is to build a distance that can recover meaningful latent information from transformed data. We explore this
motivation in the next subsection through several special approximation results for Gaussian data.

\section{Approximation results on a Gaussian submanifold}\label{Special approximation results}
In the Euclidean setting, \Cref{singer's result} shows that when we have observations
$y_i=f(x_i)$ from an unknown transformation $f$ of latent points $\{x_i\}$, $d_M(y_i,y_j)$ can approximate $\|x_i-x_j\|$ under certain local
assumptions. In this section, we consider a similar approximation result in the Wasserstein space. Here, each latent parameter is associated with a probability measure, and the observed objects are pushforwards of these measures under an unknown map. The question is whether $\dwm$ between the observed pushforward measures can recover the Euclidean distance between the underlying latent parameters. We set up the problem as follows.

For a positive definite matrix $\Sigma$, we consider the submanifold
\begin{equation}\label{eq:Gaussian_sub}
    \mathcal{M}_{\Sigma} = \{\mathcal{N}(m,\Sigma): m\in \mathbb{R}^n\}
\end{equation}
in the Wasserstein space. We are motivated by \cite{CarlenGangbo2003} to study the Wasserstein Mahalanobis distance on a Gaussian submanifold. 

$\mathcal{M}_{\Sigma}$ naturally carries the Euclidean geometry:
$W_2(\mathcal{N}(m_i,\Sigma),\mathcal{N}(m_j,\Sigma)) = \|m_i-m_j\|$. Our goal is to show that $\dwm(f_{\sharp}\alpha,f_{\sharp}\beta) \approx d_M(f(m_{\alpha}),f(m_{\beta})),$
where $\alpha,\beta \in \mathcal{M}_{\Sigma}$ with means $m_{\alpha}$ and $m_{\beta}$, respectively.

\subsection{Affine transformation}

We first consider the case where $f$ is an affine map. This choice is motivated by the fact affine maps preserve Gaussian structure, so
$f_\sharp\alpha$ and $f_\sharp\beta$ remain Gaussian, thus
$\dwm (f_\sharp \alpha, f_\sharp \beta)$ can be computed explicitly.

\begin{theorem}\label{Gaussian Linear Theorem}
    Consider an affine function $f: \mathbb{R}^n \to \mathbb{R}^n$. Let 
$\alpha = \mathcal{N}(m_\alpha, \Sigma)$ and $\beta = \mathcal{N}(m_\beta, \Sigma)$ be Gaussian distributions on $\mathbb{R}^n$. Let $\alpha_i=\mathcal N(m_{\alpha_i},\Sigma)$, $i=1,\dots,N$ and $\beta_j=\mathcal N(m_{\beta_j},\Sigma)$, $j=1,\dots,K$, where the means $m_{\alpha_i}$, $m_{\beta_j}$ are uniformly distributed in ${B}(m_\alpha,\delta)$ and ${B}(m_\beta,\delta)$, respectively. Then
 \begin{equation*}
          d_M(f(m_\alpha),f(m_\beta))   = \dwm(f_{\sharp}\alpha,f_{\sharp}\beta),
 \end{equation*}
 where the distances are computed with respect to the neighborhoods $\{{f(m_{\alpha_i})}\}_{j=1}^N,\{{f(m_{\beta_i})}\}_{j=1}^K$ and $\{{f_\sharp}\alpha_i\}_{i=1}^N,\{{f_\sharp}\beta_j\}_{j=1}^K$, respectively.
\end{theorem}

\begin{proof}
Since $f$ is an affine transformation, $f_\sharp$ applied to Gaussian distributions will produce Gaussian distributions. Therefore, one can obtain a closed form of the OT displacement functions. The result therefore follows from direct computations. See \Cref{Proof: Gaussian linear theorem} for details. 
\end{proof}

\begin{remark}\label{remark:affine_mahalanobis_and_euclidean}
Consider the same setup as in \Cref{Gaussian Linear Theorem}. Then \Cref{remark:singer_special} implies
    \begin{equation}
         \|m_\alpha - m_\beta\|^2 = \frac{\delta^2}{n+2}\dwm^2({{f}}_{\sharp}\alpha,{{f}}_{\sharp}\beta).
    \end{equation}
This means that in the setting of affine functions, the Wasserstein Mahalanobis distance exactly recovers the means that were used to create the latent Gaussian data.
\end{remark}

\subsection{General transformations}\label{sec:nonlinear transformation}

We now move beyond the affine case and consider a general map $f:\R^n\to\R^n$, which may not preserve
Gaussian structure. As a consequence, the OT displacement $g_{f_\sharp\alpha}^{f_\sharp\beta}$ generally does not have a closed form, and therefore, the explicit calculation of $\dwm(f_\sharp\alpha,f_\sharp\beta)$ is difficult. 

Our goal is to prove the main theorem (\Cref{theorem:main}), which shows an approximation of $\dwm$ by $d_M$ up to second order on the Gaussian submanifold $\mathcal{M}_{\Sigma}$. The theorem builds on four technical assumptions, which are discussed in more detail in \Cref{sec:tech_assump}.

\begin{theorem}[Main Result]\label{theorem:main}
     Let $\alpha = \mathcal{N}(m_\alpha, \Sigma)$, $\beta = \mathcal{N}(m_\beta, \Sigma)$ be Gaussian distributions on $\mathbb R^n$. For $\delta>0$ let 
 \begin{equation*}
     \alpha_i=\mathcal N(m_{\alpha_i},\Sigma),
\qquad i=1,\dots,N,
 \end{equation*}
and
\begin{equation*}
    \beta_j=\mathcal N(m_{\beta_j},\Sigma),
\qquad j=1,\dots,K,
\end{equation*}
where the means $m_{\alpha_i}$ and $m_{\beta_j}$  are  uniformly distributed in ${B}(m_\alpha,\delta)$ and ${B}(m_\beta,\delta)$,  respectively. Let $f:\mathbb{R}^n\to\mathbb{R}^n$ satisfy \Cref{assump:regularity}, and let $N$, $K$, and $\delta$ satisfy \Cref{Assumption: empirical sampling}. Denote by $\tilde f^{m_\beta}$ and $\tilde f^{m_\alpha}$ the linear approximations of $f$ at $m_\beta$ and $m_\alpha$, respectively, i.e.\
\begin{equation*}
\tilde f^{m_\beta}(x)
=
f(m_\beta)+J_f(m_\beta)(x-m_\beta) \quad \text{and}
\quad
\tilde f^{m_\alpha}(x)
=
f(m_\alpha)+J_f(m_\alpha)(x-m_\alpha).
\end{equation*}
Suppose that there exists $p\geq 2$ such that \Cref{displacement function assumption} holds for $\alpha$ and $\beta$ with $r_\alpha, r_\beta \geq \max\{\|m_\alpha - m_\beta\|, \delta\}$. 
If $p=2$, suppose in addition that  \Cref{assuption:dispacement_local} holds for both $\alpha$ and $\beta$. Then
\begin{align*}
\dwm^2({{f}}_{\sharp}\alpha,{{f}}_{\sharp}\beta) 
 = & d_M^2(f(m_\alpha),f(m_\beta))  + O \left(\frac{ \|m_\beta-m_\alpha\|^{2}}{\delta^{3-\frac{p}{2}}}\right)\\&
+
O\left(\frac{\|m_\alpha - m_\beta\|^2}{\delta^2}\left[H(\tilde f^{m_\beta}_\sharp\beta,\ f_\sharp\beta)+H(\tilde f^{m_\alpha}_\sharp\alpha,\ f_\sharp\alpha) \right]\right),
\end{align*}
where the distances are constructed
 with respect to the neighborhoods $\{{f_\sharp}\alpha_i\}_{i=1}^N,\{{f_\sharp}\beta_j\}_{j=1}^K$ and $\{{f(m_{\alpha_i})}\}_{i=1}^N,\{{f(m_{\beta_j})}\}_{j=1}^K$. Here,
 $H$ denotes the Hellinger distance; see \Cref{def:hellinger_standard}.
\end{theorem}
\begin{remark}
The details of the required assumptions are discussed in \Cref{sec:tech_assump}. In summary,
\Cref{assump:regularity,Assumption: empirical sampling} in \Cref{theorem:main} are basic regularity and sampling requirements. While the displacement conditions in \Cref{displacement function assumption,assuption:dispacement_local} are more restrictive, we demonstrate in \Cref{sec:tech_assump} that they can be satisfied in nontrivial settings.
\end{remark}
\begin{remark}\label{remark:dwm_and_euclidean}
    In \Cref{remark: displacement function} we show that \Cref{theorem:main} always holds for $p=2$. In combination with \Cref{singer's result}, this allows for the following conclusion
\begin{equation*}
\|m_\alpha - m_\beta\|^2 =  \frac{\delta^2}{n+2}\dwm^2({{f}}_{\sharp}\alpha,{{f}}_{\sharp}\beta)  + O(\|m_\alpha - m_\beta\|^2).
  \end{equation*}
    We provide numerical results for this identity in \Cref{experiment:3A}.
\end{remark}
We next consider a limiting case $\Sigma\to0_{n\times n}$ of \Cref{remark:dwm_and_euclidean}, in which Gaussian distributions collapse to point masses. In this limit, we expect our approximation to agree with  \Cref{singer's result}.
\begin{corollary}\label{cor:point_mass_limit}
Under the same setting and notation as in \Cref{theorem:main}, we have
\begin{equation*}
\frac{\delta^2}{n+2}
\lim_{\quad\Sigma\to 0_{n\times n}}
\dwm^2(f_\sharp\alpha,f_\sharp\beta)
=
\|m_\alpha-m_\beta\|^2
+
O\left(
\delta\|m_\alpha-m_\beta\|^2
\right)+
O\left(
\|m_\alpha-m_\beta\|^4
\right).
\end{equation*}
\begin{proof}
    As $\Sigma\to 0_{n\times n}$,  the Hellinger terms in \Cref{theorem:main} vanish and \Cref{displacement function assumption} holds with $p=4$. The result then follows from \Cref{theorem:main} and \Cref{singer's result}. See \Cref{Proof:point_mass_limit} for details.
\end{proof}
\end{corollary}
\begin{remark}
    Since $\delta$ can be chosen arbitrarily small, we may take $\delta \leq \|m_\alpha-m_\beta\|^2$. Then $O(\delta\|m_\alpha-m_\beta\|^2)$ is absorbed into $O(\|m_\alpha-m_\beta\|^4)$, and the result agrees with \Cref{singer's result}. A corresponding numerical experiment is presented in \Cref{experiment:3B}.
\end{remark}

\subsubsection{Technical assumptions}\label{sec:tech_assump}
To prove our main result (\Cref{theorem:main}), we need four technical assumptions, which we now discuss.

\begin{assumption}\label{assump:regularity}
    $f:\mathbb{R}^n\to\mathbb{R}^n$ is a smooth diffeomorphism  with uniformly bounded first and
second derivatives.
\end{assumption}
\begin{assumption}\label{Assumption: empirical sampling}
Let $f:\mathbb{R}^n\to\mathbb{R}^n$, let $m\in\mathbb{R}^n$ and let $\delta>0$. Let $\{m_{i}\}_{i=1}^K$ be uniformly distributed in $B(m,\delta)$. Define
\begin{equation*}
C_{f(m)}
=
\frac{1}{K}
\sum_{j=1}^{K}
\big(f(m_{j})-f(m)\big)
\big(f(m_{j})-f(m)\big)^\top.
\end{equation*}
Let $X$ be a random variable uniformly distributed in $B(m,\delta)$, and define
\begin{equation*}
\mathrm{Cov}(f(X))
=
\mathbb E\Big[
(f(X)-f(m))
(f(X)-f(m))^\top
\Big].
\end{equation*}
We assume that the number of empirical samples $K$ is always sufficiently large and $\delta$ is small, so the sampling error between the covariance matrices can be arbitrarily small and satisfies
\begin{equation}\label{eq:sampling_error_small_enough}
\big\|C_{f(m)}-\mathrm{Cov}(f(X))\big\|_2
\le \frac{1}{2}\lambda_{\min}^+\big(\mathrm{Cov}(f(X))\big),
\end{equation}
where $\lambda_{\min}^+(\mathrm{Cov}(f(X)))$ denotes the smallest positive eigenvalue of $\mathrm{Cov}(f(X))$. 
\end{assumption}
\begin{assumption}\label{displacement function assumption}
Let $f:\mathbb R^n\to\mathbb R^n$ be differentiable and let $p\geq 2$. For all $\nu=\mathcal N(m_\nu,\Sigma)\in\mathcal M_\Sigma$ \eqref{eq:Gaussian_sub}, there exists $r_\nu>0 $ such that for every $\rho=\mathcal N(m_\rho,\Sigma)\in\mathcal M_\Sigma$
satisfying
\begin{equation*}
\|m_\rho-m_\nu\|\leq r_\nu,
\end{equation*}
we have
\begin{equation}
\label{eq:disp_error_assump}
\left\|
g^{f_\sharp\rho}_{f_\sharp\nu}
-
g^{\tilde f_\sharp\rho}_{\tilde f_\sharp\nu}
\right\|_{L^2(f_\sharp\nu)}^2
= O(
\|m_\rho-m_\nu\|^p),
\end{equation}
where $\tilde f$ denotes the linear approximation of $f$ at $m_\nu$ \eqref{eq:affine-approx}.
\end{assumption}

The next Lemma shows that the displacement error bound \eqref{eq:disp_error_assump} always holds with exponent $p=2$ when $f$ satisfy  \Cref{assump:regularity}.
\begin{lemma}\label{remark: displacement function}
    Using the same notation as in \Cref{displacement function assumption} and assume that $f$ satisfies  \Cref{assump:regularity}, the displacement function error satisfies
    \begin{equation*}
\left\|
g^{f_\sharp\rho}_{f_\sharp\nu}
-
g^{\tilde f_\sharp\rho}_{\tilde f_\sharp\nu}
\right\|_{L^2(f_\sharp\nu)}^2
=
O\left(\|m_\rho-m_\nu\|^2\right).
    \end{equation*}
\end{lemma}
\begin{proof}
Since the OT map 
between the original Gaussians is a simple translation between their means, pushing this translation through 
$f$ yields a transport map whose displacement takes the form 
$f(x+m_\rho-m_\nu)-f(x)$, which is an upper bound of $g^{f_\sharp\rho}_{f_\sharp\nu}$ in the $L^2$-norm. The rest of the proof follows from using Taylor expansions and the triangle inequality. See \Cref{proof:remark: displacement function} for details. 
\end{proof}

We now give a simple one dimensional example where the OT map can be found explicitly. This allows us to compute the displacement error directly and verify \Cref{remark: displacement function}, namely, that the second order leading term is sharp.
\begin{example}\label{example:displacement_function}
    Let $f(x)=x+x^3$, $\alpha=\mathcal N(m_\alpha,\sigma^2)$, $\beta=\mathcal N(m_\beta,\sigma^2)$, and $h=m_\beta-m_\alpha$. Denote by $\tilde f^{m_\alpha}$ the linear
approximation of $f$ at $m_\alpha$ \eqref{eq:affine-approx}. Since $f$ is increasing and defined on $\mathbb{R}$, the OT map between $f_\sharp\alpha$ and $f_\sharp\beta$ can be found explicitly. A direct computation gives
\begin{equation*}
\|g_{f_\sharp\alpha}^{f_\sharp\beta}
-g_{\tilde f^{m_\alpha}_\sharp\alpha}^{\tilde f^{m_\alpha}_\sharp\beta}\|_{L^2(f_\sharp\alpha)}^2
=
9(4m_\alpha^2\sigma^2+3\sigma^4)h^2
+54m_\alpha\sigma^2h^3
+(9m_\alpha^2+15\sigma^2)h^4
+6m_\alpha h^5
+h^6.
\end{equation*}
Thus the leading error is $O(h^2)$. The detailed calculation for this example is in \Cref{computation_of_example_1}.

Note that if $\sigma\to0$, the measures collapse to point masses and the leading nonzero term becomes $O(h^4)$, matching an equivalent Euclidean example which we discuss in \Cref{example: euclidean example} in \ref{sec: Appendix B}. 
\end{example}

\begin{assumption}\label{assuption:dispacement_local}
Let $\nu=\mathcal N(m_\nu,\Sigma)\in\mathcal M_\Sigma$ \eqref{eq:Gaussian_sub}, let $\delta > 0$ and denote by
\begin{equation*}
\nu_i=\mathcal N(m_{\nu_i},\Sigma) \in \mathcal M_\Sigma, \quad i=1,\dots,K,
\end{equation*}
the neighboring Gaussians of $\nu$,
where the means $m_{\nu_i}$ are
uniformly distributed in $B(m_\nu,\delta)$. Let $f:\mathbb R^n\to\mathbb R^n$ be differentiable with non-singular Jacobian. Denote by
$\tilde f$ the linear approximation of $f$ at $m_\nu$ \eqref{eq:affine-approx}. We assume that there exists $c_\xi$ such that
\begin{equation}\label{eq:local_displacement_bound_p=2}
\left\|
g^{f_\sharp\nu_i}_{f_\sharp\nu}
-
g^{\tilde f_\sharp\nu_i}_{\tilde f_\sharp\nu}
\right\|_{L^2(f_\sharp\nu)}^2
\leq
c_\xi\|m_{\nu_i}-m_\nu\|^2,
\end{equation}
where
\begin{equation*}
0<c_\xi<
\min\left\{
1,
\left[
\frac{\|J_f(m_\nu)\|_2^4}
{15(n+2)^2\kappa_\nu^4}
\right]^2,
\left[
\frac{\|J_f(m_\nu)\|_2}
{15(n+2)^2\kappa_\nu^4}
\right]^2
\right\}
\end{equation*}
and $\kappa_\nu
=
\frac{\|J_f(m_\nu)\|_2}
{\sigma_{\min}(J_f(m_\nu))}$
denotes the local condition number of $J_f$ at $m_\nu$.
\end{assumption}
\begin{remark}
    When $f$ satisfies \Cref{assump:regularity}, \Cref{remark: displacement function} guarantees that \eqref{eq:local_displacement_bound_p=2} holds for some constant $c_\xi>0$; here \Cref{assuption:dispacement_local} imposes an additional smallness requirement on the constant $c_\xi$. This condition is only used in the proof of \Cref{lem:BC_comparison}, and its purpose is to ensure that within the local $\delta$ neighborhood and in the case $p=2$, the nonlinear displacement functions are sufficiently close to their linearized version. The
assumption becomes less restrictive when the condition number $\kappa_\beta$ is small and $\|J_f(m_\nu)\|_2$ is not too small (not strongly contractive), since in this case a larger local displacement error is allowed.
\end{remark}
The following example shows that the smallness condition in \Cref{assuption:dispacement_local} can be satisfied for a nonlinear transformation. We modify the function $f(x,y)=(x+y^3,y-x^3)$ of \cite{singer08}, for which \Cref{assuption:dispacement_local} does not hold, by forcing it to be less ``nonlinear''.

\begin{example}\label{exp:assuption_D}
Fix $\delta>0$ and let $\nu=\mathcal N(m_\nu,I_2)$, $\nu_i=\mathcal N(m_{\nu_i},I_2)$, where $m_\nu,m_{\nu_i}\in[0,1]^2$ and $m_{\nu_i}\in B(m_\nu,\delta)$. Let 
\begin{equation*}
f(x,y)
=
\left(
x+\epsilon y^3,
y-\epsilon x^3
\right).
\end{equation*}
Then \Cref{assuption:dispacement_local} is satisfied when
\begin{equation*}
\epsilon^2
\Bigl[
99
+54\sqrt{2}\delta
+33\delta^2
+6\sqrt{2}\delta^3
+\delta^4
\Bigr] <\frac{
\left(
1-
\frac{1}{2}
\sqrt{
81\epsilon^4+36\epsilon^2
}
\right)^4
}
{
57600
\left(
1+9\epsilon^2
+
\frac{1}{2}
\sqrt{
81\epsilon^4+36\epsilon^2
}
\right)^3
}.
\end{equation*}
As $\epsilon\to0$, the left-hand side converges to $0$, whereas the right-hand side converges to $\frac{1}{57600}$. Therefore, \Cref{assuption:dispacement_local} can always be satisfied by choosing $|\epsilon|$ sufficiently small. The detailed calculation is in \Cref{computation_of_example_2}.
\end{example}

\subsubsection{Proof of the main result (\texorpdfstring{\Cref{theorem:main}}{Theorem~\ref{theorem:main}})}

We aim at approximating  $\dwm(f_\sharp\alpha,f_\sharp\beta)$ by
$d_M(f(m_{\alpha}),f(m_{\beta}))$. 
Specifically, in the following we compare
\begin{equation}\label{eq:dwm_def_transition}
\dwm^2(f_\sharp\alpha,f_\sharp\beta)
=\frac12\Big(
\big\langle g^{f_\sharp\alpha}_{f_\sharp\beta}, F^\dagger_{N_{f_\sharp\beta}}\big(g^{f_\sharp\alpha}_{f_\sharp\beta}\big)\big\rangle_{f_\sharp\beta}
+
\big\langle g^{f_\sharp\beta}_{f_\sharp\alpha}, F^\dagger_{N_{f_\sharp\alpha}}\big(g^{f_\sharp\beta}_{f_\sharp\alpha}\big)\big\rangle_{f_\sharp\alpha}
\Big),
\end{equation}
to
\begin{align}\label{eq:f-M-symm}
      &d_{M}^2(f(m_\alpha),f(m_\beta))=\frac{1}{2} [ f(m_\alpha) - f(m_\beta) ]^\top [C^\dagger_{f(m_\beta)} + C^\dagger_{f(m_\alpha)}][ f(m_\alpha) - f(m_\beta) ].
\end{align}

To make $\dwm$
analyzable, we consider the first order approximation of $f$ near the relevant means.  Applying Taylor's theorem to $f$ at $m_\beta$, we obtain 

\begin{equation*}
f(x)=f(m_\beta)+J_f(m_\beta)(x-m_\beta)+O(\|x-m_\beta\|^2),
\end{equation*}
and denote its affine approximation as
\begin{equation}\label{eq:affine-approx}
\tilde f^{m_\beta}(x)=f(m_\beta)+J_f(m_\beta)(x-m_\beta).
\end{equation}

We approximate the first inner product in \eqref{eq:dwm_def_transition} by replacing
$f$ with $\tilde f^{m_\beta}$ and aim to find the difference between { a Euclidean Mahalanobis term and a linearized Wasserstein Mahalanobis term:
}

\begin{equation}\label{eq:two_quadratic_forms_beta}
    [ f(m_\alpha) - f(m_\beta) ]^\top C^\dagger_{f(m_\beta)}[ f(m_\alpha) - f(m_\beta) ] \quad
\text{and} \quad
\big\langle g^{\tilde f^{m_\beta}_\sharp\alpha}_{\tilde f^{m_\beta}_\sharp\beta}, 
F^\dagger_{N_{\tilde f^{m_\beta}_\sharp\beta}}
\big(g^{\tilde f^{m_\beta}_\sharp\alpha}_{\tilde f^{m_\beta}_\sharp\beta}\big)
\big\rangle_{\tilde f^{m_\beta}_\sharp\beta},
\end{equation}
where \begin{equation}\label{eq:emp_cov_beta}
    C_{f(m_\beta)}
    =\frac{1}{K}\sum_{i=1}^K
    \big(f(m_{\beta_i})-f(m_\beta)\big)\big(f(m_{\beta_i})-f(m_\beta)\big)^\top, 
\end{equation}
\begin{equation}\label{eq:F_matrix_transition}
   F_{N_{\tilde f^{m_\beta}_\sharp\beta}} (h) = \frac{1}{K}\sum_{i=1}^K
    J_f(m_\beta)(m_{\beta_i}-m_\beta)\big(J_f(m_\beta)(m_{\beta_i}-m_\beta)\big)^\top\Biggr. \int_{\mathbb{R}^n}h(y)   d \tilde f^{m_\beta}_\sharp\beta (y).
\end{equation}
{A direct computation (cf. \Cref{lemma: WM-inner-covariance}) shows that the second term in \eqref{eq:two_quadratic_forms_beta} reduces to 
\begin{equation}\label{eq:inner product constant form}
    \big\langle g^{\tilde f^{m_\beta}_\sharp\alpha}_{\tilde f^{m_\beta}_\sharp\beta}, 
F^\dagger_{N_{\tilde f^{m_\beta}_\sharp\beta}}
\big(g^{\tilde f^{m_\beta}_\sharp\alpha}_{\tilde f^{m_\beta}_\sharp\beta}\big)
\big\rangle_{\tilde f^{m_\beta}_\sharp\beta} = (J_f(m_\beta)(m_\alpha-m_\beta))^\top M^\dagger_{N_{\tilde f^{m_\beta}_\sharp\beta}}(J_f(m_\beta)(m_\alpha-m_\beta)),
\end{equation}
where $  M_{N_{\tilde f^{m_\beta}_\sharp\beta}} = \frac{1}{K}\sum_{i=1}^K
    J_f(m_\beta)(m_{\beta_i}-m_\beta)\big(J_f(m_\beta)(m_{\beta_i}-m_\beta)\big)^\top$.
The key idea for comparing the terms in \eqref{eq:two_quadratic_forms_beta} is the following perturbation result for pseudo-inverses, with $A, B$ being the corresponding covariance matrices $C_{f(m_\beta)} $ and $M_{N_{\tilde f^{m_\beta}_\sharp\beta}}$ in  \eqref{eq:emp_cov_beta} and \eqref{eq:F_matrix_transition}, respectively.

\begin{lemma}\label{pseudo inverse lemma}\cite[Theorem 3.3]{stewart1977pseudo-inverse-purturbation}
For $A, B\in \mathbb{C}^{m \times n}$, we have $
    \|B^\dagger - A^\dagger\| \leq \mu \max\{ \| A^\dagger\|_2 \|B^\dagger\|_2 \} \|B-A\|,
$
where $\mu$ is a constant with $\mu \leq 3$, and $\|\cdot\|$ denotes any unitary invariant matrix norm.
\end{lemma}

\begin{corollary}\label{rmk:cov inverse Eulidean}
Let $f:\mathbb R^n\to\mathbb R^n$ satisfy \Cref{assump:regularity}, and $C_{f(m_\beta)}$ in \eqref{eq:emp_cov_beta} satisfies \Cref{Assumption: empirical sampling}. Then 
\begin{equation*}
 \quad  \|C_{f(m_\beta)}^\dagger\|_2 = O \left(\delta^{-2}\right). 
\end{equation*}
Moreover,
\begin{equation*}
\big\|
C_{f(m_\beta)}
-
M_{N_{\tilde f^{m_\beta}_\sharp\beta}}
\big\|_2
=
O(\delta^3),
\end{equation*}
and consequently, by \Cref{pseudo inverse lemma},
\begin{equation*}
\big\|
C_{f(m_\beta)}^\dagger
-
M^\dagger_{N_{\tilde f^{m_\beta}_\sharp\beta}}
\big\|_2
=
O(\delta^{-1}).
\end{equation*}
\end{corollary}
}
\begin{proof}
    The proof is given in \Cref{proof:cov inverse Euclidean}. The same
estimates hold at the base point $m_\alpha$.
\end{proof}

\begin{proposition} \label{linear approx of f approx euclidean mahalanobis}

Let $ \alpha = \mathcal{N}(m_\alpha, \Sigma)$, $\beta = \mathcal{N}(m_\beta, \Sigma)$ be Gaussian distributions on $\mathbb{R}^n$. For $\delta>0$, let $\alpha_i=\mathcal N(m_{\alpha_i},\Sigma)$, $i=1,\dots,N$, and $ \beta_j=\mathcal N(m_{\beta_j},\Sigma)$, $j=1,\dots,K$, where the means $m_{\alpha_i}$, $m_{\beta_j}$  are uniformly distributed in ${B}(m_\alpha,\delta)$ and ${B}(m_\beta,\delta)$,  respectively. Let $f:\mathbb{R}^n\to\mathbb{R}^n$ satisfy \Cref{assump:regularity}, and assume $N,K$ and $\delta$ satisfy \Cref{Assumption: empirical sampling}.
Denote by $\tilde f^{m_\beta}$, $\tilde f^{m_\alpha}$ linear approximations of $f$ at $m_\beta$ and $m_\alpha$, respectively; see \eqref{eq:affine-approx}. Then
 \begin{align*}
        &d_M^2(f(m_\alpha),f(m_\beta)) \\
          =& \frac{1}{2}\left(\langle g^{\tilde{f}^{m_\beta}_\sharp\alpha}_{\tilde{f}^{m_\beta}_\sharp\beta}, F^\dagger_{N_{\tilde{f}^{m_\beta}_\sharp\beta}}(g^{\tilde{f}^{m_\beta}_\sharp\alpha}_{\tilde{f}^{m_\beta}_\sharp\beta}) \rangle_{\tilde{f}^{m_\beta}_\sharp\beta} + \langle g^{\tilde{f}^{m_\alpha}_{\sharp} \beta}_{\tilde{f}^{m_\alpha}_{\sharp} \alpha},   F^\dagger_{N_{\tilde{f}^{m_\alpha}_{\sharp} \alpha}} ( g^{\tilde{f}^{m_\alpha}_{\sharp} \beta}_{\tilde{f}^{m_\alpha}_{\sharp} \alpha} ) \rangle_{\tilde{f}^{m_\alpha}_{\sharp} \alpha}\right) + O\left(\frac{\|m_{\alpha} - m_\beta\|^2}{\delta}\right).
 \end{align*}
 \begin{proof}
 {
Recall that $d_{M}^2(f(m_\alpha),f(m_\beta))=\frac{1}{2} (f(m_\alpha) - f(m_\beta) )^\top (C^\dagger_{f(m_\beta)} + C^\dagger_{f(m_\alpha)})( f(m_\alpha) - f(m_\beta) )$. By \eqref{eq:inner product constant form}, it suffices to bound 
\begin{equation}\label{eq:difference for proposition}
   ( f(m_\alpha) - f(m_\beta) )^\top C^\dagger_{f(m_\beta)}( f(m_\alpha) - f(m_\beta) ) - (J_f(m_\beta)(m_\alpha-m_\beta))^\top M^\dagger_{N_{\tilde f^{m_\beta}_\sharp\beta}}(J_f(m_\beta)(m_\alpha-m_\beta)).
\end{equation}
The symmetric terms can be handled in the same way by linearizing $f$ at $m_\alpha$. The desired bound follows from \eqref{eq:difference for proposition}, which bounds the pseudo-inverse difference $\big\|C_{f(m_\beta)}^\dagger - M^\dagger_{N_{\tilde f^{m_\beta}_\sharp\beta}}
\big\|_2$; cf.\ \Cref{rmk:cov inverse Eulidean}. The details of this computation involve adding and subtracting an intermediate term and applying the triangle inequality, which is shown in \Cref{Proof of:linear approx of f approx euclidean mahalanobis}.
}
\end{proof}
\end{proposition}

\Cref{linear approx of f approx euclidean mahalanobis} finds the gap between the Euclidean Mahalanobis distance and the Wasserstein Mahalanobis distance applied to the linearization of $f$. To connect back to the true $\dwm(f_\sharp\alpha,f_\sharp\beta)$, we need to compare the following inner products:

\begin{equation}\label{eq:inner product between linear and nonlinear}
   \big\langle g^{f_\sharp\alpha}_{f_\sharp\beta}, F^\dagger_{N_{f_\sharp\beta}}\big(g^{f_\sharp\alpha}_{f_\sharp\beta}\big)\big\rangle_{f_\sharp\beta} \quad
\text{and} \quad
\big\langle g^{\tilde f^{m_\beta}_\sharp\alpha}_{\tilde f^{m_\beta}_\sharp\beta}, 
F^\dagger_{N_{\tilde f^{m_\beta}_\sharp\beta}}
\big(g^{\tilde f^{m_\beta}_\sharp\alpha}_{\tilde f^{m_\beta}_\sharp\beta}\big)
\big\rangle_{\tilde f^{m_\beta}_\sharp\beta}.
\end{equation}
We summarize the structure of the remaining argument in \Cref{fig:nonlinear_flowchart}.

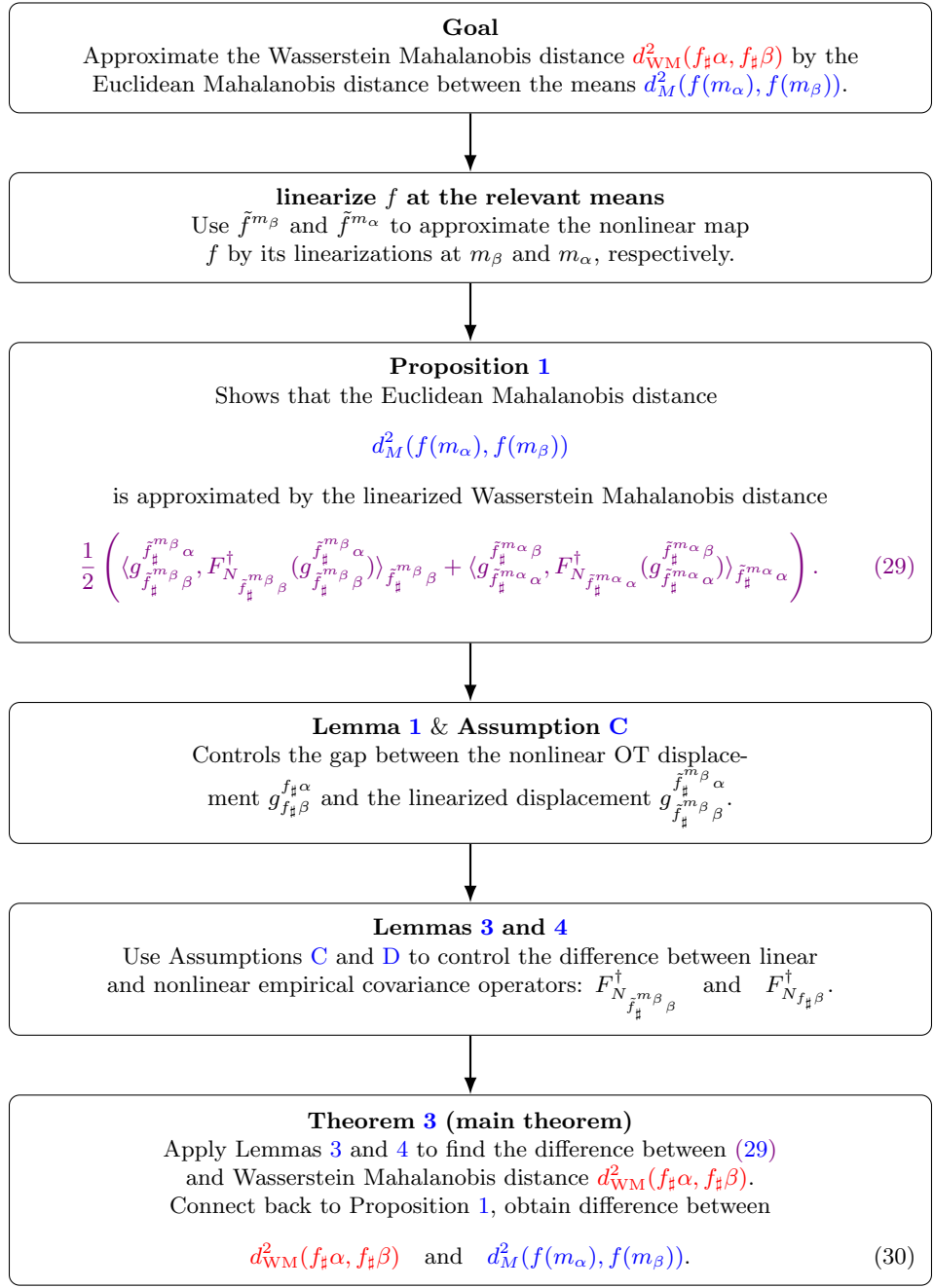
\begin{figure}[H]
\centering
\begin{tikzpicture}[
    >=Latex,
    node distance=8mm,
    every node/.style={font=\small},
    box/.style={
        rectangle,
        rounded corners,
        draw,
        align=center,
        text width=0.8\linewidth,
        inner sep=6pt
    },
    arrow/.style={->, thick}
]

\node[box] (goal) {
\textbf{Goal}\\
Approximate the Wasserstein Mahalanobis distance
{\color{red}{$\dwm^2(f_\sharp\alpha,f_\sharp\beta)$}}
by the Euclidean Mahalanobis distance between the means
{\color{blue}{$d_M^2(f(m_\alpha),f(m_\beta))$}}.
};

\node[box, below=of goal] (step1) {
\textbf{linearize $f$ at the relevant means}\\
Use $\tilde f^{m_\beta}$ and $\tilde f^{m_\alpha}$ to approximate the nonlinear map $f$ by its linearizations at $m_\beta$ and $m_\alpha$, respectively.
};

\node[box, below=of step1] (prop) {
\textbf{\Cref{linear approx of f approx euclidean mahalanobis}}\\
Shows that the Euclidean Mahalanobis distance {\color{blue}{$$d_M^2(f(m_\alpha),f(m_\beta))$$}} is approximated by the
linearized Wasserstein Mahalanobis distance
\color{violet}{ \begin{align}\label{eq:flowchart_eq1}
          \frac{1}{2}\left(\langle g^{\tilde{f}^{m_\beta}_\sharp\alpha}_{\tilde{f}^{m_\beta}_\sharp\beta}, F^\dagger_{N_{\tilde{f}^{m_\beta}_\sharp\beta}}(g^{\tilde{f}^{m_\beta}_\sharp\alpha}_{\tilde{f}^{m_\beta}_\sharp\beta}) \rangle_{\tilde{f}^{m_\beta}_\sharp\beta} + \langle g^{\tilde{f}^{m_\alpha}_{\sharp} \beta}_{\tilde{f}^{m_\alpha}_{\sharp} \alpha},   F^\dagger_{N_{\tilde{f}^{m_\alpha}_{\sharp} \alpha}} ( g^{\tilde{f}^{m_\alpha}_{\sharp} \beta}_{\tilde{f}^{m_\alpha}_{\sharp} \alpha} ) \rangle_{\tilde{f}^{m_\alpha}_{\sharp} \alpha}\right).
 \end{align}}
};

\node[box, below=of prop] (disp) {
\textbf{\Cref{remark: displacement function} $\&$  \Cref{displacement function assumption}}\\
 Controls the gap between the nonlinear OT displacement $g^{{f}_\sharp\alpha}_{{f}_\sharp\beta}$
and the linearized displacement $g^{\tilde{f}^{m_\beta}_\sharp\alpha}_{\tilde{f}^{m_\beta}_\sharp\beta}$.
};

\node[box, below=of disp] (op) {
\textbf{\Cref{lem:BC_comparison,operator differece lemma}}\\
Use \Cref{displacement function assumption,assuption:dispacement_local} to control the difference between linear and nonlinear empirical covariance operators: $F^\dagger_{N_{\tilde{f}^{m_\beta}_\sharp\beta}}\quad \text{and} \quad F^\dagger_{N_{{f}_\sharp\beta}}.$
};

\node[box, below=of op] (main) {
\textbf{Theorem \ref{theorem:main} (main theorem)}\\
Apply \Cref{lem:BC_comparison,operator differece lemma} to find the difference between {\color{violet}{\eqref{eq:flowchart_eq1}}} and Wasserstein Mahalanobis distance
{\color{red}{$\dwm^2(f_\sharp\alpha,f_\sharp\beta)$}}. 

Connect back to \Cref{linear approx of f approx euclidean mahalanobis}, obtain difference between 
\begin{equation}
    {\color{red}\dwm^2(f_\sharp\alpha,f_\sharp\beta)} \quad \text{and} \quad {\color{blue}{d_M^2(f(m_\alpha),f(m_\beta))}}  .
\end{equation}
};

\draw[arrow] (goal) -- (step1);
\draw[arrow] (step1) -- (prop);
\draw[arrow] (prop) -- (disp);
\draw[arrow] (disp) -- (op);
\draw[arrow] (op) -- (main);

\end{tikzpicture}
\caption{Flow chart for nonlinear special approximation section.}
\label{fig:nonlinear_flowchart}
\end{figure}

Notice that the two inner products in \eqref{eq:inner product between linear and nonlinear} are taken in different $L^2$ spaces, however, in  \eqref{eq:inner product constant form}, the corresponding $L^2$ inner product reduces to the integral of a constant scalar. Since both $f_\sharp\beta$ and
$\tilde f^{m_\beta}_\sharp\beta$ are probability measures, the value of
\eqref{eq:inner product constant form} stays unchanged if we take the inner product with respect to $f_\sharp\beta$ instead of
$\tilde f^{m_\beta}_\sharp\beta$, that is
\begin{equation*}
\big\langle g^{\tilde f^{m_\beta}_\sharp\alpha}_{\tilde f^{m_\beta}_\sharp\beta}, 
F^\dagger_{N_{\tilde f^{m_\beta}_\sharp\beta}}
\big(g^{\tilde f^{m_\beta}_\sharp\alpha}_{\tilde f^{m_\beta}_\sharp\beta}\big)
\big\rangle_{\tilde f^{m_\beta}_\sharp\beta} = \big\langle g^{\tilde f^{m_\beta}_\sharp\alpha}_{\tilde f^{m_\beta}_\sharp\beta}, 
F^\dagger_{N_{\tilde f^{m_\beta}_\sharp\beta}}
\big(g^{\tilde f^{m_\beta}_\sharp\alpha}_{\tilde f^{m_\beta}_\sharp\beta}\big)
\big\rangle_{f_\sharp\beta}.
\end{equation*}
Consequently, the main difference between the two inner products in \eqref{eq:inner product between linear and nonlinear}
is driven primarily by how well the true OT displacement $g^{f_\sharp\alpha}_{f_\sharp\beta}$ is approximated by the OT displacement of linearized $f$, namely $g^{\tilde f^{m_\beta}_\sharp\alpha}_{\tilde f^{m_\beta}_\sharp\beta}$.
Moreover, finding the difference between 
\begin{equation}\label{eq:diff_dispalcement}
g^{f_\sharp\alpha}_{f_\sharp\beta} \quad \text{and}\quad g^{\tilde f^{m_\beta}_\sharp\alpha}_{\tilde f^{m_\beta}_\sharp\beta}
\end{equation}
extends to finding the difference between each
OT displacement pair in the neighborhood sets $\{g^{f_\sharp\beta_i}_{f_\sharp\beta}\}_{i=1}^K$ and $\{g^{\tilde f^{m_\beta}_\sharp\beta_i}_{\tilde f^{m_\beta}_\sharp\beta}\}_{i=1}^K,$
which can lead to the difference between
\begin{equation}\label{eq:operater_diff}
    F^\dagger_{N_{f_\sharp\beta}} \quad \text{ and} \quad F^\dagger_{N_{\tilde f^{m_\beta}_\sharp\beta}},
\end{equation} 
as they are constructed from those two sets. We need the general displacement error assumption (\Cref{displacement function assumption}) to first control the difference between \eqref{eq:diff_dispalcement}, which will be used to compute differences between the operators \eqref{eq:operater_diff}. To do that, we first compare the coefficient matrices used in the nonlinear and linearized pseudo-inverse covariance operators.

\begin{lemma}\label{lem:BC_comparison}
Let $\beta = \mathcal{N}(m_\beta, \Sigma)$ be a Gaussian distribution on $\mathbb R^n$, and let $\beta_i=\mathcal N(m_{\beta_i},\Sigma), i=1,\dots,K$, where the means $m_{\beta_i}$ are  uniformly distributed in ${B}(m_\beta,\delta)$. Let $f:\mathbb{R}^n\to\mathbb{R}^n$ satisfy \Cref{assump:regularity}, and assume that $K$, $\delta$ satisfy \Cref{Assumption: empirical sampling}. Let $\tilde{f}^{m_\beta}$ be linear approximation of $f$ at $m_\beta$, see \eqref{eq:affine-approx}. Denote by
\begin{equation*}
N_{\tilde f^{m_\beta}_\sharp\beta} = \Big\{g^{\tilde f^{m_\beta}_\sharp\beta_i}_{\tilde f^{m_\beta}_\sharp\beta}\Big\}_{i=1}^K,
\quad \text{and} \quad
N_{f_\sharp\beta} = \Big\{g^{f_\sharp \beta_i}_{f_\sharp \beta}\Big\}_{i=1}^K,
\end{equation*}
the two sets containing OT displacement functions \eqref{eq:displacement_mu_mui} from $\tilde f^{m_\beta}_\sharp\beta$ to $\tilde f^{m_\beta}_\sharp\beta_i$, and from ${f}_\sharp \beta$ to ${f}_\sharp \beta_i$, respectively. Let $B$ and $C$ be the matrices associated with ${N_{f_\sharp\beta}}$ and ${N_{\tilde f^{m_\beta}_\sharp\beta}}$, respectively, following \eqref{eq:B_def}:

\begin{equation*}
B_{kl}
=
\frac1K\sum_{i=1}^K
\Big\langle g^{f_\sharp\beta_i}_{f_\sharp\beta},\ g^{f_\sharp\beta_l}_{f_\sharp\beta}\Big\rangle_{f_\sharp\beta}
\Big\langle g^{f_\sharp\beta_i}_{f_\sharp\beta},\ g^{f_\sharp\beta_k}_{f_\sharp\beta}\Big\rangle_{f_\sharp\beta},
\end{equation*}
and
\begin{equation*}
C_{kl}
=
\frac1K\sum_{i=1}^K
\Big\langle g^{\tilde f^{m_\beta}_\sharp\beta_i}_{\tilde f^{m_\beta}_\sharp\beta},\ g^{\tilde f^{m_\beta}_\sharp\beta_l}_{\tilde f^{m_\beta}_\sharp\beta}\Big\rangle_{\tilde f^{m_\beta}_\sharp\beta}
\Big\langle g^{\tilde f^{m_\beta}_\sharp\beta_i}_{\tilde f^{m_\beta}_\sharp\beta},\ g^{\tilde f^{m_\beta}_\sharp\beta_k}_{\tilde f^{m_\beta}_\sharp\beta}\Big\rangle_{\tilde f^{m_\beta}_\sharp\beta}.
\end{equation*}
Suppose that there exists $p\geq 2$ such that \Cref{displacement function assumption} holds for $\beta$ with $r_\beta \geq \delta$. If $p=2$, suppose in addition that  \Cref{assuption:dispacement_local} holds for $\beta$. Then, we have the following,
\begin{equation*}
\|C^\dagger\|_2 = O(\delta^{-4}),\quad \|B-C\|_2 = O(\delta^{3+p/2}), \quad \|B^\dagger\|_2 = O(\delta^{-4}).
\end{equation*}

\begin{proof}
We prove the above three estimates separately in \Cref{Proof:lem_BC_comparison}. The main idea is as follows.  First, all displacement functions in $N_{\tilde f^{m_\beta}_\sharp\beta}$ are constant functions, so the matrix $C$ can be written explicitly. The uniform sampling of $m_{\beta_i}$ in $B(m_\beta,\delta)$ and \Cref{Assumption: empirical sampling} together give
\begin{equation*}
    \|C^\dagger\|_2=O(\delta^{-4}).
\end{equation*}
Second, using \Cref{lemma:bounds for B-C}, one can obtain
\begin{equation*}
    |(B-C)_{kl}|
    =
    O(\delta^{3}(\max_i\|g^{f_\sharp\beta_i}_{f_\sharp\beta} - g^{\tilde f^{m_\beta}_\sharp\beta_i}_{\tilde f^{m_\beta}_\sharp\beta}\|_{L^2(f_\sharp\beta)})).
\end{equation*}
Applying \Cref{displacement function assumption} gives
\begin{equation*}
    \|B-C\|_2=O(\delta^{3+p/2}).
\end{equation*}
Finally, since $B$ is a small perturbation of $C$, Weyl's inequality (\Cref{Weyl's inequality}) connects the lower singular value bound from $C$ to $B$, yielding
\begin{equation*}
    \|B^\dagger\|_2=O(\delta^{-4}).\qedhere
\end{equation*}
\end{proof}

\end{lemma}

The lemma above controls the coefficients in the pseudo-inverse operator representations. To compare the operators themselves, we must also account for the fact that they are defined on different $L^2$ spaces.
$F^\dagger_{N_{f_\sharp\beta}}$ is an operator on
$L^2(f_\sharp\beta)$, whereas the linearized 
$F^\dagger_{N_{\tilde f^{m_\beta}_\sharp\beta}}$ is an operator on
$L^2(\tilde f^{m_\beta}_\sharp\beta)$.
In order to obtain a bound on the difference between them,
we first unite these two $L^2$ spaces through a transformation. Define the Radon-Nikodym derivative
\begin{equation*}
w(y)=\frac{d(\tilde f^{m_\beta}_\sharp\beta)}{d(f_\sharp\beta)}(y),
\end{equation*}
which is well-defined (\Cref{lemma:Radon-Nikodym well define}, \ref{sec: Appendix B}). Then define the operator $U$ by
\begin{equation}\label{eq:U_def}
U:\ L^{2}(\tilde f^{m_\beta}_\sharp\beta)\ \longrightarrow\ L^{2}(f_\sharp\beta),
\qquad U\phi(y)=\phi\sqrt{w},
\end{equation}
which is an isometry with inverse
\begin{equation*}
U^{-1}:\ L^{2}(f_\sharp\beta)\ \longrightarrow\ L^{2}(\tilde f^{m_\beta}_\sharp\beta),
\qquad U^{-1}\psi=\frac{\psi}{\sqrt{w}}.
\end{equation*}
We then define the operator $\tilde F^\dagger_{N_{\tilde f^{m_\beta}_\sharp\beta}}$ acting on $L^2(f_\sharp\beta)$ by
\begin{equation}\label{eq:def_Ftilde_dagger}
\tilde F^\dagger_{N_{\tilde f^{m_\beta}_\sharp\beta}}
=
UF^\dagger_{N_{\tilde f^{m_\beta}_\sharp\beta}} U^{-1}
:
L^2(f_\sharp\beta)\to L^2(f_\sharp\beta).
\end{equation}
Since $U$ is an isometry, the operator norm is preserved,
\begin{equation}\label{eq:op_norm_invariance}
\bigl\|\tilde F^\dagger_{N_{\tilde f^{m_\beta}_\sharp\beta}}\bigr\|_{\mathrm{op}}
=
\bigl\|F^\dagger_{N_{\tilde f^{m_\beta}_\sharp\beta}}\bigr\|_{\mathrm{op}}.
\end{equation}
With this preparation, both $F^\dagger_{N_{f_\sharp\beta}}$ and
$\tilde F^\dagger_{N_{\tilde f^{m_\beta}_\sharp\beta}}$ act on the same space
$L^2(f_\sharp\beta)$, so their difference can be computed. 

\begin{lemma}\label{operator differece lemma}
Let $\beta = \mathcal{N}(m_\beta, \Sigma)$ be a Gaussian distribution on $\mathbb R^n$. For $\delta>0$, let $\beta_i=\mathcal N(m_{\beta_i},\Sigma), i=1,\dots,K$, where the means $m_{\beta_i}$ are uniformly distributed in ${B}(m_\beta,\delta)$. Let $f:\mathbb{R}^n\to\mathbb{R}^n$ satisfy \Cref{assump:regularity}, and assume that $K$, $\delta$ satisfy \Cref{Assumption: empirical sampling}. Let $\tilde{f}^{m_\beta}$ be linear approximation of $f$ at $m_\beta$; see \eqref{eq:affine-approx}. Denote by
\begin{equation*}
N_{\tilde f^{m_\beta}_\sharp\beta} = \Big\{g^{\tilde f^{m_\beta}_\sharp\beta_i}_{\tilde f^{m_\beta}_\sharp\beta}\Big\}_{i=1}^K,
\qquad
N_{f_\sharp\beta} = \Big\{g^{f_\sharp \beta_i}_{f_\sharp \beta}\Big\}_{i=1}^K,
\end{equation*}
the sets containing OT displacement functions \eqref{eq:displacement_mu_mui} from $\tilde f^{m_\beta}_\sharp\beta$ to $\tilde f^{m_\beta}_\sharp\beta_i$, and from ${f}_\sharp \beta$ to ${f}_\sharp \beta_i$, respectively. Let $F^\dagger_{N_{f_\sharp\beta}}$, $F^\dagger_{N_{\tilde f^{m_\beta}_\sharp\beta}}$
be the pseudo-inverse empirical covariance operators \eqref{eq:Fdagger_Nalpha} defined by $N_{f_\sharp\beta}$ and $N_{\tilde f^{m_\beta}_\sharp\beta}$, respectively. Let
$U$
be the operator defined in \eqref{eq:U_def} and let
\begin{equation*}
\tilde F^\dagger_{N_{\tilde f^{m_\beta}_\sharp\beta}}
=
UF^\dagger_{N_{\tilde f^{m_\beta}_\sharp\beta}}U^{-1}.
\end{equation*}
Suppose that there exists $p\geq 2$ such that \Cref{displacement function assumption} holds for $\beta$ with $r_\beta \geq \delta$. If $p=2$, suppose in addition that  \Cref{assuption:dispacement_local} holds for $\beta$. Then
\begin{equation*}
\|F^\dagger_{N_{f_\sharp \beta}} -\tilde F^\dagger_{N_{\tilde f^{m_\beta}_\sharp\beta}} \|_{op} 
  =
O(\delta^{p/2-3})
+
O\Big(\delta^{-2} H (\tilde f^{m_\beta}_\sharp\beta,\ f_\sharp\beta)\Big),
\end{equation*}
where $H$ denotes the Hellinger distance (\Cref{def:hellinger_standard}).
\begin{proof}
    From \Cref{Lemma:operater_defference_explicit}, $F^\dagger_{N_{f_\sharp \beta}} -\tilde F^\dagger_{N_{\tilde f^{m_\beta}_\sharp\beta}}$ can be explicitly written as in \eqref{eq:operator_diff_explicit}. Using \Cref{pseudo inverse lemma} and \Cref{lem:BC_comparison}, we obtain
    \begin{equation}
        \|B^\dagger\|_2,\|C^\dagger\|_2=O(\delta^{-4}),
\qquad
\|B^\dagger-C^\dagger\|_2=O(\delta^{p/2-5}).
    \end{equation}
The linearized displacement functions are constant,
\begin{equation}
    \left\|
g^{\tilde f^{m_\beta}_\sharp\beta_i}_{\tilde f^{m_\beta}_\sharp\beta}
\right\|_{L^2(f_\sharp\beta)}
= O(\delta),
\end{equation}
while by \Cref{displacement function assumption}, the displacement difference satisfies $$\|g^{f_\sharp \beta_i}_{f_\sharp \beta} - g^{\tilde f^{m_\beta}_\sharp\beta_i}_{\tilde f^{m_\beta}_\sharp\beta} \|_{L^2(f_\sharp\beta)}
=
O(\delta^{p/2}), \qquad i = 1,\ldots,K.$$ Substituting these into \eqref{eq:operator_diff_explicit} and applying the triangle inequality, the terms involving $B^\dagger-C^\dagger$ and $\xi_i$ are bounded by $O(\delta^{p/2 - 3})$. Finally, by the definition of the Hellinger distance (\Cref{def:hellinger_standard}), the term related to $U$ is bounded by $O(\delta^{-2}H\bigl(\tilde f^{m_\beta}_\sharp\beta,\ f_\sharp\beta\bigr))$.
The detailed computations are given in \Cref{Proof:operator differece lemma}.
\end{proof}
\end{lemma}

With the operator difference controlled in \Cref{operator differece lemma}, we can now compare the nonlinear and linearized inner products in \eqref{eq:inner product between linear and nonlinear}. Applying this comparison at both base measures $\alpha$ and $\beta$, and then averaging the two results, gives a bound on the difference between
$\dwm^2(f_\sharp\alpha,f_\sharp\beta)$
and
$d_M^2(f(m_\alpha),f(m_\beta))$,
which yields our main theorem.

\begin{proof}[Proof of the main result, \Cref{theorem:main}]
The key point here is that there are two main reasons the nonlinear inner product and its linearized
counterpart in \eqref{eq:inner product between linear and nonlinear} are different: the pseudo-inverse empirical covariance operators
are different, and the displacement functions themselves are also different. The operator difference is controlled by \Cref{operator differece lemma}, while the displacement difference is controlled by \Cref{displacement function assumption}. Combining these two gives \Cref{lemma:main_theorem_lemma}.

Applying \Cref{lemma:main_theorem_lemma} at the base measure $\beta$, we obtain
\begin{align*}
&\big\langle g^{f_\sharp\alpha}_{f_\sharp\beta}, F^\dagger_{N_{f_\sharp\beta}}\big(g^{f_\sharp\alpha}_{f_\sharp\beta}\big)\big\rangle_{f_\sharp\beta} \\
&= \big\langle g^{\tilde f^{m_\beta}_\sharp\alpha}_{\tilde f^{m_\beta}_\sharp\beta},  F^\dagger_{N_{\tilde f^{m_\beta}_\sharp\beta}} \big(g^{\tilde f^{m_\beta}_\sharp\alpha}_{\tilde f^{m_\beta}_\sharp\beta}\big) \big\rangle_{\tilde f^{m_\beta}_\sharp\beta}+ O(\delta^{p/2-3}\|m_\beta-m_\alpha\|^2)
+
O \left(\frac{\|m_\beta-m_\alpha\|^{2}}{\delta^2}  H(\tilde f^{m_\beta}_\sharp\beta,\ f_\sharp\beta)\right).
\end{align*}
The same argument at the base measure $\alpha$ gives
\begin{align*}
    &\langle g_{f_\sharp \alpha}^{f_\sharp \beta}, F^\dagger_{N_{f_\sharp \alpha}}(g_{f_\sharp \alpha}^{f_\sharp \beta})\rangle_{f_\sharp \alpha}\\ &= \langle g^{\tilde{f}_\sharp^{m_\alpha}\beta}_{\tilde{f}_\sharp^{m_\alpha}\alpha}, {F}^\dagger_{{N_{\tilde f^{m_\alpha}_\sharp\alpha}}}(g^{\tilde{f}_\sharp^{m_\alpha}\beta}_{\tilde{f}_\sharp^{m_\alpha}\alpha})\rangle_{\tilde f^{m_\alpha}_\sharp\alpha}+ O(\delta^{p/2-3}\|m_\beta-m_\alpha\|^2)
+
O \left(\frac{\|m_\beta-m_\alpha\|^{2}}{\delta^2}  H(\tilde f^{m_\alpha}_\sharp\alpha,\ f_\sharp\alpha) \right).
\end{align*}
It follows that 
\begin{align*}
\dwm^2({{f}}_{\sharp}\alpha,{{f}}_{\sharp}\beta) &= \frac{1}{2} (\langle g_{f_\sharp \alpha}^{f_\sharp \beta}, F^\dagger_{N_{f_\sharp \alpha}}(g_{f_\sharp \alpha}^{f_\sharp \beta})\rangle_{f_\sharp \alpha} +  \langle g^{f_\sharp \alpha}_{f_\sharp \beta}, F^\dagger_{N_{f_\sharp \beta}}(g^{f_\sharp \alpha}_{f_\sharp \beta})\rangle_{f_\sharp \beta})\\
& = \frac{1}{2}(\langle g^{\tilde{f}_\sharp^{m_\alpha}\beta}_{\tilde{f}_\sharp^{m_\alpha}\alpha}, F^\dagger_{N_{\tilde f^{m_\alpha}_\sharp\alpha}}(g^{\tilde{f}_\sharp^{m_\alpha}\beta}_{\tilde{f}_\sharp^{m_\alpha}\alpha})\rangle_{\tilde f^{m_\alpha}_\sharp\alpha} + \langle g^{\tilde{f}^{m_\beta}_\sharp\alpha}_{\tilde{f}^{m_\beta}_\sharp\beta}, F^\dagger_{N_{\tilde f^{m_\beta}_\sharp\beta}}(g^{\tilde{f}^{m_\beta}_\sharp\alpha}_{\tilde{f}^{m_\beta}_\sharp\beta})\rangle_{\tilde f^{m_\beta}_\sharp\beta} )\\
&+ O \left(\frac{ \|m_\beta-m_\alpha\|^{2}}{\delta^{3-\frac{p}{2}}}\right)
+
O\left(\frac{\|m_\alpha - m_\beta\|^2}{\delta^2} \left[H(\tilde f^{m_\beta}_\sharp\beta,\ f_\sharp\beta)+H(\tilde f^{m_\alpha}_\sharp\alpha,\ f_\sharp\alpha) \right]\right).
\end{align*}
Finally, by \Cref{linear approx of f approx euclidean mahalanobis}, we obtain the desired result,
\begin{align*}
\dwm^2({{f}}_{\sharp}\alpha,{{f}}_{\sharp}\beta) 
& = d_M^2(f(m_\alpha),f(m_\beta)) \\&+ O \left(\frac{ \|m_\beta-m_\alpha\|^{2}}{\delta^{3-\frac{p}{2}}}\right)
+
O\left(\frac{\|m_\alpha - m_\beta\|^2}{\delta^2} \left[H(\tilde f^{m_\beta}_\sharp\beta,\ f_\sharp\beta)+H(\tilde f^{m_\alpha}_\sharp\alpha,\ f_\sharp\alpha) \right]\right). \qedhere
\end{align*}
\end{proof}

\section{Numerical Experiments}\label{sec:numerics}
Our goal is to numerically test the result from \Cref{theorem:main}. For two Gaussian measures
$\alpha=\mathcal{N}(m_\alpha,\Sigma)$ and 
$\beta=\mathcal{N}(m_\beta,\Sigma)$
we compare the quantities
\begin{equation*}
\dwm^2(f_\sharp\alpha,f_\sharp\beta) \qquad \text{and} \qquad d_M^2(f(m_\alpha),f(m_\beta)).
\end{equation*}
Recall that in \Cref{theorem:main} the difference is controlled by two terms 
\begin{equation}\label{eq:Numerical_error}
Err
= O \left(\frac{ \|m_\beta-m_\alpha\|^{2}}{\delta^{3-\frac{p}{2}}}\right)
+
O\left(\frac{\|m_\alpha - m_\beta\|^2}{\delta^2}\left[H(\tilde f^{m_\beta}_\sharp\beta,\ f_\sharp\beta)+H(\tilde f^{m_\alpha}_\sharp\alpha,\ f_\sharp\alpha) \right]\right).
\end{equation}
In our previous analysis, we established that the case $p=2$ always holds (see \Cref{remark: displacement function}), so that both terms have the same order
\begin{equation}\label{eq:order_p=2}
    Err = O\Big(
\frac{\|m_\beta-m_\alpha\|^2}{\delta^2}
\Big).
\end{equation}
To verify \eqref{eq:order_p=2}, we organize the numerical study into two groups of experiments: \hyperref[sec:exp1A]{Experiment 1A} -- \hyperref[sec:exp1B]{1B} examine the dependence on $\|m_\beta-m_\alpha\|$, while experiments \hyperref[sec:exp2A]{Experiment 2A} -- \hyperref[sec:exp2B]{2B} examine the dependence on $\delta$.

In the first group of  experiments, we fix $\delta$ and let $m_\alpha$ and $m_\beta$ be neighboring grid points from a uniform grid with size $h$ on the unit square. For each such pair, we construct empirical Gaussian clouds $\widehat\alpha$ and $\widehat\beta$ centered at $m_\alpha$ and $m_\beta$, respectively, and denote their empirical means by $\widehat m_\alpha$ and $\widehat m_\beta$.
The two target quantities are then computed are
\begin{equation}\label{eq:numerical_comparison}
\dwmhat^2(f_\sharp\widehat\alpha, f_\sharp\widehat\beta)
\qquad \text{and} \qquad
\hat d_{M}^2(f(\widehat m_\alpha), f(\widehat m_\beta)),
\end{equation}
here $\dwmhat^2(f_\sharp\widehat\alpha, f_\sharp\widehat\beta)$ and $\hat d_{M}^2(f(\widehat m_\alpha), f(\widehat m_\beta))$ are the finite sample approximation of $\dwm^2(f_\sharp\alpha,f_\sharp\beta)$ and  $d_M^2(f(m_\alpha),f(m_\beta))$, respectively.
We then define the numerical error by
\begin{equation*}
\mathrm{Err}(\delta, h)
=
\left|
\dwmhat^2(f_\sharp\widehat\alpha, f_\sharp\widehat\beta)
-
\hat d_{M}^2(f(\widehat m_\alpha), f(\widehat m_\beta))
\right|.
\end{equation*}
To estimate the observed order, we compare two successive grid sizes as the following, for some constant $C>0$,
\begin{equation*}
\frac{\mathrm{Err}(\delta, h)}{\mathrm{Err}(\delta, h/2)} \approx \frac{C h^q}{C (h/2)^q}
=
2^q,
\end{equation*}
and therefore
\begin{equation*}
\mathrm{order}(h)
=
\log_2\left(
\frac{\mathrm{Err}(\delta, h)}{\mathrm{Err}(\delta, h/2)}
\right) \approx q.
\end{equation*}
From \eqref{eq:order_p=2}, we expect $q=2$.

In the second group of experiments (\hyperref[sec:exp2A]{Experiment 2A}), we fix the grid size $h$ and vary $\delta$. We compare two successive values of $\delta$ and write
\begin{equation*}
\frac{\mathrm{Err}(\delta, h)}{\mathrm{Err}(\delta/2, h)}
\approx
\frac{C\delta^{-r}}{C(\delta/2)^{-r}}
=
2^{-r},
\end{equation*}
so that
\begin{equation*}
\mathrm{order}(\delta)
=
\log_2\left(
\frac{\mathrm{Err}(\delta, h)}{\mathrm{Err}(\delta/2, h)}
\right)
\approx -r.
\end{equation*}
From \eqref{eq:order_p=2}, we expect $r=2$. Furthermore, instead of only examining the error
\begin{equation*}
\mathrm{Err}(\delta,h),
\end{equation*}
we directly study \eqref{eq:numerical_comparison} individually in \hyperref[sec:exp2B]{Experiment 2B}. This allows us to check whether both quantities themselves exhibit the predicted $\delta^{-2}$ scaling. In this way, even when the error in \hyperref[sec:exp2A]{Experiment 2A} is noisy, we can still verify whether the main $\delta$-dependence predicted by the theory is reflected in the individual terms.

\subsection{Experiment Design}
 This section describes the detailed numerical setup used to empirically test \Cref{theorem:main}. As explained above, our numerical study consists of two separate groups of experiments. In both groups of experiments, we work on the unit square and discretize it with a uniform $n\times n$ grid. Let
\begin{equation*}
\mathcal{M}_n
=
\Bigl\{
m_{k}\in[0,1]^2
\Bigr\}_{k=1}^{n^2},
\qquad
h(n)
=
\frac{1}{n-1},
\end{equation*}
where $m_k$ are the grid points. For each $n$, we form a set of neighbor pairs (vertical/horizontal grid neighbors) and randomly sample a fixed number of such pairs for testing. 

Fix a covariance matrix $\Sigma \in \mathbb{R}^{2\times 2}$. Instead of sampling an independent Gaussian cloud at each grid point, we first generate a single prototype empirical cloud at the origin,
\begin{equation*}
X^{(0)}
=
\{x_r^{(0)}\}_{r=1}^{N_{c}},
\qquad
x_r^{(0)} \sim \mathcal{N}(0,\Sigma),
\end{equation*}
and define its empirical mean by
\begin{equation*}
\bar x
=
\frac{1}{N_{c}}
\sum_{r=1}^{N_{c}} x_r^{(0)}.
\end{equation*}
We then recenter the cloud by setting
\begin{equation*}
u_r
=
x_r^{(0)} - \bar x,
\qquad r=1,\dots,N_{c},
\end{equation*}
so that
\begin{equation*}
\frac{1}{N_{c}}
\sum_{r=1}^{N_{c}} u_r
=
0.
\end{equation*}
Thus the prototype empirical cloud is
$
U^{(0)}
=
\{u_r\}_{r=1}^{N_{c}}
$ and the empirical mean is exactly equal to zero.   For each grid point $m_k \in \mathcal{M}_n$, we define the empirical point cloud only by translation
\begin{equation*}
U_k
=
\{m_k + u_r\}_{r=1}^{N_{c}},
\end{equation*}
by construction, the empirical mean of $U_k$ is exactly $m_k$.

Fix a neighborhood radius $\delta>0$.
Around each grid point $m_k$, we sample $M$ centers uniformly in the Euclidean ball $B(m_k,\delta)$,
\begin{equation*}
c_{k,\ell}\sim \mathrm{Unif}\bigl(B(m_k,\delta)\bigr),
\qquad
\ell=1,\dots,M.
\end{equation*}
For each center $c_{k,\ell}$, we define a point cloud again by translating the prototype cloud
\begin{equation*}
V_{k,\ell}
=
\{c_{k,\ell} + u_r\}_{r=1}^{N_{c}},
\qquad
\ell=1,\dots,M.
\end{equation*}
We consider the same nonlinear map $f:\mathbb{R}^2\to\mathbb{R}^2$ as in \cite{singer08}, namely
\begin{equation}\label{eq:nonlinear_transformation_numerical}
f(x,y)=(x+y^3,  y-x^3).
\end{equation} 
We note that this map does not satisfy \Cref{assuption:dispacement_local}, but a less ``nonlinear'' version of it does (see \Cref{exp:assuption_D}). We decided to still use \eqref{eq:nonlinear_transformation_numerical} to demonstrate that the numerical results extend beyond \Cref{theorem:main}.

We apply $f$ to transform the point cloud
\begin{equation*}
\widetilde U_k
=
f(U_k),
\qquad
\widetilde V_{k,\ell}
=
f(V_{k,\ell}),
\end{equation*}
and to transform the mean
$
z_k
=
f(m_k).
$
Denote the empirical covariance of the transformed  centers $\{f(c_{k,\ell})\}_{\ell=1}^M$ by
\begin{equation*}
\widehat \Sigma^{f}_k
=
\mathrm{Cov}\Bigl(
\{f(c_{k,\ell})\}_{\ell=1}^{M}
\Bigr)
\in\mathbb{R}^{2\times 2}.
\end{equation*}
We compute the  symmetric locally centered Mahalanobis distance between transformed means by
\begin{equation}\label{eq:true_dm_def}
\widehat d_M^2\bigl(z_i,z_j\bigr)
=
\frac{1}{2} 
(z_i-z_j)^\top
\Bigl(
\bigl(\widehat \Sigma^{f}_i\bigr)^{\dagger}
+
\bigl(\widehat \Sigma^{f}_j\bigr)^{\dagger}
\Bigr)
(z_i-z_j).
\end{equation}

At each index $k$, we build an empirical covariance operator $F_k^{\dagger}$ from the family of transformed clouds $\{\widetilde U_k,\widetilde V_{k,\ell}\}_{\ell=1}^M$ (\Cref{algo:Pesudo_Inverse_Operater}). Given a neighbor pair $(i,j)$, compute the Mahalanobis distance between transformed clouds 
\begin{equation}\label{eq:approx_dm_def}
\dwmhat^2(\widetilde U_i,\widetilde U_j)
=
\frac{1}{2}
\Bigl[
v_{ij}^\top  F_i^{\dagger} v_{ij}
+
v_{ji}^\top  F_j^{\dagger} v_{ji}
\Bigr],
\end{equation}
where $v_{ij}$ denotes the discrete OT displacement vector (\Cref{Algo:OT_displacement}) from $\widetilde U_i$ to $\widetilde U_j$. Let $\mathcal{P}_n$ denote the set of tested neighbor pairs at grid size $n$, we measure the difference between \eqref{eq:true_dm_def} and \eqref{eq:approx_dm_def} by
\begin{equation}\label{eq:error_def}
\mathrm{Err}(\delta, h(n))
=
\frac{1}{|\mathcal{P}_n|}
\sum_{(i,j)\in\mathcal{P}_n}
\Bigl|
\hat d_M^2(z_i,z_j)
-
\dwmhat^2(\widetilde U_i,\widetilde U_j)
\Bigr|.
\end{equation}
In \hyperref[sec:exp1A]{Experiment 1A} and \hyperref[sec:exp1B]{1B}, where $\delta$ is fixed and $h(n)$ varies, we estimate the observed order by
\begin{equation}\label{eq:rate_def_h}
\log_2\left(\frac{\mathrm{Err}(\delta, h(n))}{\mathrm{Err}(\delta, h(2n-1))}\right)\approx 2.
\end{equation}
In \hyperref[sec:exp2A]{Experiment 2A}, we fix the grid size $n$ and hence fix $h$, while varying $\delta$. We expect
\begin{equation}\label{eq:rate_def_delta}
\log_2\left(
\frac{\mathrm{Err}(\delta, h(n))}{\mathrm{Err}(\delta/2, h(n))}
\right)
\approx -2.
\end{equation}
In \hyperref[sec:exp2B]{Experiment 2B}, instead of considering only the error \eqref{eq:error_def}, we directly examine the two quantities
$
\hat d_M^2(z_i,z_j) \text{ and }
\dwmhat^2(\widetilde U_i,\widetilde U_j)
$
separately, and test whether each of them individually shows the predicted $\delta^{-2}$ scaling. Thus the first group of experiments tests the second-order behavior in $h$, while the second group of experiments tests the $\delta^{-2}$ scaling directly.

\subsection[Numerical Results]{Numerical Results \footnotemark}
\footnotetext{The code for all numerical experiments is available at this
\href{https://github.com/ChuxiangboWang/ChuxiangboWang/tree/main/Wasserstein\%20Mahalanobis\%20Distances\%20for\%20Recovering\%20Latent\%20Geometry}
{GitHub repository}.}
We first test the $h$-dependence of the numerical error while keeping the neighborhood radius $\delta$ fixed.
\subsubsection{Experiment 1A: fixed \texorpdfstring{$\delta$}{delta}, varying \texorpdfstring{$h$}{h}} \label{sec:exp1A}

In this experiment, we work on the unit square $[0,1]^2$ and consider uniform grids with
\begin{equation*}
n\in\{5,9,17,33,65\},
\qquad
h(n)=\frac{1}{n-1}.
\end{equation*}
For each grid, we randomly sample $16$ horizontal or vertical neighbor pairs for testing. The neighborhood radius is fixed at
$
\delta=0.05,
$
and around each grid point we sample $M=100$ centers uniformly in the ball of radius $\delta$. The prototype cloud is fixed throughout the experiment, with covariance
$
\Sigma=(0.01)^2 I
$
chosen so that the standard deviation of the prototype cloud, $0.01$, is smaller than the smallest grid spacing considered, $h\approx 0.0156$ for $n=65$. The cloud size is
$
N_{c}=100.
$
We use the nonlinear map \eqref{eq:nonlinear_transformation_numerical}.
Since $\delta$ is fixed and neighboring grid points satisfy $\|m_\beta-m_\alpha\|=h$, \eqref{eq:order_p=2} predicts that the numerical error should decay quadratically in $h$. Therefore, we plot $\mathrm{Err}(\delta,h(n))$ against $h$ on a log-log scale and compute the observed order using \eqref{eq:rate_def_h}. 

The experiment is repeated $5$ times and  the average errors are reported. As shown in \Cref{fig:Experiment1} below, the mean error decreases steadily as the grid spacing $h$ shrinks and closely follows the slope-$2$ reference line over all tested grid sizes. The table of \Cref{fig:Experiment1} shows that when $h$ is reduced by half, the mean error decreases by approximately a factor of $4$, the observed $h$-orders range from approximately $1.94$ to $2.05$, agree with our predicted $O(h^2)$ behavior.
\begin{figure}[H]
    \centering
    \includegraphics[width=0.97\linewidth]{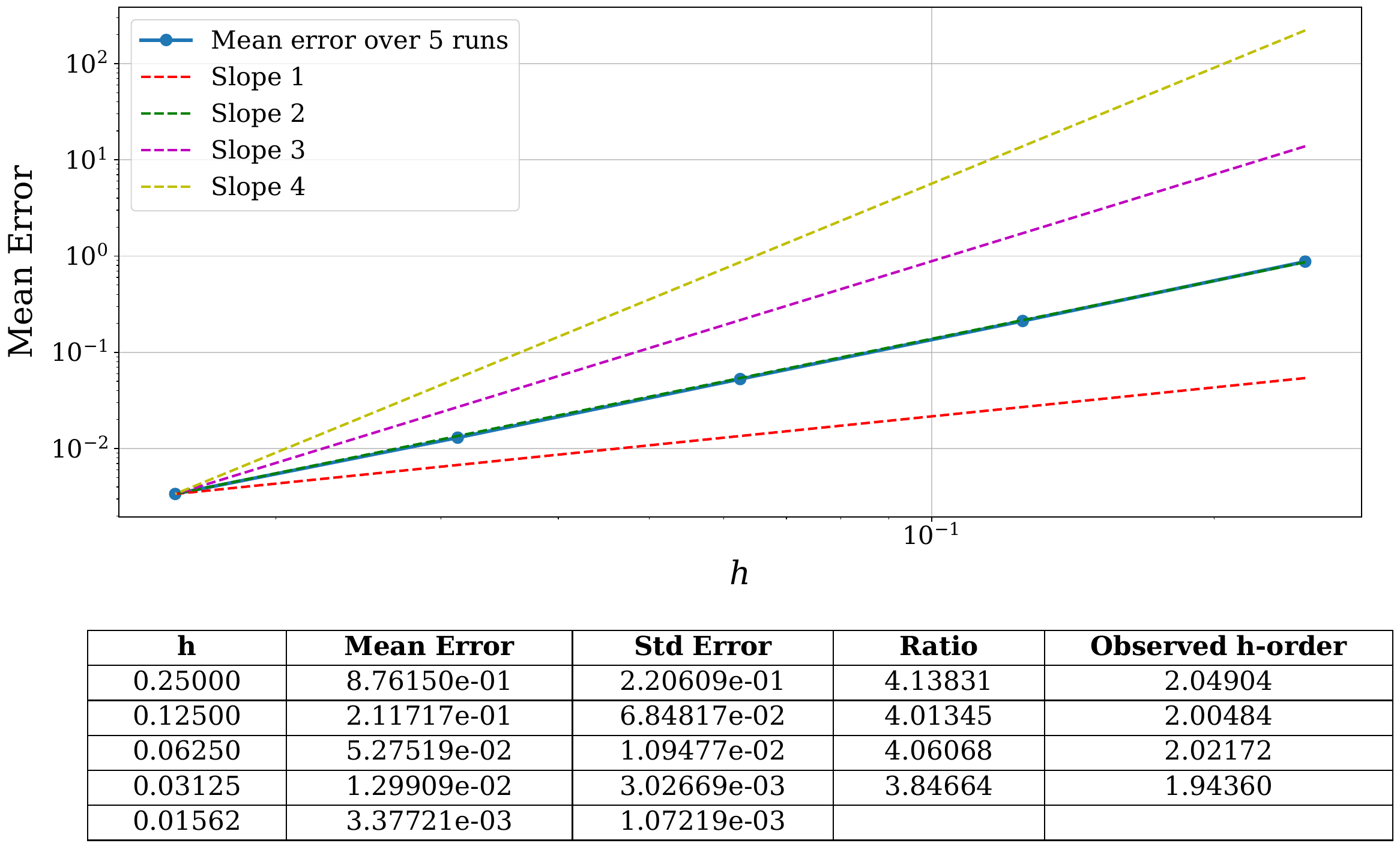}
    \caption{Log-log plot of $\mathrm{Err}(\delta,h(n))$ versus $h$ for \hyperref[sec:exp1A]{Experiment 1A} with fixed $\delta=0.05$. The dashed lines indicate reference slopes $1$, $2$, $3$, and $4$. The computed error is closest to the slope $2$, consistent with the predicted $O(h^2)$ behavior.}
    \label{fig:Experiment1}
\end{figure}
To further check that this observed behavior is robust under different choices of empirical sampling parameters, we study the sensitivity of the computation to $M$. In particular, we examine how the error changes when varying the number of nearby centers $M$, while keeping the same experimental setting as above.

\subsubsection{Experiment 1B: sensitivity to \texorpdfstring{$M$}{M}} \label{sec:exp1B}

Using the same setting as in \hyperref[sec:exp1A]{Experiment 1A}, we vary the empirical sampling parameters $M$ while keeping all other parameters fixed. Since $M$ only affects the finite-sample approximation, and not the theoretical $h$ dependence in \eqref{eq:order_p=2}, we expect the observed order in $h$ to remain approximately unchanged.

In this experiment, we compare two choices of the number of nearby centers,
$M=400$
and
$M=1600.$
For each choice of $M$, we compute $\mathrm{Err}(\delta,h(n))$ from \eqref{eq:error_def} and estimate the observed order by
\begin{equation*}
\log_2\left(
\frac{\mathrm{Err}(\delta,h(n))}{\mathrm{Err}(\delta,h(2n-1))}
\right).
\end{equation*}
We expect that increasing $M$ should improve the stability of the empirical covariance approximation and may reduce the absolute error, while the convergence trend in $h$ should remain close to second order. Again, we repeat the experiment $5$ times and report the averaged error at each grid size.
\begin{figure}[H]
    \centering

    \begin{minipage}{0.49\textwidth}
        \centering
        \includegraphics[width=\linewidth, trim={0 6.5cm 0 0}, clip]{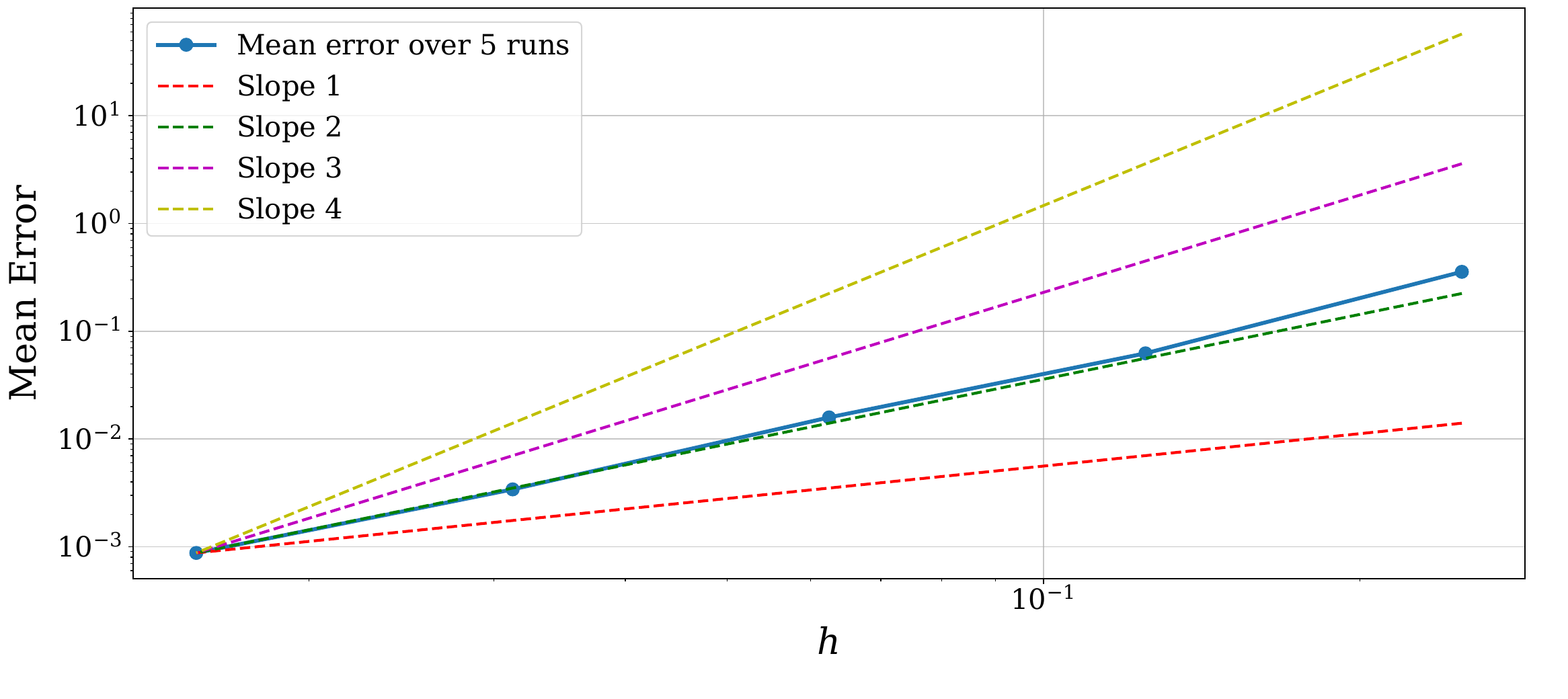}
    \end{minipage}
    \hfill
    \begin{minipage}{0.49\textwidth}
        \centering
        \includegraphics[width=\linewidth, trim={0 6.5cm 0 0}, clip]{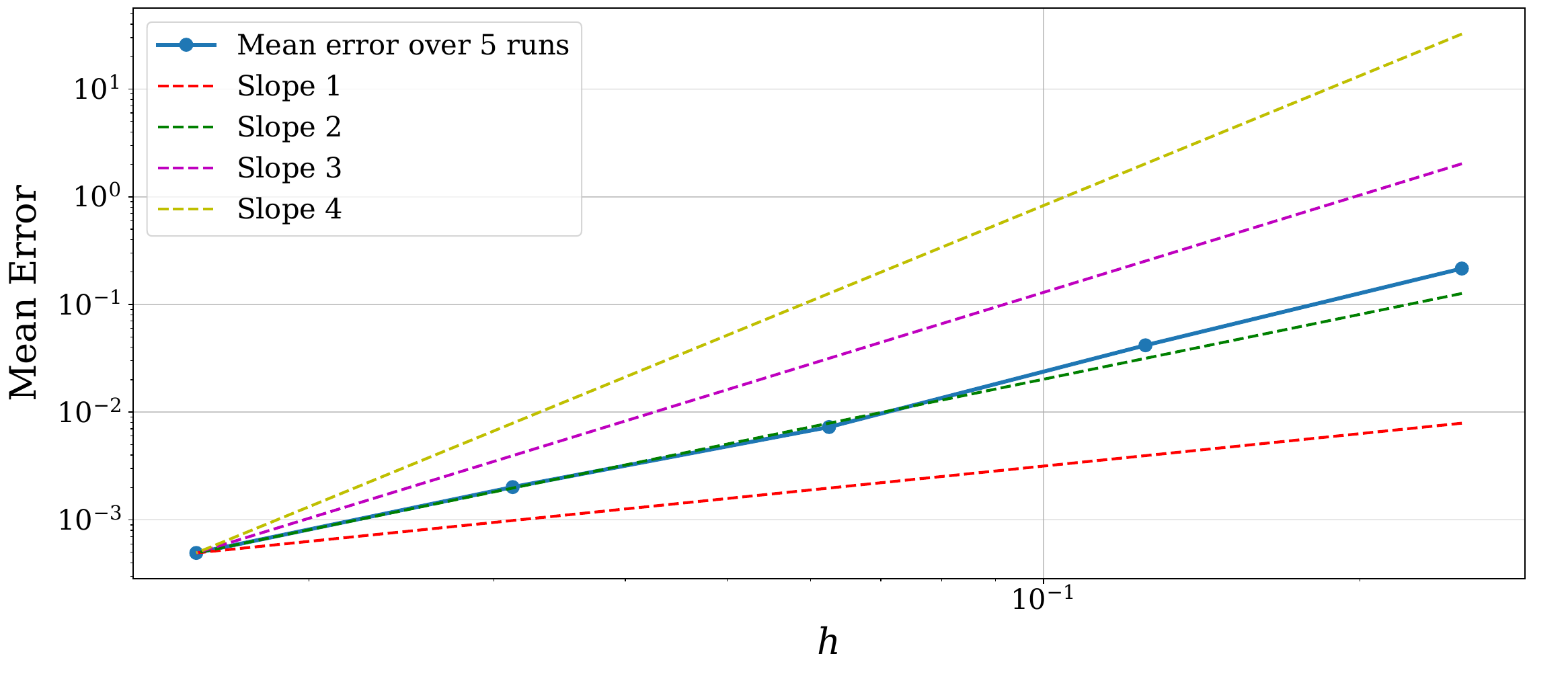}
    \end{minipage}

    \vspace{0.2cm}

    {\scriptsize
    \setlength{\tabcolsep}{3pt}
    \renewcommand{\arraystretch}{1.08}
    \resizebox{\textwidth}{!}{
    \begin{tabular}{c|cccc|cccc}
    \hline
    & \multicolumn{4}{c|}{$M=400$}
    & \multicolumn{4}{c}{$M=1600$} \\
    \cline{2-9}
    $h$
    & Mean Error & Std Error & Ratio & Observed $h$-order
    & Mean Error & Std Error & Ratio & Observed $h$-order \\
    \hline
    0.25000 & $3.56285\mathrm{e}{-01}$ & $1.17580\mathrm{e}{-01}$ & 5.70745 & 2.51285
            & $2.15494\mathrm{e}{-01}$ & $5.35671\mathrm{e}{-02}$ & 5.15754 & 2.36668 \\
    0.12500 & $6.24246\mathrm{e}{-02}$ & $1.48628\mathrm{e}{-02}$ & 3.93466 & 1.97624
            & $4.17823\mathrm{e}{-02}$ & $5.92121\mathrm{e}{-03}$ & 5.73757 & 2.52044 \\
    0.06250 & $1.58653\mathrm{e}{-02}$ & $2.68526\mathrm{e}{-03}$ & 4.65530 & 2.21887
            & $7.28223\mathrm{e}{-03}$ & $2.32371\mathrm{e}{-03}$ & 3.60828 & 1.85131 \\
    0.03125 & $3.40801\mathrm{e}{-03}$ & $9.78880\mathrm{e}{-04}$ & 3.88961 & 1.95962
            & $2.01820\mathrm{e}{-03}$ & $5.53029\mathrm{e}{-04}$ & 4.08515 & 2.03039 \\
    0.01562 & $8.76184\mathrm{e}{-04}$ & $2.51434\mathrm{e}{-04}$ & -- & --
            & $4.94033\mathrm{e}{-04}$ & $1.50506\mathrm{e}{-04}$ & -- & -- \\
    \hline
    \end{tabular}
    }
    }

    \caption{Log-log plots of $\mathrm{Err}(\delta,h(n))$ versus $h$ for \hyperref[sec:exp1B]{Experiment 1B}. Left: $M=400$. Right: $M=1600$. The dashed lines indicate reference slopes $1$, $2$, $3$, and $4$. The numerical values are shown in the table below the plots. In both cases, the computed error remains close to the slope $2$, showing approximate second-order behavior in $h$.}
    \label{fig:Experiment1B-1}
\end{figure}
As shown in \Cref{fig:Experiment1B-1}, both error curves decrease steadily as $h$ decreases and remain close to the slope-$2$ reference line over our grid scales. The error curve for $M=1600$ lies consistently below that for $M=400$, showing a reduction in the absolute error across all grid sizes. Table of \Cref{fig:Experiment1B-1} also shows that increasing $M$ reduces the absolute error at each grid size. For example, at the coarsest grid size $h=0.25$, the error decreases from about $3.5\times 10^{-1}$ when $M=400$ to about $2.1 \times 10^{-1}$ when $M=1600$. Similar reductions exist at the other grid sizes as well, consistent with our expectation.

\subsubsection{Experiment 2A: fixed \texorpdfstring{$h$}{h}, varying \texorpdfstring{$\delta$}{delta}}\label{sec:exp2A}
We now fix the grid size and study how the numerical error depends on the neighborhood radius $\delta$. In this experiment, we again use a uniform grid on $[0,1]^2$ but with
$
n=33,
h=\frac{1}{n-1}=\frac{1}{32},
$
and choose $16$ horizontal or vertical neighbor pairs. The neighborhood radius varies over
\begin{equation*}
\delta\in\{0.1, 0.05, 0.025, 0.0125\}.
\end{equation*}
For each value of $\delta$, we sample
$
M
=100
$
centers uniformly in the ball with radius $\delta$ around each grid point.
The prototype cloud is fixed throughout the experiment, with covariance
$
\Sigma=(0.01)^2 I
$
and cloud size
$
N_{c}=100.
$
The nonlinear map is taken to be the same as \eqref{eq:nonlinear_transformation_numerical}. We repeat the experiment
5 times and plot \eqref{eq:rate_def_delta}, see \Cref{fig:Experiment2}.

\begin{figure}[H]
    \centering
    \includegraphics[width=0.97\linewidth]{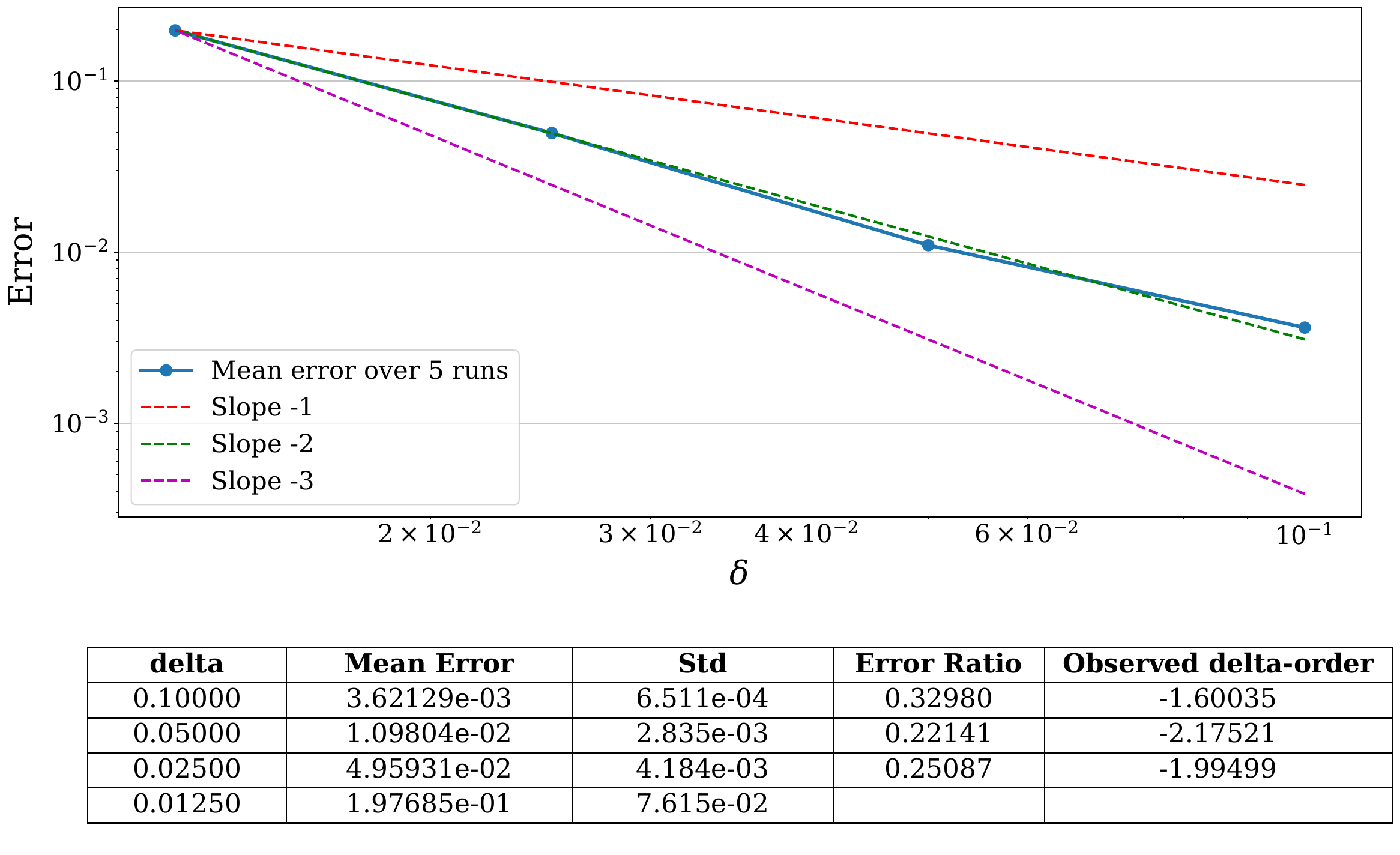}
    \caption{Log-log plot of the mean error versus $\delta$ for \hyperref[sec:exp2A]{Experiment 2A} with fixed $h=1/32$. The dashed lines indicate reference slopes $-1$, $-2$, and $-3$. The computed error is close to the slope $-2$, consistent with the predicted $O(\delta^{-2})$ behavior.}
    \label{fig:Experiment2}
\end{figure}
Table of \Cref{fig:Experiment2} shows that the numerical error increases as $\delta$ decreases, which is consistent with the prediction that the error should scale like a negative power of $\delta$. The observed $\delta$ orders are approximately in the range of the predicted value $-2$, but they are not always fully stable. Therefore, in the \hyperref[sec:exp2B]{Experiment 2B} below, rather than looking only at the error
$
\mathrm{Err}(\delta,h),
$
we directly examine the two quantities being compared,
\begin{equation*}
\hat d_M^2(z_i,z_j)
\qquad\text{and}\qquad
\dwmhat^2(\widetilde U_i,\widetilde U_j),
\end{equation*}
to check whether they themselves exhibit the expected $\delta^{-2}$ scaling.

\subsubsection{Experiment 2B: \texorpdfstring{$\delta$}{delta} order of the two compared distances} \label{sec:exp2B}

In this experiment, we further examine the $\delta$-dependence as mentioned above. With the exact same setting as in \hyperref[sec:exp2A]{Experiment 2A} , we fix the grid size and vary $\delta$. For each tested neighbor pair $(i,j)$, we compute both
$
\hat d_M^2(z_i,z_j)$
and $
\dwmhat^2(\widetilde U_i,\widetilde U_j),
$
and average these two quantities separately over the tested pairs. We repeat the experiment 5 times and report the mean values. We plot the averaged Euclidean Mahalanobis distance and the averaged Wasserstein Mahalanobis distance versus $\delta$.

\begin{figure}[H]
    \centering
    \begin{minipage}{0.48\textwidth}
        \centering
        \includegraphics[width=\linewidth, trim={0 6.5cm 0 0}, clip]{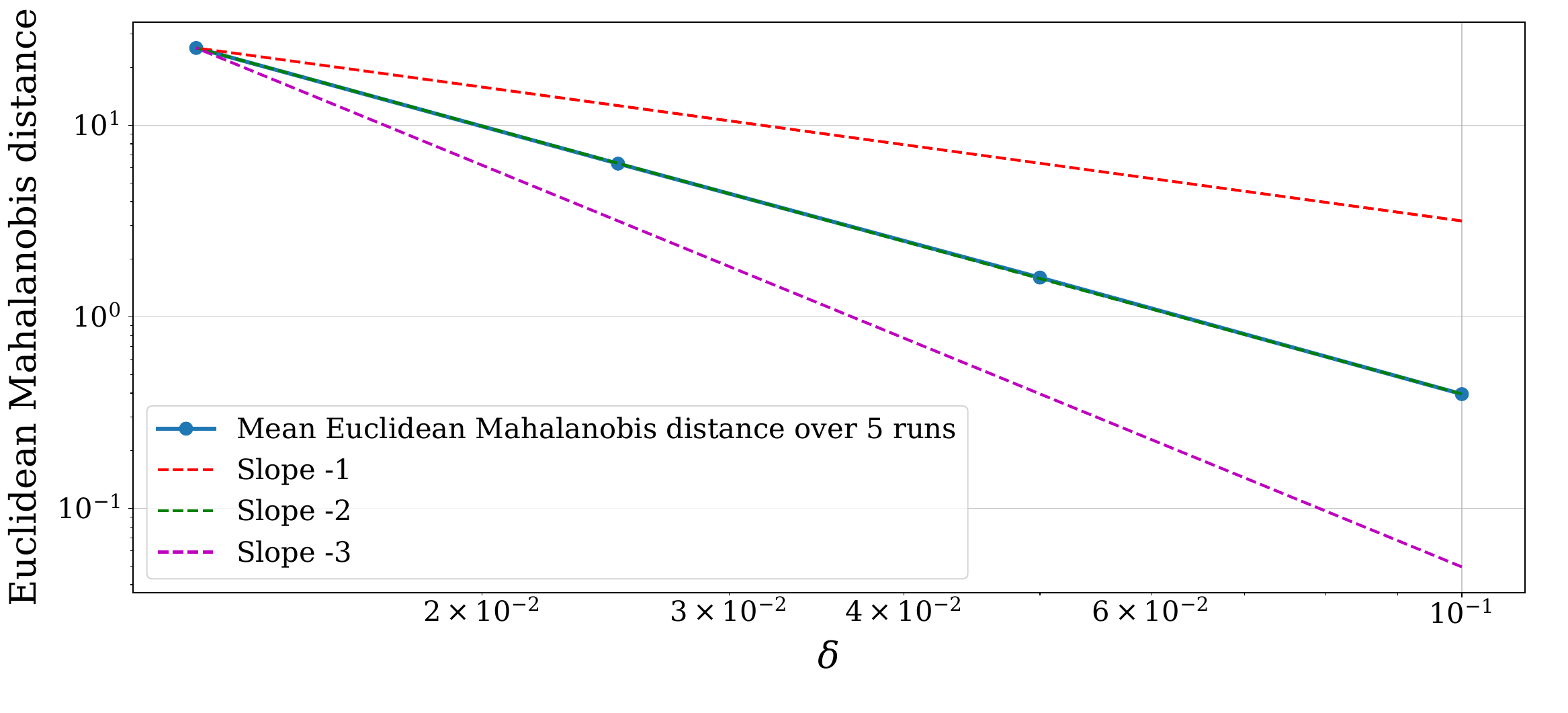}
    \end{minipage}
    \hfill
    \begin{minipage}{0.48\textwidth}
        \centering
        \includegraphics[width=\linewidth, trim={0 6.5cm 0 0}, clip]{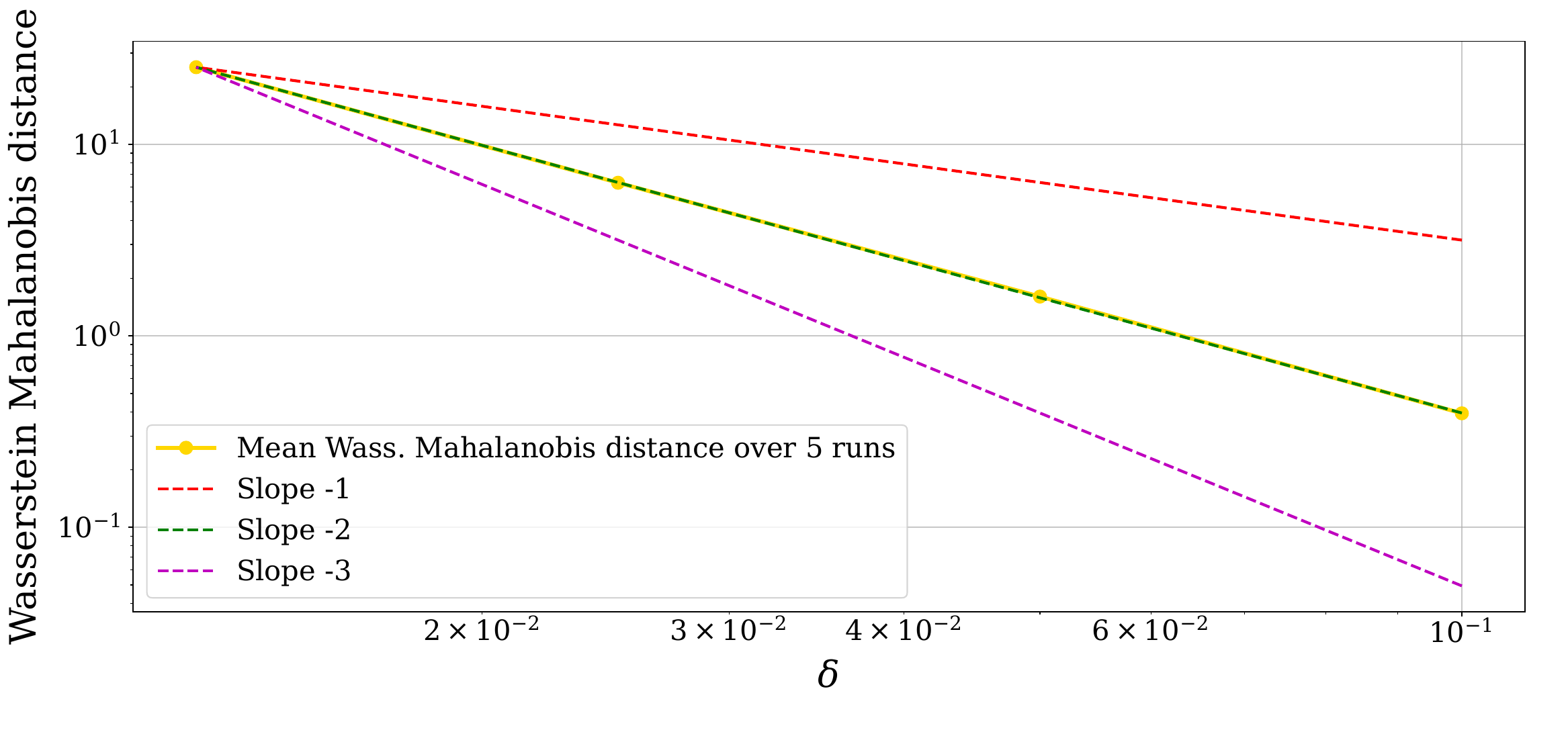}
    \end{minipage}

        \vspace{0.2cm}

    {\scriptsize
    \setlength{\tabcolsep}{3pt}
    \renewcommand{\arraystretch}{1.08}
    \resizebox{\textwidth}{!}{
\begin{tabular}{c|cc|cc|cc|cc}
\hline
$\delta$
& \multicolumn{4}{c|}{Euclidean Mahalanobis}
& \multicolumn{4}{c}{Wasserstein Mahalanobis} \\
\hline
& Mean & Std & Ratio & $\delta$-order
& Mean & Std & Ratio & $\delta$-order \\
\hline
0.10000 
& $3.95050\mathrm{e}{-01}$ & $5.080\mathrm{e}{-03}$ & 0.24640 & -2.02095
& $3.93823\mathrm{e}{-01}$ & $4.319\mathrm{e}{-03}$ & 0.24560 & -2.02560 \\

0.05000
& $1.60331\mathrm{e}{+00}$ & $4.113\mathrm{e}{-02}$ & 0.25411 & -1.97648
& $1.60350\mathrm{e}{+00}$ & $3.926\mathrm{e}{-02}$ & 0.25438 & -1.97492 \\

0.02500
& $6.30952\mathrm{e}{+00}$ & $8.748\mathrm{e}{-02}$ & 0.24896 & -2.00604
& $6.30345\mathrm{e}{+00}$ & $8.945\mathrm{e}{-02}$ & 0.24915 & -2.00489 \\

0.01250
& $2.53440\mathrm{e}{+01}$ & $6.522\mathrm{e}{-01}$ &  & 
& $2.52993\mathrm{e}{+01}$ & $6.355\mathrm{e}{-01}$ &  &  \\
\hline
\end{tabular}
}
}
    \caption{Log-log plots of the averaged Euclidean Mahalanobis distance (left) and the averaged Wasserstein Mahalanobis distance (right) versus $\delta$ for \hyperref[sec:exp2B]{Experiment 2B}. The dashed lines indicate reference slopes $-1$, $-2$, and $-3$. The computed quantities are very close to the slope $-2$.}
    \label{fig:experiment3-1}
\end{figure}
Tables of \Cref{fig:experiment3-1} show that for both averaged quantities, the observed $\delta$ orders are very close to $-2$ and exhibit very similar behavior. Combined with the less stable behavior observed in Experiment 2, this suggests that the main $\delta^{-2}$ scaling is indeed present in both compared terms, but their difference can be more sensitive to finite sampling and possible cancellation effects.

\subsection{Additional Experiment for \texorpdfstring{\Cref{remark:dwm_and_euclidean}}{Remark DWM and Euclidean} and \texorpdfstring{\Cref{cor:point_mass_limit}}{corollary point mass}}\label{subsec:Additional Experiment}
In this subsection, we provide additional numerical experiments for
\Cref{remark:dwm_and_euclidean} and \Cref{cor:point_mass_limit}. Recall that in \Cref{remark:dwm_and_euclidean}, when $p=2$, 
the squared
Wasserstein Mahalanobis distance should approximate the squared Euclidean distance
between the original means:
\begin{equation}\label{eq:remark_numerical_expectation}
\frac{\delta^2}{n+2}
\dwm^2(f_\sharp\alpha,f_\sharp\beta)
=
\|m_\alpha-m_\beta\|^2
+
O(\|m_\alpha-m_\beta\|^2).
\end{equation}
In the limiting case where $\Sigma\to 0_{n\times n}$, the Gaussian measures
collapse to point masses, and \Cref{cor:point_mass_limit} gives
\begin{equation}\label{eq:limit_Gaussian_pointmass}
\frac{\delta^2}{n+2}
\lim_{\quad\Sigma\to 0_{n\times n}}
\dwm^2(f_\sharp\alpha,f_\sharp\beta)
=
\|m_\alpha-m_\beta\|^2
+
O\left(
\delta\|m_\alpha-m_\beta\|^2
\right)+
O\left(
\|m_\alpha-m_\beta\|^4
\right).
\end{equation}
Therefore, we design two additional experiments. The first experiment directly
tests the second-order approximation in \eqref{eq:remark_numerical_expectation}.
The second experiment decreases $\Sigma$ toward
$0_{n\times n}$ and examines \eqref{eq:limit_Gaussian_pointmass}.

\subsubsection{Experiment 3A: comparison with the original Euclidean distance}\label{experiment:3A}

In this experiment, we use the same numerical setting as in \hyperref[sec:exp1A]{Experiment 1A}. The only difference from \hyperref[sec:exp1A]{Experiment 1A} is the quantity being compared.  Instead of
comparing
$
\dwmhat^2(\widetilde U_i,\widetilde U_j)$ with 
$\hat d_M^2(z_i,z_j),$
we compare the scaled Wasserstein Mahalanobis distance with the squared
Euclidean distance between the original grid centers. For each
tested pair $(i,j)$, we compute
\begin{equation*}
\|m_i-m_j\|^2
\qquad\text{and}\qquad
\frac{\delta^2}{4}
\dwmhat^2(\widetilde U_i,\widetilde U_j).
\end{equation*}
We define the numerical error by
\begin{equation}\label{eq:err_euc}
\mathrm{Err}_{\mathrm{Euc}}(\delta,h(n))
=
\frac{1}{|\mathcal P_n|}
\sum_{(i,j)\in\mathcal P_n}
\left|
\|m_i-m_j\|^2
-
\frac{\delta^2}{4}
\dwmhat^2(\widetilde U_i,\widetilde U_j)
\right|,
\end{equation}
and expect
\begin{equation*}
\log_2\left(
\frac{
\mathrm{Err}_{\mathrm{Euc}}(\delta,h(n))
}{
\mathrm{Err}_{\mathrm{Euc}}(\delta,h(2n-1))
}
\right) \approx2.
\end{equation*}

The results in \Cref{fig:Experiment3} show that the Euclidean error decreases consistently as h becomes smaller. The error curve follows the slope $2$ reference line closely over most grid scales, and the observed h-orders are generally around 2, aside from the coarsest grid. This is consistent with the predicted behavior $O(h^2)$.

\begin{figure}[H]
    \centering
    \includegraphics[width=0.97\linewidth]{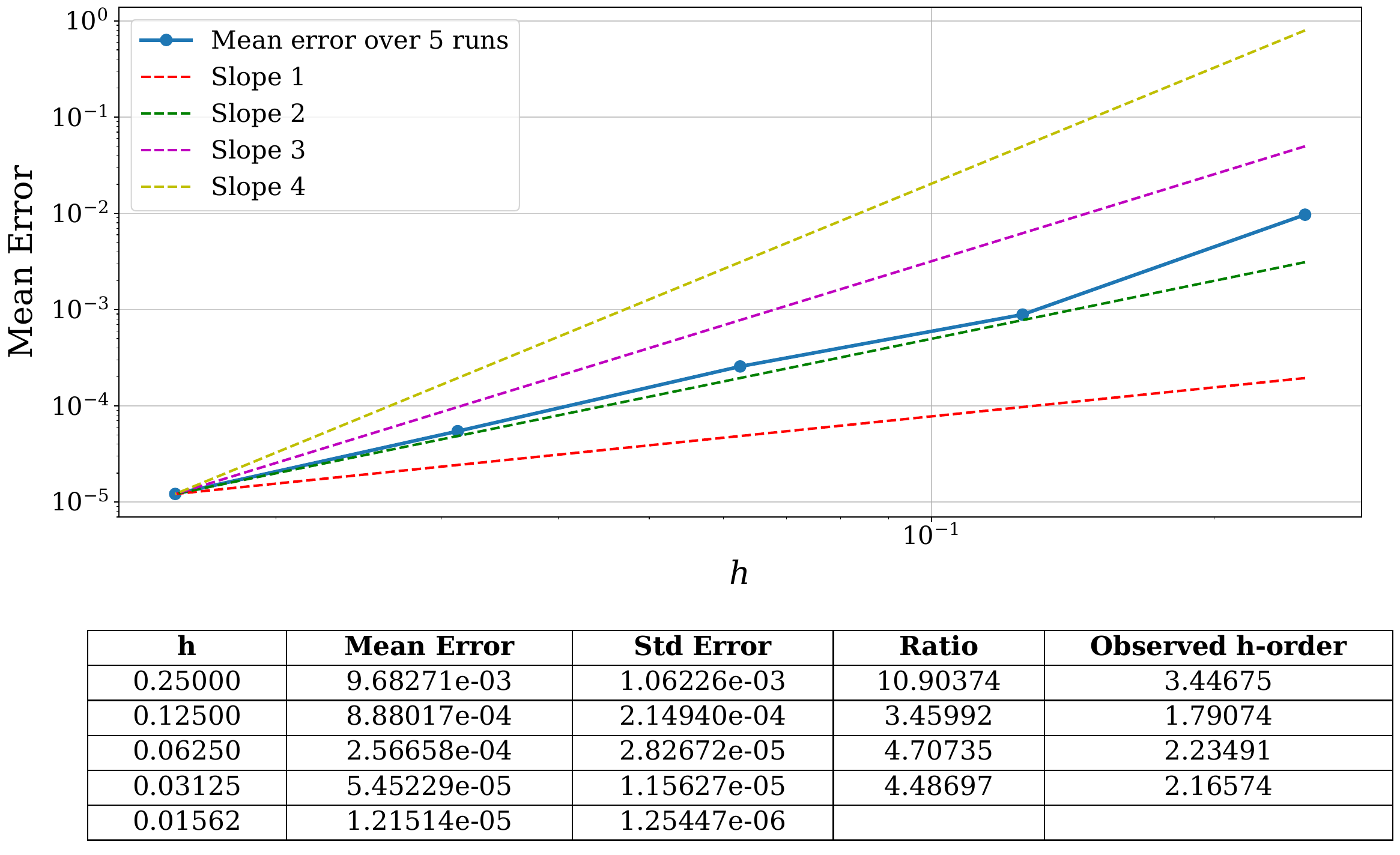}
    \caption{Log-log plot of
    $\mathrm{Err}_{\mathrm{Euc}}(\delta,h(n))$ versus $h$ for \hyperref[experiment:3A]{Experiment 3A}. The dashed lines indicate reference slopes $1$, $2$, $3$, and $4$. The computed error is closest to the slope $2$, consistent with the predicted $O(h^2)$ behavior. }
    \label{fig:Experiment3}
\end{figure}
\subsubsection{Experiment 3B: point-mass limit}\label{experiment:3B}
We next numerically examine the point-mass limit in
\Cref{cor:point_mass_limit}. We use the same experiment setting as in \hyperref[experiment:3A]{Experiment 3A} but set $\Sigma = 10^{-150} I$ and $N_c=1$. To reduce the finite sampling error we increase $M$ to $320000$. For neighboring grid points, $\|m_\alpha-m_\beta\|=h(n)$, so
\eqref{eq:err_euc} predicts
\begin{equation*}
\mathrm{Err}_{\mathrm{Euc}}(\delta,h(n))
=
O\!\left(\delta h^2(n)\right)
+
O\!\left(h^4(n)\right).
\end{equation*}

We consider two choices of the neighborhood radius. First, we set $\delta=10^{-8}$, this choice satisfies $\delta<h^2(n)$ for all tested grid sizes, the fourth-order term dominates, and we expect
\begin{equation*}
\log_2\left(
\frac{
\mathrm{Err}_{\mathrm{Euc}}(10^{-8},h(n))
}{
\mathrm{Err}_{\mathrm{Euc}}(10^{-8},h(2n-1))
}
\right)
\approx 4.
\end{equation*}
We then set $\delta=0.3$. Since $\delta>h(n)$ for all tested grid sizes, the second-order term dominates, and we expect
\begin{equation*}
\log_2\left(
\frac{
\mathrm{Err}_{\mathrm{Euc}}(0.3,h(n))
}{
\mathrm{Err}_{\mathrm{Euc}}(0.3,h(2n-1))
}
\right)
\approx 2.
\end{equation*}

As in \Cref{fig:Experiment3B_1,fig:Experiment3B_2}, the two choices of $\delta$ show the expected behavior. For $\delta=10^{-8}$ (\Cref{fig:Experiment3B_1}),
the observed $h$-orders show nearly fourth-order convergence, with some deviation at the
finer grid. For $\delta=0.3$ (\Cref{fig:Experiment3B_2}), the observed orders remain close to second order.
These results are consistent with \Cref{cor:point_mass_limit}.
\begin{figure}[H]
    \centering
    \includegraphics[width=0.97\linewidth]{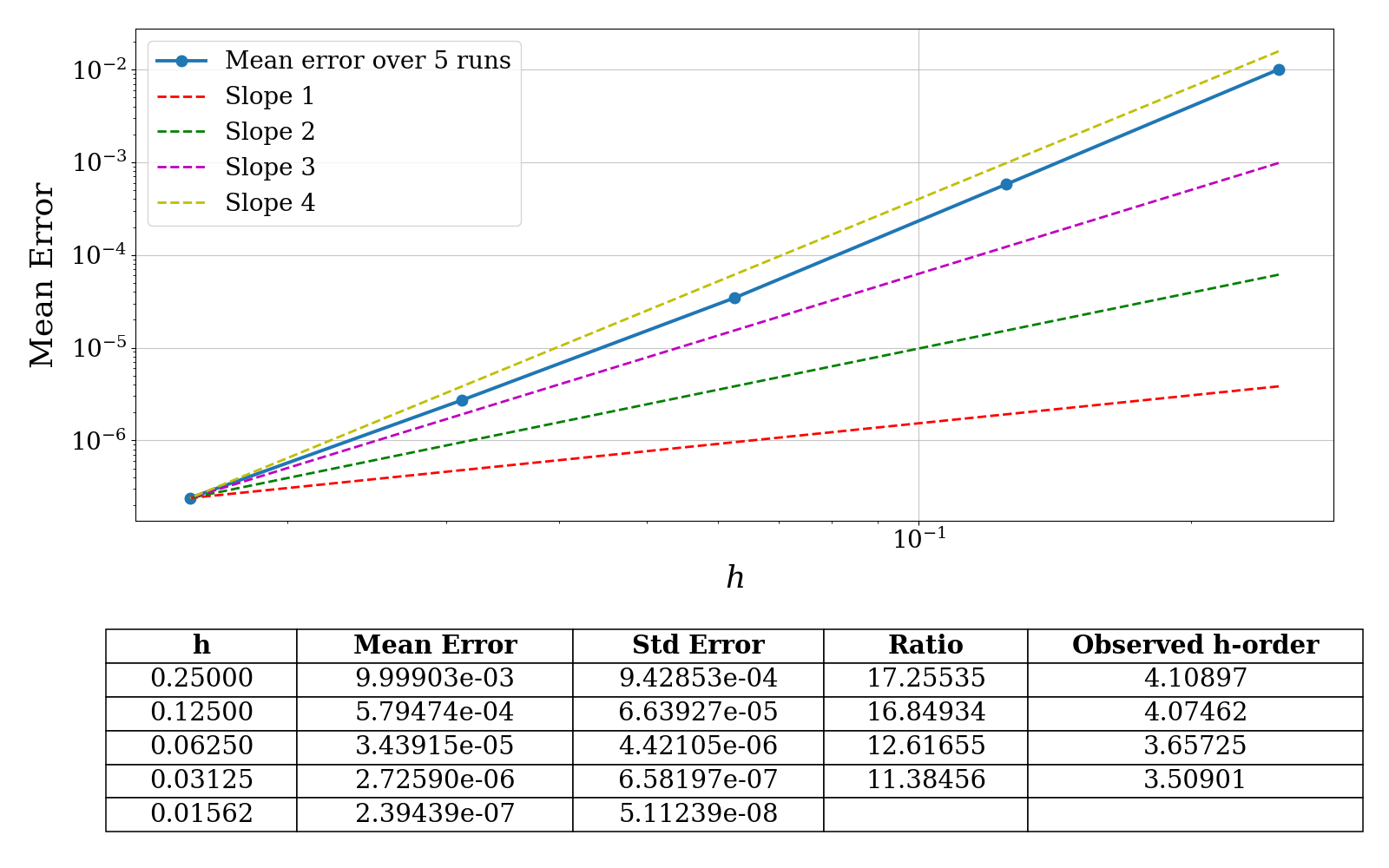}
    \caption{Log-log plot of $\mathrm{Err}_{\mathrm{Euc}}(\delta,h(n))$ versus $h$ for \hyperref[experiment:3B]{Experiment 3B} with $\delta=10^{-8}$. Since $\delta<h^2(n)$ for all tested grid sizes, the $O\left(\|m_\alpha-m_\beta\|^4\right)$ term dominates, and the computed error exhibits approximately fourth-order behavior. The dashed lines indicate reference
 slopes $1$, $2$, $3$, and $4$. }
    \label{fig:Experiment3B_1}
\end{figure}

\begin{figure}[H]
    \centering
    \includegraphics[width=0.97\linewidth]{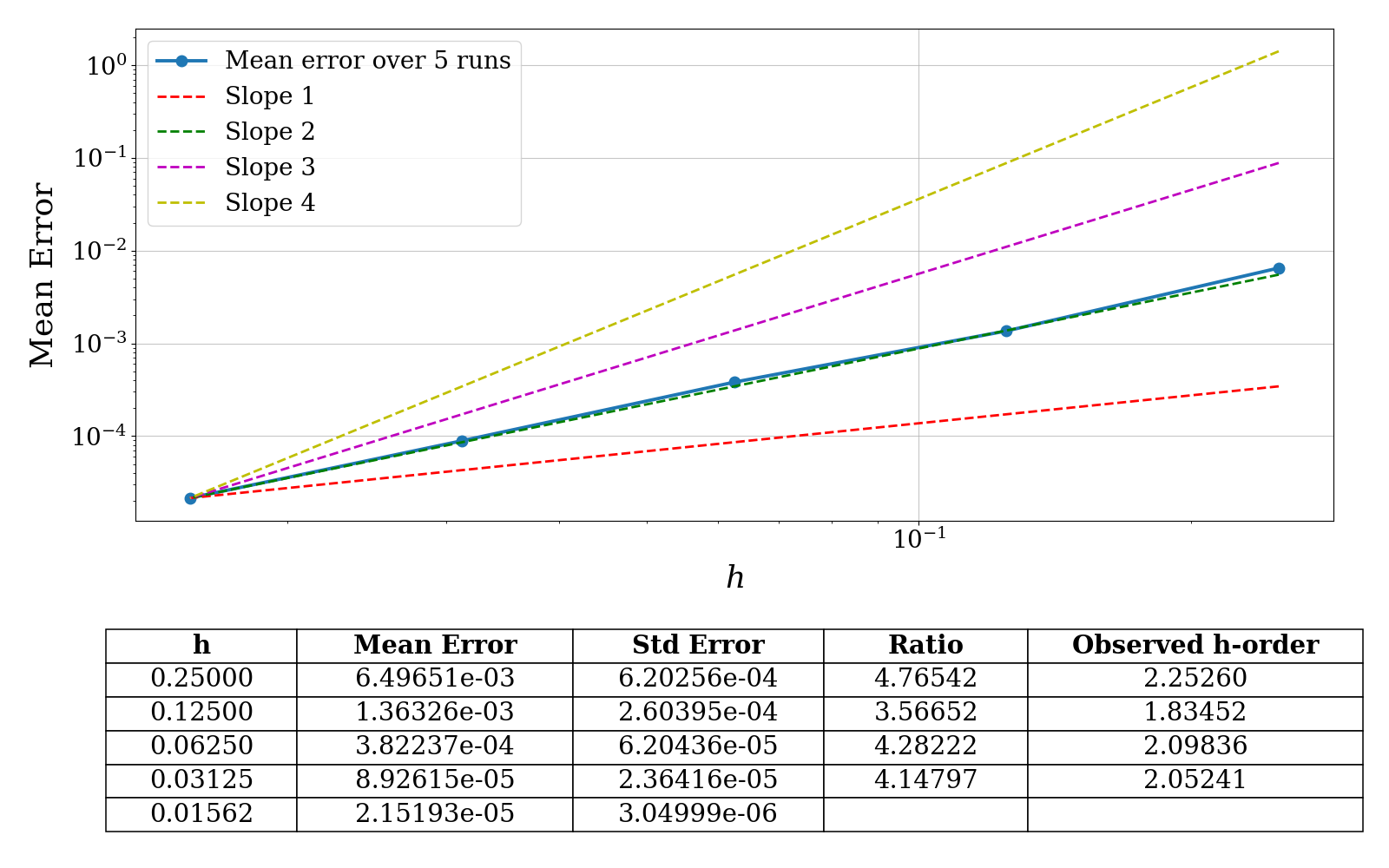}
    \caption{Log-log plot of $\mathrm{Err}_{\mathrm{Euc}}(\delta,h(n))$ versus $h$ for \hyperref[experiment:3B]{Experiment 3B} with $\delta=0.3$. Since $\delta>h(n)$ for all tested grid sizes, the   $O\left(\delta\|m_\alpha-m_\beta\|^2\right)$ term dominates, and the computed error exhibits approximately second-order behavior. The dashed lines indicate reference
 slopes $1$, $2$, $3$, and $4$. }
    \label{fig:Experiment3B_2}
\end{figure}

\section*{Acknowledgments}
The authors were partially supported by NSF award DMS-2410140.

\bibliographystyle{abbrv}
\bibliography{OT-references}

@article{xing2002distance,
  title={Distance metric learning with application to clustering with side-information},
author = {Xing, E. and Jordan, M. and Russell, S. J. and Ng, A.},
  journal={Advances in neural information processing systems},
  volume={15},
  year={2002}
}

@article{rousseeuw1990unmasking,
  title={Unmasking multivariate outliers and leverage points},
 author = {Rousseeuw, P. J. and Van Zomeren, B. C.},
  journal={Journal of the American Statistical association},
  volume={85},
  number={411},
  pages={633--639},
  year={1990},
  publisher={Taylor \& Francis}
}

@article{leys2018detecting,
  title={Detecting multivariate outliers: Use a robust variant of the Mahalanobis distance},
author = {Leys, C. and Klein, O. and Dominicy, Y. and Ley, C.},
  journal={Journal of experimental social psychology},
  volume={74},
  pages={150--156},
  year={2018},
  publisher={Elsevier}
}

@article{weinberger2009distance,
  title={Distance metric learning for large margin nearest neighbor classification.},
 author = {Weinberger, K. Q. and Saul, L. K.},
  journal={Journal of machine learning research},
  volume={10},
  number={2},
  year={2009}
}

@article{talmon2013empirical,
  title={Empirical intrinsic geometry for nonlinear modeling and time series filtering},
  author = {Talmon, R. and Coifman, R. R.},
  journal={Proceedings of the National Academy of Sciences},
  volume={110},
  number={31},
  pages={12535--12540},
  year={2013},
  publisher={National Academy of Sciences}
}

@article{CarlenGangbo2003,
  author = {Carlen, E. A. and Gangbo, W.},
  title     = {Constrained Steepest Descent in the 2-{W}asserstein Metric},
  journal   = {Annals of Mathematics},
  volume    = {157},
  number    = {3},
  pages     = {807--846},
  year      = {2003},
  doi       = {10.4007/annals.2003.157.807}
}

@inproceedings{zhang2024mahalanobis,
  title={Mahalanobis distance-based multi-view optimal transport for multi-view crowd localization},
  author = {Zhang, Q. and Zhang, K. and Chan, A. B. and Huang, H.},
  booktitle={European Conference on Computer Vision},
  pages={19--36},
  year={2024},
  organization={Springer}
}

@article{Hamm2025ManifoldLearningW2,
author = {Hamm, Keaton and Moosm\"{u}ller, Caroline and Schmitzer, Bernhard and Thorpe, Matthew},
title = {Manifold Learning in {W}asserstein Space},
journal = {SIAM Journal on Mathematical Analysis},
volume = {57},
number = {3},
pages = {2983-3029},
year = {2025},
doi = {10.1137/23M1617540}
}

@incollection{LMW2025linear,
  title={Linear Independent Component Analysis in {W}asserstein Space},
  author={Li, Shiying and Moosm{\"u}ller, Caroline and Wang, Chuxiangbo},
  booktitle={Advances in Data Science: Women in Data Science and Mathematics (WiSDM) 2023},
  pages={131--150},
  year={2025},
  publisher={Springer}
}

@inproceedings{kerdoncuff2020metric,
  title={Metric learning in optimal transport for domain adaptation},
  author={Kerdoncuff, Tanguy and Emonet, R{\'e}mi and Sebban, Marc},
  booktitle={International joint conference on artificial intelligence},
  pages={2162--2168},
  year={2020},
  organization={IJCAI}
}

@article{bonet2025sliced,
  title={Sliced-{W}asserstein distances and flows on {C}artan-{H}adamard manifolds},
  author={Bonet, Cl{\'e}ment and Drumetz, Lucas and Courty, Nicolas},
  journal={Journal of Machine Learning Research},
  volume={26},
  number={32},
  pages={1--76},
  year={2025}
}

@article{lin2024small,
  title={Small sample behavior of {W}asserstein projections, connections to empirical likelihood, and other applications},
  author={Lin, Sirui and Blanchet, Jose and Glynn, Peter and Nguyen, Viet Anh},
  journal={arXiv preprint arXiv:2408.11753},
  year={2024}
}

@Article{singer08,
  Title                    = {Non-linear independent component analysis with diffusion maps},
  Author                   = {Singer, A. and Coifman, R.~R.},
  Journal                  = {Applied and Computational Harmonic Analysis},
  Year                     = {2008},
  Number                   = {2},
  Pages                    = {226--239},
  Volume                   = {25},
  Doi                      = {10.1016/j.acha.2007.11.001}
}

@article{rubner2000earth,
  title={The earth mover's distance as a metric for image retrieval},
  author = {Rubner, Y. and Tomasi, C. and Guibas, L. J.},
  journal={Int J Comput Vis},
  volume={40},
  number={2},
  pages={99--121},
  year={2000},
  publisher={Springer}
}

@article{mathews18,
	author = {J. Mathews and M. Pouryahya and C. Moosm\"uller and I.~G. Kevrekidis and J. Deasy and A. Tannenbaum},
	title = {Molecular phenotyping using networks, diffusion, and topology: soft-tissue sarcoma},
	journal = {Scientific Reports}, 
    doi = {10.1038/s41598-019-50300-2},
    year = {2019},
}

@article{zhang2010understanding,
  title={Understanding bag-of-words model: a statistical framework},
 author = {Zhang, Y. and Jin, R. and Zhou, Z.-H.},
  journal={International Journal of Machine Learning and Cybernetics},
  volume={1},
  number={1-4},
  pages={43--52},
  year={2010},
  publisher={Springer}
}

@book{Villani2009,
	Author = {Villani, C.},
	Publisher = {Springer-Verlag Berlin Heidelberg},
	Series = {Grundlehren der mathematischen Wissenschaften},
	Title = {Optimal Transport: Old and New},
	Volume = {338},
	Year = {2009}}

@book{villani2003topics,
  title={Topics in Optimal Transportation},
  author={Villani, C.},
  isbn={9780821833124},
  lccn={2003040350},
  series={Graduate studies in mathematics},
  url={https://books.google.com/books?id=idyFAwAAQBAJ},
  year={2003},
  publisher={American Mathematical Society}
}

@article{peyre19,

    year = {2019},

    volume = {11},

    journal = {Foundations and Trends in Machine Learning},

    title = {Computational Optimal Transport},

    doi = {10.1561/2200000073},

    number = {5-6},

    pages = {355-607},

    author = {G. Peyr{\'e} and M. Cuturi}

}

@article{Roberto2013LocalMahalanobis,
title = {Locally centred Mahalanobis distance: A new distance measure with salient features towards outlier detection},
journal = {Analytica Chimica Acta},
volume = {787},
pages = {1-9},
year = {2013},
issn = {0003-2670},
doi = {https://doi.org/10.1016/j.aca.2013.04.034},
url = {https://www.sciencedirect.com/science/article/pii/S0003267013005618},
author = {Todeschini, R. and Ballabio, D. and Consonni, V. and Sahigara, F. and Filzmoser, P.}
}

@article{python2021pot,
author = {Flamary, R. and
          Courty, N. and
          Gramfort, A. and
          Alaya, M. Z. and
          Boisbunon, A. and
          Chambon, S. and
          Chapel, L. and
          Corenflos, A. and
          Fatras, K. and
          Fournier, N. and
          Gautheron, L. and
          Gayraud, N. T. H. and
          Janati, H. and
          Rakotomamonjy, A. and
          Redko, I. and
          Rolet, A. and
          Schutz, A. and
          Seguy, V. and
          Sutherland, D. J. and
          Tavenard, R. and
          Tong, A. and
          Vayer, T.},
  title   = {POT: Python Optimal Transport},
  journal = {Journal of Machine Learning Research},
  year    = {2021},
  volume  = {22},
  number  = {78},
  pages   = {1-8},
  url     = {http://jmlr.org/papers/v22/20-451.html}
}

@article{stewart1977pseudo-inverse-purturbation,
 ISSN = {00361445, 10957200},
 URL = {http://www.jstor.org/stable/2030248},
 author = {G. W. Stewart},
 journal = {SIAM Review},
 number = {4},
 pages = {634--662},
 publisher = {Society for Industrial and Applied Mathematics},
 title = {On the Perturbation of Pseudo-Inverses, Projections and Linear Least Squares Problems},
 urldate = {2024-12-05},
 volume = {19},
 year = {1977}
}

@misc{Tao2010Eigenvalues,
  author = {Tao, T.},
  title        = {254A, Notes 3a: Eigenvalues and sums of Hermitian matrices},
  year         = {2010},
  month        = jan,
  day          = {13},
  howpublished = {\url{https://terrytao.wordpress.com/2010/01/13/254a-notes-3a-eigenvalues-and-sums-of-hermitian-matrices/}},
  note         = {Accessed: 2015-05-25}
}

@article{Weyl1912Das_asymptotische,
	author = {Weyl, H.},
	date = {1912/12/01},
	doi = {10.1007/BF01456804},
	id = {Weyl1912},
	isbn = {1432-1807},
	journal = {Mathematische Annalen},
	number = {4},
	pages = {441--479},
	title = {Das asymptotische {V}erteilungsgesetz der {E}igenwerte linearer partieller {D}ifferentialgleichungen (mit einer {A}nwendung auf die Theorie der {H}ohlraumstrahlung)},
	url = {https://doi.org/10.1007/BF01456804},
	volume = {71},
	year = {1912}}

@article{mahalanobis1936generalized,
 ISSN = {0976836X, 09768378},
 URL = {https://www.jstor.org/stable/48723335},
 author = {P. C. Mahalanobis},
 journal = {Sankhyā: The Indian Journal of Statistics, Series A (2008-)},
 number = {},
 pages = {pp. S1--S7},
 publisher = {[Springer, Indian Statistical Institute]},
 title = {ON THE GENERALIZED DISTANCE IN STATISTICS},
 urldate = {2026-05-13},
 volume = {80},
 year = {2018}
}

@article{de2000mahalanobis,
  title={The {M}ahalanobis distance},
  author = {De Maesschalck, R. and Jouan-Rimbaud, D. and Massart, D. L.},
  journal={Chemometrics and intelligent laboratory systems},
  volume={50},
  number={1},
  pages={1--18},
  year={2000},
  publisher={Elsevier}
}

@InProceedings{verde2008comparing,
author = {Verde, R. and Irpino, A.},
editor="Brito, Paula",
title="Comparing Histogram Data Using a {M}ahalanobis--{W}asserstein Distance",
booktitle="COMPSTAT 2008",
year="2008",
publisher="Physica-Verlag HD",
address="Heidelberg",
pages="77--89",
isbn="978-3-7908-2084-3"
}

@INPROCEEDINGS{lima2025robust,
  author = {Lima, D. and Sampaio, G. and Rocha, C. and Viana, J. and Gouveia, C.},
  booktitle={2025 IEEE International Conference on Systems, Man, and Cybernetics (SMC)}, 
  title={A Robust Phase Mapping Approach Using the {M}ahalanobis-{W}asserstein Distance}, 
  year={2025},
  volume={},
  number={},
  pages={7035-7040},
  doi={10.1109/SMC58881.2025.11343320}}

@article{otto2000geometry,
author = {Otto, F.},
year = {2000},
month = {04},
pages = {},
title = {The Geometry of Dissipative Evolution Equations: The Porous Medium Equation},
volume = {26},
journal = {Comm Partial Differential Equations},
doi = {10.1081/PDE-100002243}
}

@book{Groetsch1977inverse,
  author    = {Groetsch, C. W.},
  title     = {Generalized Inverses of Linear Operators: Representation and Approximation},
  series    = {Monographs and Textbooks in Pure and Applied Mathematics},
  volume    = {37},
  publisher = {Marcel Dekker, Inc.},
  address   = {New York and Basel},
  year      = {1977},
  isbn      = {0-8247-6615-6}
}

@article{Stuart2010inverse, 
title={Inverse problems: A Bayesian perspective}, volume={19}, 
DOI={10.1017/S0962492910000061}, 
journal={Acta Numerica}, 
author={Stuart, A. M.}, 
year={2010}, 
pages={451–559}}

\appendix
\renewcommand{\thesection}{Appendix \Alph{section}}
\addtocontents{toc}{\protect\setlength{\cftsecnumwidth}{7em}}

\renewcommand{\thesubsection}{\Alph{section}.\arabic{subsection}}
\titleformat{\subsection}
  {\normalfont\large\bfseries}
  {\Alph{section}.\arabic{subsection}}
  {1em}
  {}

\section{Proofs}\label{sec:AppendixA}
\subsection{Proof of \texorpdfstring{\Cref{Gaussian Linear Theorem}}{}} \label{Proof: Gaussian linear theorem}

\begin{proof}[Proof of \Cref{Gaussian Linear Theorem}]
        Since $f$ is an affine transformation, we have
\begin{align*}
    f_{\sharp} \alpha = \mathcal{N}(A m_\alpha + b, \Sigma_{f_{\sharp} \alpha}), \quad
    f_{\sharp} \beta = \mathcal{N}(A m_\beta + b, \Sigma_{f_{\sharp} \beta}),
\end{align*}
where the covariance matrices are
$
    \Sigma_{f_{\sharp} \alpha} = \Sigma_{f_{\sharp} \beta} = A \Sigma A^\top.
$
The optimal transport map between $f_{\sharp} \alpha$ and $f_{\sharp} \beta$ is
$    T^{f_{\sharp} \alpha}_{f_{\sharp} \beta}(y) = y + A ( m_\alpha - m_\beta  ).
$
This leads to constant OT displacements
\begin{equation*}
    g^{f_{\sharp} \alpha}_{f_{\sharp} \beta} = T^{f_{\sharp} \alpha}_{f_{\sharp} \beta}(y) - y = A ( m_\alpha - m_\beta  ).
\end{equation*}

Let  $N_{f_\sharp\beta} = \{g^{f_{\sharp} \beta_j}_{f_{\sharp} \beta} \}_{j=1}^K$ be the set that contains the displacement functions from $f_{\sharp} \beta$ to its neighboring Gaussians $ f_{\sharp} \beta_j$, where
$
    f_{\sharp} \beta_j = \mathcal{N} \left( A m_{\beta_j} + b, \Sigma_{f_{\sharp} \beta_j}\right),
$
with $\Sigma_{f_{\sharp} \beta_j} = A \Sigma A^\top$.
It follows that each displacement function is
\begin{equation*}
    g^{f_{\sharp} \beta_j}_{f_{\sharp} \beta}(y) = T^{f_{\sharp} \beta_j}_{f_{\sharp} \beta}(y) - y = A (m_{\beta_j} - m_\beta ).
\end{equation*}
By definition we have
$
  F_{N_{f_\sharp\beta}} =\frac{1}{K}\sum_{k=1}^K g^{f_{\sharp} \beta_k}_{f_{\sharp} \beta} \langle g^{f_{\sharp} \beta_k}_{f_{\sharp} \beta}, \cdot \rangle_{f_\sharp \beta}.
$
Let $\phi \in L^2(f_\sharp\beta)$ and let $ x\in \mathbb{R}^n$. Applying $F_{N_{f_\sharp\beta}}$ to $\phi$, we get

\begin{align*}
      F_{N_{f_\sharp\beta}} (\phi)(x)  &= \frac{1}{K}\sum_{k=1}^K g^{f_{\sharp} \beta_k}_{f_{\sharp} \beta}(x) \langle g^{f_{\sharp} \beta_k}_{f_{\sharp} \beta}, \phi \rangle_{f_\sharp \beta} 
      = \frac{1}{K}\sum_{k=1}^K g^{f_{\sharp} \beta_k}_{f_{\sharp} \beta}(x) \int_{\mathbb{R}^n} (g^{f_{\sharp} \beta_k}_{f_{\sharp} \beta}(y))^\top \phi(y) d f_\sharp \beta (y)\\
      & = \frac{1}{K}\sum_{k=1}^K [A (m_{\beta_k} - m_\beta)] \int_{\mathbb{R}^n} [A (m_{\beta_k} - m_\beta)]^\top \phi(y) d f_\sharp \beta (y)\\
    & = \frac{1}{K}\sum_{k=1}^K [A (m_{\beta_k} - m_\beta)]  [A (m_{\beta_k} - m_\beta)]^\top  \int_{\mathbb{R}^n}\phi(y) d f_\sharp \beta (y)\\
    & = \frac{1}{K}VV^\top  \int_{\mathbb{R}^n}\phi(y)  d f_\sharp \beta (y),
\end{align*}
where $V = [A(m_{\beta_1} - m_\beta), \cdots ,A(m_{\beta_K} - m_\beta)]$. 

The pseudo-inverse of $F_{N_{f_\sharp\beta}}$ is defined by
\begin{equation}\label{pesudo iverse linear case}
  F_{N_{f_\sharp\beta}}^\dagger = \sum_{k=1}^K \sum_{\ell=1}^K (B^\dagger )_{k\ell} 
   g^{f_{\sharp} \beta_k}_{f_{\sharp} \beta}
    \langle g^{f_{\sharp} \beta_\ell}_{f_{\sharp} \beta} , \cdot\rangle_{f_\sharp \beta},
\end{equation}
where $B_{kl}= \langle F_{N_{f_\sharp\beta}}(g^{f_{\sharp} \beta_\ell}_{f_{\sharp} \beta} ), g^{f_{\sharp} \beta_k}_{f_{\sharp} \beta}  \rangle _ {f_\sharp \beta}$. To find $B$, we first compute
\begin{equation*}
  F_{N_{f_\sharp\beta}}(g^{f_{\sharp} \beta_\ell}_{f_{\sharp} \beta} )
   = 
  \frac{1}{K} V  V^\top \int_{\mathbb{R}^n}g^{f_{\sharp} \beta_\ell}_{f_{\sharp} \beta} (y) d f_\sharp \beta (y) =   \frac{1}{K} V  V^\top [A (m_{\beta_l} - m_\beta)].
\end{equation*}
We then get the following.
\begin{align*}
  B_{k\ell}
   &= 
  \langle F_{N_{f_\sharp\beta}}(g^{f_{\sharp} \beta_\ell}_{f_{\sharp} \beta}), g^{f_{\sharp} \beta_k}_{f_{\sharp} \beta}\rangle_{f_\sharp \beta}
   = 
  \langle 
    \tfrac{1}{K} V  V^\top [A (m_{\beta_l} - m_\beta)] ,
     [A (m_{\beta_k} - m_\beta)]
  \rangle_{f_\sharp \beta}\\
   &= 
  \frac{1}{K} 
  \langle 
     V^\top [A (m_{\beta_l} - m_\beta)], 
     V^\top [A (m_{\beta_k} - m_\beta)]
  \rangle_{f_\sharp \beta}
   = 
  \frac{1}{K}
    ( V^\top [A (m_{\beta_l} - m_\beta)] )^\top
    ( V^\top [A (m_{\beta_k} - m_\beta)] )\\
    &=\frac{1}{K}\sum_{i=1}^K ( V^\top [A (m_{\beta_l} - m_\beta)])_i( V^\top [A (m_{\beta_k} - m_\beta)])_i.
\end{align*}
We also have $( V^\top [A (m_{\beta_k} - m_\beta)])_i = \langle g^{f_{\sharp} \beta_i}_{f_{\sharp} \beta},g^{f_{\sharp} \beta_k}_{f_{\sharp} \beta}\rangle_{f_\sharp \beta} = [A (m_{\beta_i} - m_\beta)]^\top [A (m_{\beta_k} - m_\beta)]$, and thus 
$
  B_{k\ell} =
  \frac{1}{K} ( V^\top V)^2_{k\ell}.
$
The Moore--Penrose inverse of $B$ can be expressed as $B^\dagger = (\frac{1}{K}( V^\top V)^2)^\dagger$.
Applying $F_{N_{f_\sharp\beta}}^\dagger$ to $\phi $ yields
\begin{align*}
  F_{N_{f_\sharp\beta}}^\dagger(\phi)(x)
   &= 
  \sum_{k=1}^K 
   \sum_{\ell=1}^K 
      B^\dagger_{k\ell} 
      g^{f_{\sharp} \beta_k}_{f_{\sharp} \beta}(x)  \langle g^{f_{\sharp} \beta_\ell}_{f_{\sharp} \beta},\phi\rangle_{f_\sharp \beta}
  =
 \sum_{k=1}^K 
   \sum_{\ell=1}^K g^{f_{\sharp} \beta_k}_{f_{\sharp} \beta} (x)B^\dagger_{k\ell} \left[  \int_{\mathbb{R}^n} (g^{f_{\sharp} \beta_\ell}_{f_{\sharp} \beta}(y))^\top\phi(y) d f_\sharp \beta (y)\right] \\
  &=
 \sum_{k=1}^K 
   \sum_{\ell=1}^K [A (m_{\beta_k} - m_\beta)]B^\dagger_{k\ell}   [A (m_{\beta_l} - m_\beta)]^\top\left[\int_{\mathbb{R}^n} \phi(y) d f_\sharp \beta (y)\right] \\
   &= VB^\dagger V^\top \int_{\mathbb{R}^n}\phi(y) d f_\sharp \beta (y).
\end{align*}
To further simplify, we apply the SVD to $V$: 
$V = U S W^T$. This implies $V^T V = W S^2 W^T$ and $(V^T V)^2 = W S^4 W^T$. Taking the pseudo-inverse yields $((V^T V)^2)^\dagger = W S^{-4} W^T$.
It then follows that
\begin{equation*}
V  ((V^T V)^2)^\dagger V^T
= (U S W^T)(W S^{-4} W^T)(U S W^T)^T
= U S^{-2} U^T
= (V V^T)^\dagger.
\end{equation*}
Recall that
\begin{equation*}
B^\dagger
= (\frac{1}{K}(V^T V)^2)^\dagger
= K ((V^T V)^2)^\dagger.
\end{equation*}
Therefore,
\begin{equation*}
V B^\dagger V^T
= K V ((V^T V)^2)^\dagger V^T
= K(V V^T)^\dagger
= \left(\frac{1}{K} V V^T\right)^\dagger
\end{equation*}
and
\begin{align*}
  F_{N_{f_\sharp\beta}}^\dagger(\phi)(x) &= (\frac{1}{K} V V^\top)^\dagger \int_{\mathbb{R}^n}\phi(y) d f_\sharp \beta (y).
\end{align*}
We know $\int_{\mathbb{R}^n} g^{f_{\sharp} \alpha}_{f_{\sharp} \beta}(y) d  f_\sharp \beta (y) = \int_{\mathbb{R}^n} \left[ A ( m_\alpha - m_\beta  ) \right] d  f_\sharp \beta (y) = A ( m_\alpha - m_\beta  )$ and
therefore we have
\begin{equation*}
    F^\dagger_{N_{f_\sharp\beta}} \left( g^{f_{\sharp} \alpha}_{f_{\sharp} \beta}  \right) = \left[ \frac{1}{K} \sum_{j=1}^K \left[ A (m_{\beta_j} - m_\beta) \right] \left[ A (m_{\beta_j} - m_\beta) \right]^\top \right]^\dagger \left[ A (  m_\alpha - m_\beta ) \right].
\end{equation*}
Recognizing that
\begin{equation*}
    C_{f(m_\beta)} = \frac{1}{K} \sum_{j=1}^K \left[ A (m_{\beta_j} - m_\beta) \right] \left[ A (m_{\beta_j} - m_\beta) \right]^\top
\end{equation*}
is the local covariance matrix of $\{f(m_{\beta_j})\}_{j=1^K}$, we can rewrite $F^\dagger_{N_{f_\sharp\beta}} \left( g^{f_{\sharp} \alpha}_{f_{\sharp} \beta}  \right)$ as
\begin{equation*}
    F^\dagger_{N_{f_\sharp\beta}} \left( g^{f_{\sharp} \alpha}_{f_{\sharp} \beta} \right) = C_{f(m_\beta)}^\dagger \left[ A ( m_\alpha - m_\beta  ) \right].
\end{equation*}
This leads to the first part inner product of the Wasserstein Mahalanobis distance:
\begin{align*}
    \left\langle g^{f_{\sharp} \alpha}_{f_{\sharp} \beta} ,   F^\dagger_{N_{f_\sharp\beta}} \left( g^{f_{\sharp} \alpha}_{f_{\sharp} \beta}  \right) \right\rangle_{f_{\sharp} \beta}
    &= \left[A ( m_\alpha - m_\beta  ) \right]^\top \left[ C_{f(m_\beta)}^\dagger \left(A ( m_\alpha - m_\beta  ) \right) \right] \\
    &= \left[ A ( m_\alpha -m_\beta  ) \right]^\top C_{f(m_\beta)}^\dagger \left[ A ( m_\alpha -m_\beta  ) \right].
\end{align*}
Since $A ( m_\alpha -m_\beta  )  =  f(m_\alpha) -f(m_\beta) $, we have
\begin{equation*}
    \left\langle g^{f_{\sharp} \alpha}_{f_{\sharp} \beta} ,   F^\dagger_{N_{f_\sharp\beta}} \left( g^{f_{\sharp} \alpha}_{f_{\sharp} \beta}  \right) \right\rangle_{f_{\sharp} \beta} = \left[  f(m_\alpha) -f(m_\beta)  \right]^\top C_{f(m_\beta)}^\dagger \left[  f(m_\alpha) -f(m_\beta)  \right].
\end{equation*}
Similarly, considering the neighbors $\{ \alpha_i \}_{i=1}^N$ of $\alpha$, we can derive the corresponding inner product

\begin{equation*}
    \left\langle g^{f_{\sharp} \beta}_{f_{\sharp} \alpha} ,   F^\dagger_{N_{f_{\sharp} \alpha}} \left( g_{f_{\sharp} \alpha}^{f_{\sharp} \beta}  \right) \right\rangle_{f_{\sharp} \alpha} = \left[ f(m_\beta) - f(m_\alpha) \right]^\top C_{f(m_\alpha)}^\dagger \left[ f(m_\beta) - f(m_\alpha) \right].
\end{equation*}
It follows that
\begin{equation*}
    \dwm^2(f_{\sharp} \alpha, f_{\sharp} \beta) =  d_M^2(f(m_\alpha), f(m_\beta)),
\end{equation*}
with respect to the neighborhoods $\{\alpha_i\}_{i=1}^N$, $\{\beta_j\}_{j=1}^K$ and $\{m_{\alpha_i}\}_{i=1}^N$, $\{m_{\beta_j}\}_{j=1}^K$. 
\end{proof}

\subsection{Proof of \texorpdfstring{\Cref{remark: displacement function}}{}} \label{proof:remark: displacement function}

\begin{proof}[Proof of \Cref{remark: displacement function}]
        Let $T^{\rho}_{\nu}$ be the OT map from $\nu$ to $\rho$, namely, 
$
T^{\rho}_{\nu}(x) = x + ( m_\rho - m_\nu ).
$
Define $\widetilde{T} = f \circ T^{\rho}_{\nu} \circ f^{-1}$, which implies
$
\widetilde{T}_{\sharp}(f_{\sharp}\nu) = f_{\sharp}\rho.
$
Therefore $\widetilde{T}$ is a map from $f_{\sharp}\nu$ to $f_{\sharp}\rho$. By definition of the OT displacement function we obtain
\begin{align*}
\int_{\mathbb{R}^n} \| g_{ f_{\sharp}\nu}^{f_{\sharp}\rho}(y) \|^2   d(f_{\sharp}\nu)(y)
= \int_{\mathbb{R}^n} \| T_{ f_{\sharp}\nu}^{f_{\sharp}\rho}(y) - y \|^2   d(f_{\sharp}\nu)(y)
\le \int_{\mathbb{R}^n} \| \widetilde{T}(y) - y \|^2   d(f_{\sharp}\nu)(y).
\end{align*}
Consider a change of variables $y = f(x)$. Since $x \sim \nu$, we have
\begin{align*}
\widetilde{T}(y) = f \left(T^{\rho}_{\nu}(f^{-1}(y))\right) 
= f \left(T^{\rho}_{\nu}(x)\right)
= f \left(x +  m_\rho - m_\nu \right).
\end{align*}
This implies
\begin{equation*}
\int_{\mathbb{R}^n} \| \widetilde{T}(y) - y \|^2   d(f_{\sharp}\nu)(y)
= \int_{\mathbb{R}^n} \| f(x + m_\rho -m_\nu ) - f(x) \|^2   d\nu(x).
\end{equation*}

Consider the Taylor expansion
$
f(x +  m_\rho - m_\nu ) - f(x)
= J_f(x) ( m_\rho - m_\nu ) + O\left(\| m_\rho -m_\nu  \|^2\right),
$
which yields
\begin{align*}
\| f(x + m_\rho -m_\nu ) - f(x) \|^2
&= \|J_f(x) ( m_\rho - m_\nu ) + O\left(\| m_\rho -m_\nu\|^2 \right)  \|^2 \\
&\le \|J_f(x) ( m_\rho - m_\nu )\|^2 + O\left(\| m_\rho -m_\nu\|^3 \right)  \\
&\le \| J_f(x) \|_2^{2}   \| m_\rho -m_\nu  \|^2
+ O\left(\| m_\rho -m_\nu \|^3\right).
\end{align*}
Thus,
\begin{align*}
\| g_{ f_{\sharp}\nu}^{f_{\sharp}\rho} \|_{L^2(f_{\sharp}\nu)}^{2}
& \le \int_{\mathbb{R}^n} \| f(x + m_\rho -m_\nu ) - f(x) \|^2   d\nu(x) \\
&\le \| m_\rho - m_\nu \|^2 \int_{\mathbb{R}^n} \| J_f(x) \|_2^{ 2}   d\nu(x)
+ O \left(\| m_\rho - m_\nu \|^3\right).
\end{align*}
Since we have that
$
g^{\tilde f_\sharp\rho}_{\tilde f_\sharp\nu}
=
J_f(m_\nu)(m_\rho-m_\nu)
$
is a constant function, it follows that
\begin{align*}
\big\| g_{ f_{\sharp}\nu}^{f_{\sharp}\rho} - g^{\tilde f_\sharp\rho}_{\tilde f_\sharp\nu} \big\|_{L^2(f_{\sharp}\nu)}^{2}
&\le 2\| g_{ f_{\sharp}\nu}^{f_{\sharp}\rho} \|_{L^2(f_{\sharp}\nu)}^{2}
+ 2\| g^{\tilde f_\sharp\rho}_{\tilde f_\sharp\nu} \|_{L^2(f_{\sharp}\nu)}^{2} \\
&\le 2\| m_\rho - m_\nu \|^2
\left[
\int_{\mathbb{R}^n} \| J_f(x) \|_2^{2}   d\nu(x)
+ \| J_f(m_\nu) \|_2^{2}
\right]
+ O(\| m_\rho - m_\nu \|^3). \qedhere
\end{align*}

\end{proof}

\subsection{Detailed computation for \texorpdfstring{\Cref{example:displacement_function}}{}}\label{computation_of_example_1}
\begin{proof}[Computation for \Cref{example:displacement_function}]
    Let
$
	f(x)=x+x^3$,
which is a strictly increasing function. Let
\begin{equation*}
\alpha=\mathcal N(m_\alpha,\sigma^2),
\qquad
\beta=\mathcal N(m_\beta,\sigma^2),
\quad \text{and} \quad
h=m_\beta-m_\alpha .
\end{equation*}
Denote the first-order approximation of $f$ at $m_\alpha$ by
\begin{equation*}
\tilde f^{m_\alpha}(x)
=
f(m_\alpha)+f'(m_\alpha)(x-m_\alpha),
\quad \text{where} \quad
f'(x)=1+3x^2 .
\end{equation*}
Since $\alpha$ and $\beta$ have the same variance, the optimal transport map
from $\alpha$ to $\beta$ is the translation
$
T_\alpha^\beta(x)=x+h .
$
Define the map
$
 T = f \circ T_\alpha^\beta \circ f^{-1}.
$
Then $T$ is increasing and satisfies
$
T_\sharp(f_\sharp\alpha)=f_\sharp\beta$. Therefore, $T$ is the OT map
from $f_\sharp\alpha$ to $f_\sharp\beta$. Therefore, if $y=f(x)$, then
\begin{equation*}
    g_{f_\sharp\alpha}^{f_\sharp\beta}(y)
    =
    T(y)-y
    =
    f(x+h)-f(x) =
    h+3x^2h+3xh^2+h^3.
\end{equation*}
Since $\tilde f^{m_\alpha}$ is affine, $\tilde f^{m_\alpha}_\sharp\alpha$ and $\tilde f^{m_\alpha}_\sharp\beta$ remain Gaussian. Therefore, the displacement function from $\tilde f^{m_\alpha}_\sharp\alpha$ to $\tilde f^{m_\alpha}_\sharp\beta$ is 
\begin{equation*}
g_{\tilde f^{m_\alpha}_\sharp\alpha}^{\tilde f^{m_\alpha}_\sharp\beta}(y)
=\tilde f^{m_\alpha} (m_\beta) - \tilde f^{m_\alpha}(m_\alpha) = (1+3m_\alpha^2)h,
\end{equation*}
which is a constant function. 
For $X\sim\mathcal N(m_\alpha,\sigma^2)$ we obtain
\begin{align*}
\|g_{f_\sharp\alpha}^{f_\sharp\beta}
-
g_{\tilde f^{m_\alpha}_\sharp\alpha}^{\tilde f^{m_\alpha}_\sharp\beta}
\|_{L^2(f_\sharp\alpha)}^2
& =
\mathbb E\left[
    \left(
    h+3X^2h+3Xh^2+h^3
    -
    (1+3m_\alpha^2)h
    \right)^2
\right] \notag \\
 &=
\mathbb E\left[
    \left(
    3h(X^2-m_\alpha^2)+3Xh^2+h^3
    \right)^2
\right].
\end{align*}
Using Gaussian moments for $X\sim\mathcal N(m_\alpha,\sigma^2)$, we get
\begin{equation*}
\mathbb E[X]=m_\alpha,\quad
\mathbb E[X^2]=m_\alpha^2+\sigma^2,\quad
\mathbb E[X^3]=m_\alpha^3+3m_\alpha\sigma^2,\quad
\mathbb E[X^4]=m_\alpha^4+6m_\alpha^2\sigma^2+3\sigma^4.
\end{equation*}
This implies
\begin{align*}
\mathbb E\Bigl[\bigl(3h(X^2-m_\alpha^2)+3h^2X+h^3\bigr)^2\Bigr]
&=
9(4m_\alpha^2\sigma^2+3\sigma^4)h^2
+54m_\alpha\sigma^2h^3
+(9m_\alpha^2+15\sigma^2)h^4
+6m_\alpha h^5
+h^6,
\end{align*}
leading to
\begin{equation*}
\|g_{f_\sharp\alpha}^{f_\sharp\beta} - g_{\tilde f^{m_\alpha}_\sharp\alpha}^{\tilde f^{m_\alpha}_\sharp\beta}\|_{L^2(f_\sharp\alpha)}^2
=
9(4m_\alpha^2\sigma^2+3\sigma^4)h^2
+54m_\alpha\sigma^2h^3
+(9m_\alpha^2+15\sigma^2)h^4
+6m_\alpha h^5
+h^6.
\end{equation*}
This expression has leading order $O(h^2)$.
\end{proof}

\subsection{Detailed computation for \texorpdfstring{\Cref{exp:assuption_D}}{}}\label{computation_of_example_2}
\begin{proof}[Computation for \Cref{exp:assuption_D}]
Let $m_\nu=(a,b)$, $m_{\nu_i}=(a_i,b_i)$, $m_{\nu_i}-m_\nu
=
(\Delta a_i,\Delta b_i)$.
Since $\nu$ and $\nu_i$ have the same covariance, the optimal transport map from $\nu$ to $\nu_i$ is
\begin{equation*}
T_{\nu}^{\nu_i}(x,y)
=
(x+\Delta a_i,y+\Delta b_i).
\end{equation*}
Denote by $\tilde f$ the linear approximation of  $f $ at  $m_\nu$ \eqref{eq:affine-approx}. By direct computation we obtain
\begin{equation*}
\tilde f(x,y)
=
\begin{pmatrix}
a+\epsilon b^3+(x-a)+3\epsilon b^2(y-b)\\
b-\epsilon a^3-3\epsilon a^2(x-a)+(y-b)

\end{pmatrix}.
\end{equation*}
Since $\tilde f$ is affine, the OT displacement from $\tilde f_\sharp\nu$ to $\tilde f_\sharp\nu_i$ is
\begin{equation*}
g^{\tilde f_\sharp\nu_i}_{\tilde f_\sharp\nu}
=
\tilde f(m_{\nu_i})-\tilde f(m_\nu)
=
\begin{pmatrix}
\Delta a_i+3\epsilon b^2\Delta b_i\\
\Delta b_i-3\epsilon a^2\Delta a_i
\end{pmatrix}.
\end{equation*}
Let $T_i=T^{f_\sharp\nu_i}_{f_\sharp\nu}$
be the OT map from  $f_\sharp\nu $ to  $f_\sharp\nu_i $. Define $\widehat T_i
=
f\circ T_{\nu}^{\nu_i}\circ f^{-1}$ such that $(\widehat T_i)_\sharp(f_\sharp\nu)=f_\sharp\nu_i$. Since both $T_i$, $\widehat T_i$ push $f_\sharp\nu$ forward to $f_\sharp\nu_i$, we have
\begin{equation}\label{eq:T_i_equal}
\int T_i(z) d(f_\sharp\nu)(z)
=
\int \widehat T_i(z) d(f_\sharp\nu)(z) = \int w d(f_\sharp\nu_i)(w).
\end{equation}
Using \eqref{eq:T_i_equal} and the optimality of $T_i$ we get
\begin{align*}
\left\|
g^{f_\sharp\nu_i}_{f_\sharp\nu}
-
g^{\tilde f_\sharp\nu_i}_{\tilde f_\sharp\nu}
\right\|_{L^2(f_\sharp\nu)}^2=\int
\left\|
T_i(z)-z-g^{\tilde f_\sharp\nu_i}_{\tilde f_\sharp\nu}
\right\|^2
 d(f_\sharp\nu)(z) \leq
\int
\left\|
\widehat T_i(z)-z-g^{\tilde f_\sharp\nu_i}_{\tilde f_\sharp\nu}
\right\|^2
 d(f_\sharp\nu)(z).
\end{align*}
Apply change of variable with  $z=f(x,y) $, where $(X,Y)\sim\mathcal N((a,b),I_2)$, we obtain
\begin{align*}
    \int
\left\|
\widehat T_i(z)-z-g^{\tilde f_\sharp\nu_i}_{\tilde f_\sharp\nu}
\right\|^2
 d(f_\sharp\nu)(z) 
 &= \int \left\|
f\circ T_{\nu}^{\nu_i} (u)-f(u)-g^{\tilde f_\sharp\nu_i}_{\tilde f_\sharp\nu}(f(u))
\right\|^2
 d\nu(u)\\
&=\mathbb E
\left[
\left\|
f(X+\Delta a_i,Y+\Delta b_i)
-
f(X,Y)
-
J_f(a,b)
\begin{pmatrix}
\Delta a_i\\
\Delta b_i
\end{pmatrix}
\right\|^2
\right].
\end{align*}
Using Gaussian moments we obtain
\begin{align*}
&
\mathbb E
\left[
\left\|
f(X+\Delta a_i,Y+\Delta b_i)
-
f(X,Y)
-
J_f(a,b)
\begin{pmatrix}
\Delta a_i\\
\Delta b_i
\end{pmatrix}
\right\|^2
\right].
\\
&\leq
\epsilon^2
\Bigl[
\left(27+36a^2\right)(\Delta a_i)^2
+
54a(\Delta a_i)^3
+
\left(15+9a^2\right)(\Delta a_i)^4
+
6a(\Delta a_i)^5
+
(\Delta a_i)^6
\\
&+
\left(27+36b^2\right)(\Delta b_i)^2
+
54b(\Delta b_i)^3
+
\left(15+9b^2\right)(\Delta b_i)^4
+
6b(\Delta b_i)^5
+
(\Delta b_i)^6
\Bigr],
\end{align*}
Since $m_{\nu_i}\in B(m_\nu,\delta)$, $(\Delta a_i)^2+(\Delta b_i)^2
=
\|m_{\nu_i}-m_\nu\|^2$, $|\Delta a_i|,|\Delta b_i|\leq\delta$, $|a|,|b|\leq\|m_\nu\|$, a direct calculation gives
\begin{equation*}
\left\|
g^{f_\sharp\nu_i}_{f_\sharp\nu}
-
g^{\tilde f_\sharp\nu_i}_{\tilde f_\sharp\nu}
\right\|_{L^2(f_\sharp\nu)}^2
\leq
\epsilon^2
\Bigl[
27+36\|m_\nu\|^2
+54\|m_\nu\|\delta
+(15+9\|m_\nu\|^2)\delta^2
+6\|m_\nu\|\delta^3
+\delta^4
\Bigr]
\|m_{\nu_i}-m_\nu\|^2.
\end{equation*}
Thus
\begin{equation*}
\left\|
g^{f_\sharp\nu_i}_{f_\sharp\nu}
-
g^{\tilde f_\sharp\nu_i}_{\tilde f_\sharp\nu}
\right\|_{L^2(f_\sharp\nu)}^2
\leq
c_\xi\|m_{\nu_i}-m_\nu\|^2,
\end{equation*}
where
\begin{equation}\label{eq:c_xi_example}
c_\xi
=
\epsilon^2
\Bigl[
27+36\|m_\nu\|^2
+54\|m_\nu\|\delta
+(15+9\|m_\nu\|^2)\delta^2
+6\|m_\nu\|\delta^3
+\delta^4
\Bigr].
\end{equation}
Let $\lambda^+$ and $\lambda^-$ be the largest and smallest eigenvalues of $J_f(m_\nu)^{\top}J_f(m_\nu)$, respectively. A calculation gives
\begin{equation*}
\begin{aligned}
\lambda^\pm
&=
1+\frac{9\epsilon^2}{2}(a^4+b^4)
\pm
\frac{1}{2}
\sqrt{
81\epsilon^4(a^4-b^4)^2
+
36\epsilon^2(a^2-b^2)^2
}.
\end{aligned}
\end{equation*}
Let $\kappa_\nu
=
\frac{\|J_f(m_\nu)\|_2}
{\sigma_{\min}(J_f(m_\nu))}
=
\sqrt{\frac{\lambda^+}{\lambda^-}}$. Since $n=2$, the upper bound in \Cref{assuption:dispacement_local} becomes
\begin{equation*}
c_\xi
<
\min
\left\{
1,
\frac{(\lambda^-)^4}{57600},
\frac{(\lambda^-)^4}
{57600(\lambda^+)^3}
\right\}.
\end{equation*}
Since $m_\nu,m_{\nu_i}\in[0,1]^2$ we have $\|m_\nu\|\leq\sqrt{2}$ and \eqref{eq:c_xi_example} can be simplified to
\begin{equation*}
c_\xi
\leq
\epsilon^2
\Bigl[
99
+54\sqrt{2}\delta
+33\delta^2
+6\sqrt{2}\delta^3
+\delta^4
\Bigr].
\end{equation*}
Moreover, we have $a^4+b^4\leq2$, $|a^4-b^4|\leq1$, $|a^2-b^2|\leq1$, which implies
\begin{align*}
\lambda^-
\geq
1-
\frac{1}{2}
\sqrt{
81\epsilon^4+36\epsilon^2
}, \qquad\lambda^+
\leq
1+9\epsilon^2+
\frac{1}{2}
\sqrt{
81\epsilon^4+36\epsilon^2.
}
\end{align*}
By direct computation, one can show that a sufficient condition for \Cref{assuption:dispacement_local} to hold uniformly for all $m_\nu,m_{\nu_i}\in[0,1]^2$ is 
\begin{equation*}
\epsilon^2
\Bigl[
99
+54\sqrt{2}\delta
+33\delta^2
+6\sqrt{2}\delta^3
+\delta^4
\Bigr] <\frac{
\left(
1-
\frac{1}{2}
\sqrt{
81\epsilon^4+36\epsilon^2
}
\right)^4
}
{
57600
\left(
1+9\epsilon^2
+
\frac{1}{2}
\sqrt{
81\epsilon^4+36\epsilon^2
}
\right)^3
}.
\end{equation*}
When $\epsilon\to0$, the left hand side goes to $0$, while the right hand side goes to $\frac{1}{57600}$.
Thus, $|\epsilon|$ can be chose sufficiently small such that \Cref{assuption:dispacement_local} is satisfied.
\end{proof}

\subsection{Proof of \texorpdfstring{\Cref{cor:point_mass_limit}}{}}

\begin{proof}[Proof of \Cref{cor:point_mass_limit}]\label{Proof:point_mass_limit}
    As $\Sigma\to0_{n\times n}$, Gaussian measures collapse to point masses at their respective means. Denote by $\tilde f^{m_\beta}$, $\tilde f^{m_\alpha}$ the linear approximations of $f$ at $m_\beta$ and $m_\alpha$, respectively, see \eqref{eq:affine-approx}. We then have
\begin{equation*}
\begin{aligned}
\lim_{\quad\Sigma\to0_{n\times n}}
H\left(
\tilde f^{m_\beta}_\sharp\beta,
f_\sharp\beta
\right)
&=0,\\
\lim_{\quad\Sigma\to0_{n\times n}}
H\left(
\tilde f^{m_\alpha}_\sharp\alpha,
f_\sharp\alpha
\right)
&=0.
\end{aligned}
\end{equation*}
Moreover, the displacement functions satisfy
\begin{equation*}
\lim_{\quad\Sigma\to0_{n\times n}}
g_{f_\sharp\beta}^{f_\sharp\alpha}
=
f(m_\alpha)-f(m_\beta)
\quad \text{and} \quad
\lim_{\quad\Sigma\to0_{n\times n}}
g_{\tilde f^{m_\beta}_\sharp\beta}
 ^{\tilde f^{m_\beta}_\sharp\alpha}
=
J_f(m_\beta)(m_\alpha-m_\beta).
\end{equation*}
By Taylor expansion
\begin{equation*}
f(m_\alpha)-f(m_\beta)
=
J_f(m_\beta)(m_\alpha-m_\beta)
+
O\left(
\|m_\alpha-m_\beta\|^2
\right),
\end{equation*}
hence,
\begin{equation*}
\lim_{\quad\Sigma\to0_{n\times n}}
\left\|
g_{f_\sharp\beta}^{f_\sharp\alpha}
-
g_{\tilde f^{m_\beta}_\sharp\beta}
 ^{\tilde f^{m_\beta}_\sharp\alpha}
\right\|_{L^2(f_\sharp\beta)}^2=
O\left(
\|m_\alpha-m_\beta\|^4
\right).
\end{equation*}
Thus, \Cref{displacement function assumption} always holds with $p=4$ when $\Sigma\to0_{n\times n}$. Substituting $p=4$ and the vanishing Hellinger terms into \Cref{theorem:main} gives
\begin{equation*}
\lim_{\quad\Sigma\to0_{n\times n}}
\dwm^2(f_\sharp\alpha,f_\sharp\beta)
=
d_M^2\left(f(m_\alpha),f(m_\beta)\right)+
O\left(
\frac{\|m_\alpha-m_\beta\|^2}{\delta}
\right).
\end{equation*}
Multiplying by $\delta^2/(n+2)$, we obtain
\begin{equation}\label{eq:limit_dwm}
\frac{\delta^2}{n+2}
\lim_{\quad\Sigma\to0_{n\times n}}
\dwm^2(f_\sharp\alpha,f_\sharp\beta)
=
\frac{\delta^2}{n+2}
d_M^2\left(f(m_\alpha),f(m_\beta)\right)+
O\left(
\delta\|m_\alpha-m_\beta\|^2
\right).
\end{equation}
Combining \eqref{eq:limit_dwm} and \Cref{singer's result} proves the result.
\end{proof}

\subsection{Proof of \texorpdfstring{\Cref{rmk:cov inverse Eulidean}}{}}

\begin{proof}[Proof of \Cref{rmk:cov inverse Eulidean}]\label{proof:cov inverse Euclidean} 
We separate the proof into three parts:

\begin{enumerate}[label=\emph{(\roman*)}, leftmargin=0pt, itemindent=*, listparindent=\parindent]
    \item \label{part:covariance_euclidean}
    $\|C_{f(m_\beta)}^\dagger\|_2=O(\delta^{-2})$
    \item \label{part:covariance_difference}
    $ \|C_{f(m_\beta)}-M_{N_{\tilde f^{m_\beta}_\sharp\beta}}\|_2=O(\delta^3)$
    \item \label{part:covariance_Peusdoinverse_difference}
    $\|C_{f(m_\beta)}^\dagger-M_{N_{\tilde f^{m_\beta}_\sharp\beta}}^\dagger\|_2=O(\delta^{-1})$
\end{enumerate}

\begin{enumerate}[label=\emph{(\roman*)}, leftmargin=0pt, itemindent=*, listparindent=\parindent]

\item 
   Recall that by definition,
\begin{equation}\label{eq:emp_cov_beta_2}
    C_{f(m_\beta)}
    =\frac{1}{K}\sum_{i=1}^K
    \big(f(m_{\beta_i})-f(m_\beta)\big)\big(f(m_{\beta_i})-f(m_\beta)\big)^\top.
\end{equation}
Let $X\sim \mathrm{Unif} (B(m_\beta,\delta))$ and consider
 \begin{equation*}
     \mathrm{Cov}(f(X))
=\mathbb E\Big[(f(X)-f(m_\beta))(f(X)-f(m_\beta))^\top\Big].
 \end{equation*}
Then \eqref{eq:emp_cov_beta_2} is a finite sample approximation of
$\mathrm{Cov}(f(X))$. Moreover, Taylor expansion at $m_\beta$ gives
$f(X)-f(m_\beta)=J_f(m_\beta)(X-m_\beta)+O(\|X-m_\beta\|^2)$, and since
$\|X-m_\beta\|\le\delta$, we obtain,
\begin{equation*}
    \mathrm{Cov}(f(X))
=J_f(m_\beta)\mathbb E\left[(X-m_\beta)(X-m_\beta)^\top\right]J_f(m_\beta)^\top
+O(\delta^3).
\end{equation*}
By \Cref{uniform distribute covariance},
$\mathrm{Cov}(X) = \mathbb E[(X-m_\beta)(X-m_\beta)^\top]=\frac{\delta^2}{n+2}I$, hence

\begin{align}
\mathrm{Cov}(f(X))
&=\frac{\delta^2}{n+2} J_f(m_\beta)J_f(m_\beta)^\top + O(\delta^3)=\frac{\delta^2}{n+2}\Bigr[J_f(m_\beta)J_f(m_\beta)^\top + O(\delta)\Bigr].\label{eq:cov_scaling}
\end{align}
Taking the pseudo-inverse both sides,
\begin{equation}\label{eq:cov_pesudoinverse_scaling}
\mathrm{Cov}(f(X))^\dagger
=\frac{n+2}{\delta^2}\Bigr[J_f(m_\beta)J_f(m_\beta)^\top + O(\delta)\Bigr]^\dagger.
\end{equation}
Since \eqref{eq:cov_pesudoinverse_scaling} is symmetric positive semi-definite (PSD), its spectral norm is
\begin{equation}\label{eq:cov_pinv_norm_by_eigs}
\big\|\mathrm{Cov}(f(X))^\dagger\big\|_2
=\frac{1}{\lambda_{\min}^+\big(\mathrm{Cov}(f(X))\big)},
\end{equation}
where $\lambda_{\min}^+$ denotes the smallest positive eigenvalue. By \eqref{eq:cov_scaling}
\begin{equation}\label{eq:eig_scaling_cov}
\lambda_{\min}^+\big(\mathrm{Cov}(f(X))\big)
=\frac{\delta^2}{n+2}
\lambda_{\min}^+\Big(J_f(m_\beta)J_f(m_\beta)^\top + O(\delta)\Big).
\end{equation}
Using Weyl's inequality (\Cref{Weyl's inequality}),
\begin{equation*}
\Big|\lambda_{\min}^+\Big(J_f(m_\beta)J_f(m_\beta)^\top + O(\delta)\Big)
-\lambda_{\min}^+\big(J_f(m_\beta)J_f(m_\beta)^\top\big)\Big|
\le C(\delta),
\end{equation*}
where $C(\delta) \in \mathbb{R}$ is $O(\delta)$. For $\delta$ sufficiently small  
\begin{equation*}
   C(\delta) \le \frac{1}{2}\lambda_{\min}^+\big(J_f(m_\beta)J_f(m_\beta)^\top\big),
\end{equation*}
we obtain the lower bound
\begin{equation*}
\lambda_{\min}^+\Big(J_f(m_\beta)J_f(m_\beta)^\top + O(\delta)\Big)
\ge \frac{1}{2}\lambda_{\min}^+\big(J_f(m_\beta)J_f(m_\beta)^\top\big).
\end{equation*}
Plugging this into \eqref{eq:eig_scaling_cov},  \eqref{eq:cov_pinv_norm_by_eigs} yields
\begin{equation*}
\big\|\mathrm{Cov}(f(X))^\dagger\big\|_2
\le \frac{n+2}{\delta^2}\cdot
\frac{2}{\lambda_{\min}^+\big(J_f(m_\beta)J_f(m_\beta)^\top\big)}
=O\left(\frac{1}{\delta^2}\right).
\end{equation*}
By \Cref{Assumption: empirical sampling}, we obtain
\begin{equation*}
    \big\|C_{f(m_\beta)}-\mathrm{Cov}(f(X))\big\|_2
\le \|\mathrm{Cov}(f(X))\|_2,
\end{equation*}
and since 
\begin{equation*}
    \|C_{f(m_\beta)}\|_2 \leq \big\|C_{f(m_\beta)}-\mathrm{Cov}(f(X))\big\|_2 +\|\mathrm{Cov}(f(X))\|_2
\end{equation*}
it follows that
$$\|C_{f(m_\beta)}\|_2=O(\delta^2).$$ 
To estimate $\|C_{f(m_\beta)}^\dagger\|_2$, apply Weyl's inequality (\Cref{Weyl's inequality}) again to get
\begin{equation*}
\big|\lambda_j(C_{f(m_\beta)})-\lambda_j(\mathrm{Cov}(f(X)))\big|
\le \big\|C_{f(m_\beta)}-\mathrm{Cov}(f(X))\big\|_2,
\end{equation*}
which holds for every $j$.
Hence, choosing $j$ such that $\lambda_j$ is the smallest positive eigenvalue $\lambda_{\min}^+$, we get
\begin{equation*}
\lambda_{\min}^+ \big(C_{f(m_\beta)}\big)
\ge \lambda_{\min}^+ \big(\mathrm{Cov}(f(X))\big)
-\big\|C_{f(m_\beta)}-\mathrm{Cov}(f(X))\big\|_2
\ge \frac{1}{2}\lambda_{\min}^+ \big(\mathrm{Cov}(f(X))\big).
\end{equation*}
Consequently,
\begin{equation*}
\big\|C_{f(m_\beta)}^\dagger\big\|_2
=\frac{1}{\lambda_{\min}^+ \big(C_{f(m_\beta)}\big)}
\le \frac{2}{\lambda_{\min}^+ \big(\mathrm{Cov}(f(X))\big)}
=2\big\|\mathrm{Cov}(f(X))^\dagger\big\|_2.
\end{equation*}
This yields
\begin{equation*}
\|C_{f(m_\beta)}^\dagger\|_2 = O \left(\frac 1 {\delta^2}\right).
\end{equation*}

\item 
Define the Taylor remainder for each neighboring mean by
\begin{equation*}
     \Delta_{1,i} = f(m_{\beta_i})-f(m_\beta) - J_f(m_\beta)(m_{\beta_i} - m_\beta).
\end{equation*}
Since $\|m_{\beta_i}-m_\beta\|\le\delta$, we have
$\|\Delta_{1,i}\|=O(\delta^2)$. Substitute $f(m_{\beta_i})-f(m_\beta)=J_f(m_\beta)(m_{\beta_i}-m_\beta)+\Delta_{1,i}$ into \eqref{eq:emp_cov_beta_2} yields

\begin{align*}
  C_{f(m_\beta)}  = \frac{1}{K} \sum_{i=1}^K \left[ J_f(m_\beta)(m_{\beta_i} - m_\beta)\right] \left[ J_f(m_\beta)(m_{\beta_i} - m_\beta) \right]^\top + \Delta_2
   = M_{N_{\tilde f^{m_\beta}_\sharp\beta}} + \Delta_2,
\end{align*}
where
\begin{align*}
  \Delta_2  &=
\frac{1}{K}\sum_{i=1}^K
\left(
[J_f(m_\beta)(m_{\beta_i} - m_\beta)]\Delta_{1,i}^\top
+
\Delta_{1,i}[J_f(m_\beta)(m_{\beta_i} - m_\beta)]^\top
+
\Delta_{1,i}\Delta_{1,i}^\top
\right)\\
&=C_{f(m_\beta)}-M_{N_{\tilde f^{m_\beta}_\sharp\beta}}.
\end{align*}
Since $\|J_f(m_\beta)\|_2$ is bounded, we have
\begin{equation}\label{eq:jacobain delta bound}
    \|J_f(m_\beta)(m_{\beta_i} - m_\beta)\|=O(\delta).
\end{equation}
By Cauchy–Schwarz we get
$$\|\Delta_{1,i}\left[ J_f(m_\beta)(m_{\beta_i} - m_\beta) \right]^\top\|_2\leq\|\Delta_{1,i}\|   \|J_f(m_\beta)(m_{\beta_i}-m_\beta) \|  =   \|\Delta_{1,i}\|   \|J_f(m_\beta)(m_{\beta_i}-m_\beta) \|.$$
This implies \footnote{Here and throughout, $c_1,c_2>0$ denote constants appearing in the asymptotic big-$O$ bounds; that is, if $A=O(r)$ then $|A|\le c r$ for some such constant $c$.}
\begin{equation*}
\|\Delta_{1,i}\|  \|J_f(m_\beta)(m_{\beta_i}-m_\beta) \| 
\le 
c_1 \delta^2 \times c_2 \delta = c_1 c_2 \delta^3.
\end{equation*}
Similarly 
\begin{equation*}
\|\left[ J_f(m_\beta)(m_{\beta_i} - m_\beta) \right]\Delta_{1,i}^\top\|_2   
\le 
 c_1 c_2 \delta^3,
\end{equation*}
and 
\begin{equation*}
    \|\Delta_{1,i}\Delta_{1,i}^\top\|_2\leq c_1^2\delta^4.
\end{equation*}
Consequently,
$\|\Delta_2\|_2 \leq 2c_1c_2\delta^3 + c_1^2\delta^4$, which has leading order $O(\delta^3)$. 
\item
We now compare the pseudo-inverses of $C_{f(m_\beta)}$ and $M_{N_{\tilde f^{m_\beta}_\sharp\beta}}$. Denote by
$
    \Delta_3 = C^\dagger _{f(m_\beta)} - M_{N_{\tilde f^{m_\beta}_\sharp\beta}}^\dagger.
$
Recall from \ref{part:covariance_euclidean} the bound
$
\|C^\dagger _{f(m_\beta)}\|_2 = O(\delta^{-2}).
$
Moreover,
\begin{equation*}
    \biggr\|\left[ \frac{1}{K} \sum_{i=1}^K \left[ J_f(m_\beta)(m_{\beta_i} - m_\beta) \right] \left[ J_f(m_\beta)(m_{\beta_i} - m_\beta) \right]^\top\right]^\dagger \biggr\|_2 = \|M^\dagger_{N_{\tilde f^{m_\beta}_\sharp\beta}}\|_2 = O(\frac{1}{\delta^2}).
\end{equation*}
Apply \Cref{pseudo inverse lemma} to obtain
\begin{equation}\label{eq:Delta3_def_app}
    \|\Delta_3\|_2 \leq \mu \max \biggr(\|C^\dagger _{f(m_\beta)}\|^2_2,  \bigr\|  M^\dagger_{N_{\tilde f^{m_\beta}_\sharp\beta}} \bigr\|^2_2 \biggr)\|\Delta_2\|_2,
\end{equation}
for some constant $\mu \leq 3$. It follows that \eqref{eq:Delta3_def_app} has order $O(\delta^{-1})$. \qedhere
\end{enumerate}

\end{proof}

\subsection{Proof of \texorpdfstring{\Cref{linear approx of f approx euclidean mahalanobis}}{}}
\begin{proof}[Proof of \Cref{linear approx of f approx euclidean mahalanobis}] \label{Proof of:linear approx of f approx euclidean mahalanobis}

\Cref{lemma: WM-inner-covariance} implies
\begin{equation*}
\big\langle g^{\tilde f^{m_\beta}_\sharp\alpha}_{\tilde f^{m_\beta}_\sharp\beta}, 
F^\dagger_{N_{\tilde f^{m_\beta}_\sharp\beta}}
\big(g^{\tilde f^{m_\beta}_\sharp\alpha}_{\tilde f^{m_\beta}_\sharp\beta}\big)
\big\rangle_{\tilde f^{m_\beta}_\sharp\beta} = (J_f(m_\beta)(m_\alpha-m_\beta))^\top M^\dagger_{N_{\tilde f^{m_\beta}_\sharp\beta}}(J_f(m_\beta)(m_\alpha-m_\beta)).
\end{equation*}
Let
$
    \Delta_3 = C^\dagger _{f(m_\beta)} - M_{N_{\tilde f^{m_\beta}_\sharp\beta}}^\dagger,
$
then
\begin{equation}\label{innerproduct1}
        \langle g^{\tilde{f}^{m_\beta}_\sharp\alpha}_{\tilde{f}^{m_\beta}_\sharp\beta}, F^\dagger_{N_{\tilde{f}^{m_\beta}_\sharp\beta}}(g^{\tilde{f}^{m_\beta}_\sharp\alpha}_{\tilde{f}^{m_\beta}_\sharp\beta}) \rangle_{\tilde{f}^{m_\beta}_\sharp\beta} = \left[ J_f(m_\beta)(m_{\alpha} - m_\beta) \right]^\top (C^\dagger _{f(m_\beta)} - \Delta_3)\left[ J_f(m_\beta)(m_{\alpha} - m_\beta) \right].
\end{equation}
Define the Taylor remainder
$
    \Delta_{4} =  f(m_\alpha) - f(m_\beta) - J_f(m_\beta)(m_{\alpha} - m_\beta),
$
where $\|\Delta_{4}\|$ is $O(\|m_{\alpha} - m_\beta\|^2)$. Substituting $J_f(m_\beta)(m_\alpha-m_\beta)=\big(f(m_\alpha)-f(m_\beta)\big)-\Delta_4$ into \eqref{innerproduct1}, we get

\begin{align*}
        \langle g^{\tilde{f}^{m_\beta}_\sharp\alpha}_{\tilde{f}^{m_\beta}_\sharp\beta}, F^\dagger_{N_{\tilde{f}^{m_\beta}_\sharp\beta}}(g^{\tilde{f}^{m_\beta}_\sharp\alpha}_{\tilde{f}^{m_\beta}_\sharp\beta}) \rangle_{\tilde{f}^{m_\beta}_\sharp\beta} &= \left[ [ f(m_\alpha) - f(m_\beta) ] - \Delta_4\right]^\top (C^\dagger _{f(m_\beta)} - \Delta_3)\left[ [  f(m_\alpha) - f(m_\beta)  ] - \Delta_4 \right].
\end{align*}
Denote by
\begin{equation*}
    \Delta_5 = [ f(m_\alpha) - f(m_\beta) ]^\top C^\dagger _{f(m_\beta)}[ f(m_\alpha) - f(m_\beta) ] -  \langle g^{\tilde{f}^{m_\beta}_\sharp\alpha}_{\tilde{f}^{m_\beta}_\sharp\beta}, F^\dagger_{N_{\tilde{f}^{m_\beta}_\sharp\beta}}(g^{\tilde{f}^{m_\beta}_\sharp\alpha}_{\tilde{f}^{m_\beta}_\sharp\beta}) \rangle_{\tilde{f}^{m_\beta}_\sharp\beta},
\end{equation*}
which consists of the cross terms involving $\Delta_4$ and the perturbation
term involving $\Delta_3$, namely,
\begin{align*}
    \Delta_5 &= ([ f(m_\alpha) - f(m_\beta) ]^\top C^\dagger _{f(m_\beta)} \Delta_4 + \Delta_4^\top C^\dagger _{f(m_\beta)} [ f(m_\alpha) - f(m_\beta) ])- \Delta_4^\top C^\dagger _{f(m_\beta)} \Delta_4\\
    &+[ f(m_\alpha) - f(m_\beta) ]^\top \Delta_3 [ f(m_\alpha) - f(m_\beta) ]
        -[ f(m_\alpha) - f(m_\beta) ]^\top \Delta_3 \Delta_4\\
        &- \Delta_4^\top\Delta_3 [ f(m_\alpha) - f(m_\beta) ]
        + \Delta_4^\top \Delta_3 \Delta_4.
\end{align*}
Using \Cref{rmk:cov inverse Eulidean}, $\|C^\dagger _{f(m_\beta)}\|_2 = O(\delta^{-2})$ and
$\|\Delta_3\|_2=O(\delta^{-1})$. We also have $\|\Delta_4\|=O(\|m_\alpha-m_\beta\|^2)$ and
$\|f(m_\alpha)-f(m_\beta)\|=O(\|m_\alpha-m_\beta\|)$. Compute each term of $\Delta_5$ as follows \footnote{Here and throughout, $c_3,\ldots, c_6>0$ denote constants appearing in the asymptotic big-$O$ bounds; that is, if $A=O(r)$ then $|A|\le c r$ for some such constant $c$.}
\begin{align*}
    |[ f(m_\alpha) - f(m_\beta) ]^\top C^\dagger _{f(m_\beta)} \Delta_4| &\leq \|[ f(m_\alpha) - f(m_\beta) ]\|\|C^\dagger _{f(m_\beta)}\|_2 \|\Delta_4\|\\
    &\leq c_3\| m_\alpha-m_\beta \|\times \frac{c_4}{\delta^2} \times c_5\| m_\alpha-m_\beta \|^2\\
    &=c_3\frac{c_4}{\delta^2}c_5\| m_\alpha-m_\beta \|^3.
\end{align*}
Similarly
\begin{equation*}
   |\Delta_4^\top C^\dagger _{f(m_\beta)} [ f(m_\alpha) - f(m_\beta) ])|\leq c_3\frac{c_4}{\delta^2}c_5\| m_\alpha-m_\beta \|^3,
\end{equation*}
\begin{equation*}
   |\Delta_4^\top C^\dagger _{f(m_\beta)} \Delta_4|_2\leq \frac{c_4}{\delta^2}c_5^2\| m_\alpha-m_\beta \|^4,
\end{equation*}
\begin{align*}
| [f(m_\alpha) - f(m_\beta)]^\top \Delta_3 [f(m_\alpha) - f(m_\beta)] |\leq c_3^2\|m_\alpha-m_\beta\|^2 \frac{c_6}{\delta},
\end{align*}

\begin{equation*}
|[ f(m_\alpha) - f(m_\beta) ]^\top \Delta_3 \Delta_4|\leq c_3c_5\| m_\alpha-m_\beta \|^3 \frac{c_6}{\delta},
\end{equation*}

\begin{equation*}
|\Delta_4^\top\Delta_3 [ f(m_\alpha) - f(m_\beta) ]|\leq c_3c_5\| m_\alpha-m_\beta \|^3 \frac{c_6}{\delta},
\end{equation*}

\begin{equation*}
|\Delta_4^\top \Delta_3 \Delta_4|
  \leq c_5^2\| m_\alpha-m_\beta \|^4\frac{c_6}{\delta}.
\end{equation*}
Therefore we obtain the bound
\begin{align*}
    |\Delta_5|\leq   & 2c_3\frac{c_4}{\delta^2}c_5\| m_\alpha-m_\beta \|^3 + \frac{c_4}{\delta^2}c_5^2\| m_\alpha-m_\beta \|^4 + c_3^2 \|m_\alpha-m_\beta\|^2 \frac{c_6}{\delta} \\
   & + 2c_3c_5\|m_\alpha-m_\beta\|^3 \frac{c_6}{\delta}+ c_5^2\| m_\alpha-m_\beta \|^4\frac{c_6}{\delta},
\end{align*}
which has leading order $O(\frac{\|m_{\alpha} - m_\beta\|^2}{\delta})$. We conclude that 
\begin{equation}\label{inner product 1}
     \langle g^{\tilde{f}^{m_\beta}_\sharp\alpha}_{\tilde{f}^{m_\beta}_\sharp\beta}, F^\dagger_{N_{\tilde{f}^{m_\beta}_\sharp\beta}}(g^{\tilde{f}^{m_\beta}_\sharp\alpha}_{\tilde{f}^{m_\beta}_\sharp\beta}) \rangle_{\tilde{f}^{m_\beta}_\sharp\beta} =[ f(m_\alpha) - f(m_\beta) ]^\top C^\dagger _{f(m_\beta)}[ f(m_\alpha) - f(m_\beta) ] + O(\frac{\|m_{\alpha} - m_\beta\|^2}{\delta}).
\end{equation}

Since the Wasserstein Mahalanobis distance uses two inner products (one with base measure $f_\sharp\beta$ and one with $f_\sharp\alpha$), we can simply repeat the same argument by linearizing $f$ at $m_\alpha$. This yields
\begin{equation}\label{inner product 2}
    \langle g_{\tilde{f}^{m_\alpha}_{\sharp} \alpha}^{\tilde{f}^{m_\alpha}_{\sharp} \beta},   F^\dagger_{N_{\tilde{f}^{m_\alpha}_{\sharp} \alpha}} ( g_{\tilde{f}^{m_\alpha}_{\sharp} \alpha}^{\tilde{f}^{m_\alpha}_{\sharp} \beta} ) \rangle_{\tilde{f}^{m_\alpha}_{\sharp} \alpha} = \left[ f(m_\beta) - f(m_\alpha) \right]^\top C_{f(m_\alpha)}^\dagger \left[ f(m_\beta) - f(m_\alpha) \right] +O(\frac{\|m_\beta -m_{\alpha}\|^2}{\delta}).
\end{equation}
Averaging \eqref{inner product 1} and \eqref{inner product 2} gives the desired approximation
 \begin{align*}
          &d_M^2(f(m_\alpha),f(m_\beta)) \\
          =& \frac{1}{2}\left(\langle g^{\tilde{f}^{m_\beta}_\sharp\alpha}_{\tilde{f}^{m_\beta}_\sharp\beta}, F^\dagger_{N_{\tilde{f}^{m_\beta}_\sharp\beta}}(g^{\tilde{f}^{m_\beta}_\sharp\alpha}_{\tilde{f}^{m_\beta}_\sharp\beta}) \rangle_{\tilde{f}^{m_\beta}_\sharp\beta} + \langle g^{\tilde{f}^{m_\alpha}_{\sharp} \beta}_{\tilde{f}^{m_\alpha}_{\sharp} \alpha},   F^\dagger_{N_{\tilde{f}^{m_\alpha}_{\sharp} \alpha}} ( g^{\tilde{f}^{m_\alpha}_{\sharp} \beta}_{\tilde{f}^{m_\alpha}_{\sharp} \alpha} ) \rangle_{\tilde{f}^{m_\alpha}_{\sharp} \alpha}\right) + O(\frac{\|m_{\alpha} - m_\beta\|^2}{\delta}). \qedhere
 \end{align*}
\end{proof}

\subsection{Proof of \texorpdfstring{\Cref{lem:BC_comparison}}{}}\label{Proof:lem_BC_comparison}

\begin{proof}[Proof of \Cref{lem:BC_comparison}]
    
By definition,
\begin{equation*}
B_{kl}
=
\frac1K\sum_{i=1}^K
\Big\langle g^{f_\sharp\beta_i}_{f_\sharp\beta},\ g^{f_\sharp\beta_l}_{f_\sharp\beta}\Big\rangle_{f_\sharp\beta}
\Big\langle g^{f_\sharp\beta_i}_{f_\sharp\beta},\ g^{f_\sharp\beta_k}_{f_\sharp\beta}\Big\rangle_{f_\sharp\beta},
\end{equation*}
and
\begin{equation*}
C_{kl}
=
\frac1K\sum_{i=1}^K
\Big\langle g^{\tilde f^{m_\beta}_\sharp\beta_i}_{\tilde f^{m_\beta}_\sharp\beta},\ g^{\tilde f^{m_\beta}_\sharp\beta_l}_{\tilde f^{m_\beta}_\sharp\beta}\Big\rangle_{\tilde f^{m_\beta}_\sharp\beta}
\Big\langle g^{\tilde f^{m_\beta}_\sharp\beta_i}_{\tilde f^{m_\beta}_\sharp\beta},\ g^{\tilde f^{m_\beta}_\sharp\beta_k}_{\tilde f^{m_\beta}_\sharp\beta}\Big\rangle_{\tilde f^{m_\beta}_\sharp\beta}.
\end{equation*}
The proof is organized into three parts.
\begin{enumerate}[label=\emph{(\roman*)}, leftmargin=0pt, itemindent=*, listparindent=\parindent]
\item \label{part:Cdagger_bound} \emph{Bound $\|C^\dagger\|_2$.} 

Since $\tilde{f}^{m_\beta}$ is affine, we have 
$
    g^{\tilde f^{m_\beta}_\sharp\beta_i}_{\tilde f^{m_\beta}_\sharp\beta}(x) = J_f(m_\beta)(m_{\beta_i} - m_\beta).
$
We can then write
$
  C =
  \frac{1}{K} ( V^\top V)^2,
$
where
\begin{equation*}
    V = [J_f(m_\beta)(m_{\beta_1} - m_\beta), \cdots ,J_f(m_\beta)(m_{\beta_K} - m_\beta)] \in \mathbb{R}^{n \times K}.
\end{equation*}
By \Cref{lemma:bounds for C}, we have
\begin{equation}\label{eq:Cdagger_bound_short}
    \|C^\dagger\|_2
    \leq
    \frac{(n+2)^2}
    {K\delta^4
    \bigl[\sigma_{\min}(J_f(m_\beta))\bigr]^4}
    =
    O(\delta^{-4}).
\end{equation}

\item \label{part:B_Cdagger_bound} \emph{Bound $\|B-C\|_2$.} 

By definition of $B$ and $C$ we obtain
\begin{align*}
(B-C)_{kl}
&=\frac1K\sum_{i=1}^K\Bigg[
\Big\langle g^{f_\sharp\beta_i}_{f_\sharp\beta}, g^{f_\sharp\beta_l}_{f_\sharp\beta}\Big\rangle_{f_\sharp\beta}
\Big\langle g^{f_\sharp\beta_i}_{f_\sharp\beta}, g^{f_\sharp\beta_k}_{f_\sharp\beta}\Big\rangle_{f_\sharp\beta} -
\Big\langle g^{\tilde f^{m_\beta}_\sharp\beta_i}_{\tilde f^{m_\beta}_\sharp\beta}, g^{\tilde f^{m_\beta}_\sharp\beta_l}_{\tilde f^{m_\beta}_\sharp\beta}\Big\rangle_{\tilde f^{m_\beta}_\sharp\beta}
\Big\langle g^{\tilde f^{m_\beta}_\sharp\beta_i}_{\tilde f^{m_\beta}_\sharp\beta}, g^{\tilde f^{m_\beta}_\sharp\beta_k}_{\tilde f^{m_\beta}_\sharp\beta}\Big\rangle_{\tilde f^{m_\beta}_\sharp\beta}
\Bigg].
\end{align*}
Applying \Cref{lemma:bounds for B-C},
\begin{equation}\label{eq:bc_diff_kl}
\begin{aligned}
   (B-C)_{kl}
&\leq\frac1K\sum_{i=1}^K
\Bigg(\Big[
\|J_f(m_\beta)\|_2\delta
\big(
\|\xi_i\|_{L^2(f_\sharp\beta)}
+
\|\xi_l\|_{L^2(f_\sharp\beta)}
\big)
+
\|\xi_i\|_{L^2(f_\sharp\beta)}
\|\xi_l\|_{L^2(f_\sharp\beta)}
\Big] \\
&\times
\Big(
\|J_f(m_\beta)\|_2\delta
+
\|\xi_i\|_{L^2(f_\sharp\beta)}
\Big)
\Big(
\|J_f(m_\beta)\|_2\delta
+
\|\xi_k\|_{L^2(f_\sharp\beta)}
\Big)\\
&
+
\|J_f(m_\beta)\|_2^2\delta^2
\Big[
\|J_f(m_\beta)\|_2\delta
\big(
\|\xi_i\|_{L^2(f_\sharp\beta)}
+
\|\xi_k\|_{L^2(f_\sharp\beta)}
\big)
+
\|\xi_i\|_{L^2(f_\sharp\beta)}
\|\xi_k\|_{L^2(f_\sharp\beta)}
\Big]\Bigg),
\end{aligned}
\end{equation}
where $\xi_i(y) = g^{f_\sharp \beta_i}_{f_\sharp \beta}(y) -g^{\tilde f^{m_\beta}_\sharp\beta_i}_{\tilde f^{m_\beta}_\sharp\beta} (y).$ By Assumption~\ref{displacement function assumption},
\begin{equation*}
    \|\xi_i\|_{L^2(f_\sharp\beta)}
    =
    O(\delta^{p/2}),\quad \text{for all } i,
     \text{ and some } p\geq 2.
\end{equation*}
It follows that 
$
    |(B-C)_{kl}|
    =
    O(\delta^{3+p/2}).
$
Therefore,
\begin{equation}\label{eq:BC_bound_short}
    \|B-C\|_2
    \leq
    \|B-C\|_F
    \leq
    K\max_{k,l}|(B-C)_{kl}|
    =
    O(\delta^{3+p/2}).
\end{equation}

\item \label{part:Bdagger_bound} \emph{ Bound $\|B^\dagger\|_2$.} 

We now use the previous two bounds to find a bound for $\|B^\dagger\|_2$. From
\eqref{eq:Cdagger_bound_short} and \eqref{eq:BC_bound_short}, we have
\begin{equation*}
\|B-C\|_2\|C^\dagger\|_2 
=
O \left(\frac{\delta^{3+p/2}}{\delta^4}\right)
=
O \left(\delta^{p/2-1}\right).
\end{equation*}
If $p>2$ and when $\delta$ sufficiently small, there always exists some $0<\epsilon<1$ such that
\begin{equation}\label{eq:BC_relative_smallness}
\|B-C\|_2\le \epsilon \sigma_{\min}^+(C).
\end{equation}
In the case $p=2$, there exists $c_\xi >0$ such that
\begin{equation*}
    \|\xi_j\|_{L^2(f_\sharp\beta)}
    \leq
    \sqrt{c_\xi} \delta,
    \qquad j=1,\dots,K.
\end{equation*}
Plug this into \eqref{eq:bc_diff_kl} and simplify to obtain
\begin{align*}
(B-C)_{kl}
&\le
\Big(2\|J_f(m_\beta)\|_2\sqrt{c_\xi} \delta^2+c_\xi\delta^2\Big)
\Big(\|J_f(m_\beta)\|_2\delta+\sqrt{c_\xi} \delta\Big)^2
+
\|J_f(m_\beta)\|_2^2\delta^2
\Big(2\|J_f(m_\beta)\|_2\sqrt{c_\xi} \delta^2+c_\xi\delta^2\Big)\\
&=
\Big(
4\|J_f(m_\beta)\|_2^3\sqrt{c_\xi}
+
6\|J_f(m_\beta)\|_2^2c_\xi
+
4\|J_f(m_\beta)\|_2c_\xi^{3/2}
+
c_\xi^2
\Big)\delta^4.
\end{align*}
It follows that,
\begin{align*}
\|B-C\|_2
&\le \|B-C\|_F
\le K \max_{k,l}|(B-C)_{kl}|\\
&\le
K\Big(
4\|J_f(m_\beta)\|_2^3\sqrt{c_\xi}
+
6\|J_f(m_\beta)\|_2^2c_\xi
+
4\|J_f(m_\beta)\|_2c_\xi^{3/2}
+
c_\xi^2
\Big)\delta^4\\
&=
K\sqrt{c_\xi}\Big(
4\|J_f(m_\beta)\|_2^3
+
6\|J_f(m_\beta)\|_2^2\sqrt{c_\xi}
+
4\|J_f(m_\beta)\|_2c_\xi
+
c_\xi^{3/2}
\Big)\delta^4.
\end{align*}
Together with \eqref{eq:Cdagger_bound_short},
\begin{align}\label{eq:p=2_explicite_form}
\|B-C\|_2\|C^\dagger\|_2
\le
\frac{(n+2)^2}{
\bigl[\sigma_{\min}(J_f(m_\beta))\bigr]^4}
\sqrt{c_\xi}
\Big(
4\|J_f(m_\beta)\|_2^3
+
6\|J_f(m_\beta)\|_2^2\sqrt{c_\xi}
+
4\|J_f(m_\beta)\|_2c_\xi
+
c_\xi^{3/2}
\Big).
\end{align}
By \Cref{assuption:dispacement_local},
\begin{equation*}
0<c_\xi<
\min\left\{
1,
\left[
\frac{\|J_f(m_\beta)\|_2^4}
{15(n+2)^2\kappa_\beta^4}
\right]^2,
\left[
\frac{\|J_f(m_\beta)\|_2}
{15(n+2)^2\kappa_\beta^4}
\right]^2
\right\},
\end{equation*}
where $\kappa_\beta=
\frac{\|J_f(m_\beta)\|_2}
{\sigma_{\min}(J_f(m_\beta))}$. By the definition of $\kappa_\beta$
\begin{equation*}
\bigl[\sigma_{\min}(J_f(m_\beta))\bigr]^4
=
\frac{\|J_f(m_\beta)\|_2^4}{\kappa_\beta^4}.
\end{equation*}
Therefore,
\begin{align*}
&
\frac{\bigl[\sigma_{\min}(J_f(m_\beta))\bigr]^4}
{(n+2)^2
\Big(
4\|J_f(m_\beta)\|_2^3
+
6\|J_f(m_\beta)\|_2^2
+
4\|J_f(m_\beta)\|_2
+
1
\Big)}
\\
=&
\frac{\|J_f(m_\beta)\|_2^4}
{(n+2)^2\kappa_\beta^4
\Big(
4\|J_f(m_\beta)\|_2^3
+
6\|J_f(m_\beta)\|_2^2
+
4\|J_f(m_\beta)\|_2
+
1
\Big)}.
\end{align*}
If $\|J_f(m_\beta)\|_2\ge 1$, then
\begin{equation*}
4\|J_f(m_\beta)\|_2^3
+
6\|J_f(m_\beta)\|_2^2
+
4\|J_f(m_\beta)\|_2
+
1
\le
15\|J_f(m_\beta)\|_2^3.
\end{equation*}
Hence
\begin{equation*}
    \frac{\bigl[\sigma_{\min}(J_f(m_\beta))\bigr]^4}
{(n+2)^2
\Big(
4\|J_f(m_\beta)\|_2^3
+
6\|J_f(m_\beta)\|_2^2
+
4\|J_f(m_\beta)\|_2
+
1
\Big)}
\ge
\frac{\|J_f(m_\beta)\|_2}
{15(n+2)^2\kappa_\beta^4}.
\end{equation*}
If $0<\|J_f(m_\beta)\|_2\le 1$, then
\begin{equation*}
    4\|J_f(m_\beta)\|_2^3
+
6\|J_f(m_\beta)\|_2^2
+
4\|J_f(m_\beta)\|_2
+
1
\le
15,
\end{equation*}
and therefore
\begin{equation*}
    \frac{\bigl[\sigma_{\min}(J_f(m_\beta))\bigr]^4}
{(n+2)^2
\Big(
4\|J_f(m_\beta)\|_2^3
+
6\|J_f(m_\beta)\|_2^2
+
4\|J_f(m_\beta)\|_2
+
1
\Big)}
\ge
\frac{\|J_f(m_\beta)\|_2^4}
{15(n+2)^2\kappa_\beta^4}.
\end{equation*}
Combining the two cases, the assumption on $c_\xi$ implies
\begin{equation*}
0<c_\xi<
\min\left\{
1,\
\left[
\frac{\bigl[\sigma_{\min}(J_f(m_\beta))\bigr]^4}
{(n+2)^2
\Big(
4\|J_f(m_\beta)\|_2^3
+
6\|J_f(m_\beta)\|_2^2
+
4\|J_f(m_\beta)\|_2
+
1
\Big)}
\right]^2
\right\}.
\end{equation*}
Using the above bound on $c_\xi$ in \eqref{eq:p=2_explicite_form}, together with $c_\xi<1$, the right-hand side \eqref{eq:p=2_explicite_form} is strictly smaller than 1, then there exists some $0<\epsilon<1$ such that \eqref{eq:BC_relative_smallness} holds. 

Thus, in both cases ($p >2$, $p=2$), we have 
\begin{equation*}
    \|B-C\|_2 \le \epsilon\sigma_{\min}^+(C),\qquad 0<\epsilon<1.
\end{equation*}
Since $B$ and $C$ are both symmetric PSD matrices by construction, their singular values coincide with their eigenvalues. By Weyl's inequality (\Cref{Weyl's inequality}),
\begin{equation*}
    \sigma_r(B)\ge \sigma_r(C)-\|B-C\|_2 \ge (1-\epsilon)\sigma_r(C).
\end{equation*}
We also have $\|B^\dagger\|_2 = 1/\sigma_{\min}^+(B)=1/\sigma_r(B)$ \footnote{Since the perturbation $\xi_i$ is small when $\delta$ is small, we assume that the matrix $B$ has the same rank as the matrix $C$. In \Cref{lemma:bounds for C} $\rank(C) = r \leq n$ was assumed.}. Therefore, using \eqref{eq:Cdagger_bound_short},
\begin{equation}\label{eq:B_dagger_bound}
    \|B^\dagger\|_2 \le \frac{1}{(1-\epsilon)\sigma_{\min}^+(C)}    \leq
   \frac{(n+2)^{2}}
        {(1-\epsilon)K\delta^{4}
         \bigl[\sigma_{\min}\bigl(J_f(m_\beta)\bigr)\bigr]^{4}},
\end{equation}
which has order $O(\delta^{-4})$.    
\end{enumerate}
\end{proof}

\subsection{Proof of \texorpdfstring{\Cref{operator differece lemma}}{}}\label{Proof:operator differece lemma}

\begin{proof}[Proof of \Cref{operator differece lemma}]
The goal is to compare the pseudo-inverse operator built from the nonlinear
displacement functions with the pseudo-inverse operator built from the
linearized displacement functions. In the calculations, three sources of error will appear:
\begin{enumerate}[label=\emph{(\roman*)}, leftmargin=0pt, itemindent=*, listparindent=\parindent]
\item coefficient error
\item displacement error
\item change of measure error
\end{enumerate}

By \Cref{Lemma:operater_defference_explicit},
\begin{align*}
F^\dagger_{N_{f_\sharp \beta}}  - \tilde F^\dagger_{N_{\tilde f^{m_\beta}_\sharp\beta}}
&=
\sum_{k=1}^K \sum_{l=1}^K \Big((B^\dagger)_{kl}-(C^\dagger)_{kl}\Big) 
g^{\tilde f^{m_\beta}_\sharp\beta_k}_{\tilde f^{m_\beta}_\sharp\beta} 
\Big\langle
g^{\tilde f^{m_\beta}_\sharp\beta_l}_{\tilde f^{m_\beta}_\sharp\beta},\cdot
\Big\rangle_{f_\sharp\beta}
\\
&+
\sum_{k=1}^K \sum_{l=1}^K (B^\dagger)_{kl}\Bigg[
g^{\tilde f^{m_\beta}_\sharp\beta_k}_{\tilde f^{m_\beta}_\sharp\beta} 
\big\langle \xi_l,\cdot\big\rangle_{f_\sharp\beta}
+
\xi_k 
\Big\langle
g^{\tilde f^{m_\beta}_\sharp\beta_l}_{\tilde f^{m_\beta}_\sharp\beta},\cdot
\Big\rangle_{f_\sharp\beta}
+
\xi_k \big\langle \xi_l,\cdot\big\rangle_{f_\sharp\beta}
\Bigg]
\\
&+
\sum_{k=1}^K \sum_{l=1}^K (C^\dagger)_{kl}\Bigg[
g^{\tilde f^{m_\beta}_\sharp\beta_k}_{\tilde f^{m_\beta}_\sharp\beta} 
\Big\langle
g^{\tilde f^{m_\beta}_\sharp\beta_l}_{\tilde f^{m_\beta}_\sharp\beta},\cdot
\Big\rangle_{f_\sharp\beta}
-
\Big(U g^{\tilde f^{m_\beta}_\sharp\beta_k}_{\tilde f^{m_\beta}_\sharp\beta}\Big) 
\Big\langle
\Big(U g^{\tilde f^{m_\beta}_\sharp\beta_l}_{\tilde f^{m_\beta}_\sharp\beta}\Big),\cdot
\Big\rangle_{f_\sharp\beta}
\Bigg],
\end{align*}
where $\xi_k=g^{f_\sharp \beta_k}_{f_\sharp \beta} - g^{\tilde f^{m_\beta}_\sharp\beta_k}_{\tilde f^{m_\beta}_\sharp\beta} $. $B$ and $C$ are defined in \Cref{lem:BC_comparison}. We decompose
\begin{equation}\label{eq:Delta_decomposition}
F^\dagger_{N_{f_\sharp \beta}}  - \tilde F^\dagger_{N_{\tilde f^{m_\beta}_\sharp\beta}}
=
\Delta^{(1)}+\Delta^{(2)}+\Delta^{(3)},
\end{equation}
where
\begin{equation*}
\Delta^{(1)}
=
\sum_{k=1}^K\sum_{l=1}^K
\bigl((B^\dagger)_{kl}-(C^\dagger)_{kl}\bigr)
g^{\tilde f^{m_\beta}_\sharp\beta_k}_{\tilde f^{m_\beta}_\sharp\beta}
\left\langle
g^{\tilde f^{m_\beta}_\sharp\beta_l}_{\tilde f^{m_\beta}_\sharp\beta},
\cdot
\right\rangle_{f_\sharp\beta},
\end{equation*}
\begin{equation*}
\begin{aligned}
\Delta^{(2)}
=
\sum_{k=1}^K\sum_{l=1}^K
(B^\dagger)_{kl}
\Big[
g^{\tilde f^{m_\beta}_\sharp\beta_k}_{\tilde f^{m_\beta}_\sharp\beta}
\left\langle
\xi_l,\cdot
\right\rangle_{f_\sharp\beta}+
\xi_k
\left\langle
g^{\tilde f^{m_\beta}_\sharp\beta_l}_{\tilde f^{m_\beta}_\sharp\beta},
\cdot
\right\rangle_{f_\sharp\beta}
+
\xi_k
\left\langle
\xi_l,\cdot
\right\rangle_{f_\sharp\beta}
\Big],
\end{aligned}
\end{equation*}
and
\begin{equation*}
\begin{aligned}
\Delta^{(3)}
=
\sum_{k=1}^K\sum_{l=1}^K
(C^\dagger)_{kl}
\Big[
g^{\tilde f^{m_\beta}_\sharp\beta_k}_{\tilde f^{m_\beta}_\sharp\beta}
\left\langle
g^{\tilde f^{m_\beta}_\sharp\beta_l}_{\tilde f^{m_\beta}_\sharp\beta},
\cdot
\right\rangle_{f_\sharp\beta}
-
\left(
U g^{\tilde f^{m_\beta}_\sharp\beta_k}_{\tilde f^{m_\beta}_\sharp\beta}
\right)
\left\langle
U g^{\tilde f^{m_\beta}_\sharp\beta_l}_{\tilde f^{m_\beta}_\sharp\beta},
\cdot
\right\rangle_{f_\sharp\beta}
\Big].
\end{aligned}
\end{equation*}
Thus $\Delta^{(1)}$ is the coefficient error, $\Delta^{(2)}$ is the displacement error,
and $\Delta^{(3)}$ is the change of measure error.

We estimate these three terms separately.

\begin{enumerate}[label=\emph{(\roman*)}, leftmargin=0pt, itemindent=*, listparindent=\parindent]

\item \label{part:E_bound} \emph{Bound the coefficient error $\Delta^{(1)}$.}

From \Cref{lem:BC_comparison}, $\|B^\dagger\|_2=O(\delta^{-4})$, $\|C^\dagger\|_2=O(\delta^{-4})$ and $\|B-C\|_2=O(\delta^{3+p/2})$. Applying \Cref{pseudo inverse lemma}, we obtain
\begin{equation}\label{eq:E_Bound}
    \|B^\dagger - C^\dagger\|_2=O(\delta^{p/2-5}).
\end{equation} 
Now let $\phi\in L^2(f_\sharp\beta)$ satisfy $\|\phi\|_{L^2(f_\sharp\beta)}=1$. Since 
  \begin{equation*}
g^{\tilde f^{m_\beta}_\sharp\beta_i}_{\tilde f^{m_\beta}_\sharp\beta}
=
J_f(m_\beta)(m_{\beta_i}-m_\beta)
\end{equation*}
we have
\begin{equation}\label{eq:g_phi_bound}
\Big|\Big\langle g^{\tilde f^{m_\beta}_\sharp\beta_i}_{\tilde f^{m_\beta}_\sharp\beta},\phi\Big\rangle_{f_\sharp\beta}\Big|
\le
\Big\|g^{\tilde f^{m_\beta}_\sharp\beta_i}_{\tilde f^{m_\beta}_\sharp\beta}\Big\|_{L^2(f_\sharp\beta)}\|\phi\|_{L^2(f_\sharp\beta)}
\le \|J_f(m_\beta)\|_2\delta,
\qquad
\Big\|g^{\tilde f^{m_\beta}_\sharp\beta_i}_{\tilde f^{m_\beta}_\sharp\beta}\Big\|_{L^2(f_\sharp\beta)}\le \|J_f(m_\beta)\|_2\delta.
\end{equation}
Therefore,
\begin{equation*}
\begin{aligned}
\|\Delta^{(1)}\phi\|_{L^2(f_\sharp\beta)}
&\le
\sum_{k=1}^K \sum_{l=1}^K |(B^\dagger)_{kl}-(C^\dagger)_{kl}| 
\Big\|g^{\tilde f^{m_\beta}_\sharp\beta_k}_{\tilde f^{m_\beta}_\sharp\beta}\Big\|_{L^2(f_\sharp\beta)}
\cdot
\Big|\Big\langle
g^{\tilde f^{m_\beta}_\sharp\beta_l}_{\tilde f^{m_\beta}_\sharp\beta},\phi
\Big\rangle_{f_\sharp\beta}\Big|
\\
&\le
(\|J_f(m_\beta)\|_2\delta)^2 \sum_{k=1}^K \sum_{l=1}^K |(B^\dagger)_{kl}-(C^\dagger)_{kl}|.
\end{aligned}
\end{equation*}
Since $\sum_{k,l}|(B^\dagger)_{kl}-(C^\dagger)_{kl}|\le K\|B^\dagger-C^\dagger\|_F\le K^\frac{3}{2}\|B^\dagger-C^\dagger\|_2$, we get
\begin{equation*}
\|\Delta^{(1)}\phi\|_{L^2(f_\sharp\beta)}
\le
K^\frac{3}{2}(\|J_f(m_\beta)\|_2\delta)^2\|B^\dagger-C^\dagger\|_2.
\end{equation*}
Using \eqref{eq:E_Bound}, this gives
\begin{equation}\label{eq:Delta1_order}
\|\Delta^{(1)}\|_{op} = O(\delta^{p/2-3}).
\end{equation}

\item \label{part:term2_bound} \emph{Bound the displacement error $\Delta^{(2)}$.}

Again, let $\|\phi\|_{L^2(f_\sharp\beta)}=1$. By \Cref{displacement function assumption}, $\|\xi_i\|_{L^2(f_\sharp \beta)}^2 \le  M_\xi, \forall i$, where $M_\xi=O(\delta^p)$ with $p\geq2$. Then
\begin{equation*}
\Big|\langle\xi_i,\phi\rangle_{f_\sharp\beta}\Big|
\le
\|\xi_i\|_{L^2(f_\sharp\beta)}\|\phi\|_{L^2(f_\sharp\beta)}
\le \sqrt{M_\xi}
\end{equation*}
and together with \eqref{eq:g_phi_bound}, we have 
\begin{equation*}
\begin{aligned}
&\Big\|
g^{\tilde f^{m_\beta}_\sharp\beta_k}_{\tilde f^{m_\beta}_\sharp\beta} 
\big\langle \xi_l,\phi\big\rangle_{f_\sharp\beta}
+
\xi_k 
\Big\langle
g^{\tilde f^{m_\beta}_\sharp\beta_l}_{\tilde f^{m_\beta}_\sharp\beta},\phi
\Big\rangle_{f_\sharp\beta}
+
\xi_k \big\langle \xi_l,\phi\big\rangle_{f_\sharp\beta}
\Big\|_{L^2(f_\sharp\beta)}
\\
&\le
\Big\|g^{\tilde f^{m_\beta}_\sharp\beta_k}_{\tilde f^{m_\beta}_\sharp\beta}\Big\|_{L^2(f_\sharp\beta)} 
\Big|\big\langle \xi_l,\phi\big\rangle_{f_\sharp\beta}\Big|
+
\|\xi_k\|_{L^2(f_\sharp\beta)} 
\Big|\Big\langle
g^{\tilde f^{m_\beta}_\sharp\beta_l}_{\tilde f^{m_\beta}_\sharp\beta},\phi
\Big\rangle_{f_\sharp\beta}\Big|
+
\|\xi_k\|_{L^2(f_\sharp\beta)} 
\Big|\big\langle \xi_l,\phi\big\rangle_{f_\sharp\beta}\Big|
\\
&\le
2\|J_f(m_\beta)\|_2\delta\sqrt{M_\xi}+M_\xi.
\end{aligned}
\end{equation*}
It follows that
\begin{equation*}
\begin{aligned}
     \|\Delta^{(2)}\phi\|_{L^2(f_\sharp\beta)}=&\|\sum_{k=1}^K \sum_{l=1}^K (B^\dagger)_{kl}\left[ g^{\tilde f^{m_\beta}_\sharp\beta_k}_{\tilde f^{m_\beta}_\sharp\beta} \left\langle \xi_l, \phi \right\rangle_{f_\sharp \beta} \right. \left. + \xi_k \left\langle g^{\tilde f^{m_\beta}_\sharp\beta_l}_{\tilde f^{m_\beta}_\sharp\beta}, \phi \right\rangle_{f_\sharp \beta} + \xi_k \left\langle \xi_l, \phi \right\rangle_{f_\sharp \beta} \right] \|_{L^2(f_\sharp\beta)}\\
     \leq& (2 \|J_f(m_\beta)\|_2 \delta \sqrt{ M_\xi} + M_\xi)\sum_{k=1}^K \sum_{l=1}^K |(B^\dagger)_{kl}| \\
    \leq& (2 \|J_f(m_\beta)\|_2 \delta \sqrt{ M_\xi} + M_\xi) K\|B^\dagger\|_F \\
    \leq& K^\frac{3}{2}\|B^\dagger\|_2(2 \|J_f(m_\beta)\|_2 \delta \sqrt{ M_\xi} + M_\xi).
\end{aligned}
\end{equation*}
Using 
\begin{equation*}
    \|B^\dagger\|_2=O(\delta^{-4}),
\qquad
M_\xi=O(\delta^p),
\qquad
p\ge2,
\end{equation*}
we get
\begin{equation}\label{eq:Delta2_order}
\|\Delta^{(2)}\|_{op}
=
O(\delta^{p/2-3}).
\end{equation}

\item \label{part:term3_bound}\emph{Bound the change of measure error $\Delta^{(3)}$.} 

By definition,
\begin{align*}
U g^{\tilde f^{m_\beta}_\sharp\beta_k}_{\tilde f^{m_\beta}_\sharp\beta}
=
g^{\tilde f^{m_\beta}_\sharp\beta_k}_{\tilde f^{m_\beta}_\sharp\beta}\sqrt w,
\end{align*}
it follows that
\begin{align*}
\Big\|
g^{\tilde f^{m_\beta}_\sharp\beta_k}_{\tilde f^{m_\beta}_\sharp\beta}
-
U g^{\tilde f^{m_\beta}_\sharp\beta_k}_{\tilde f^{m_\beta}_\sharp\beta}
\Big\|_{L^2(f_\sharp\beta)}
&=
\Big\|
g^{\tilde f^{m_\beta}_\sharp\beta_k}_{\tilde f^{m_\beta}_\sharp\beta}(1-\sqrt w)
\Big\|_{L^2(f_\sharp\beta)}
\le
\Big\|g^{\tilde f^{m_\beta}_\sharp\beta_k}_{\tilde f^{m_\beta}_\sharp\beta}\Big\|_{L^2(f_\sharp\beta)}
\|1-\sqrt w\|_{L^2(f_\sharp\beta)}\\
&\le
\|J_f(m_\beta)\|_2\delta \|1-\sqrt w\|_{L^2(f_\sharp\beta)}.
\end{align*}
Adding and subtracting 
$U g^{\tilde f^{m_\beta}_\sharp\beta_k}_{\tilde f^{m_\beta}_\sharp\beta}
\langle g^{\tilde f^{m_\beta}_\sharp\beta_l}_{\tilde f^{m_\beta}_\sharp\beta},\phi\rangle_{f_\sharp\beta}$, we obtain
\begin{equation*}
\begin{aligned}
&\Big\|
g^{\tilde f^{m_\beta}_\sharp\beta_k}_{\tilde f^{m_\beta}_\sharp\beta}
\Big\langle
g^{\tilde f^{m_\beta}_\sharp\beta_l}_{\tilde f^{m_\beta}_\sharp\beta},\phi
\Big\rangle_{f_\sharp\beta}
-
\Big(U g^{\tilde f^{m_\beta}_\sharp\beta_k}_{\tilde f^{m_\beta}_\sharp\beta}\Big)
\Big\langle
\Big(U g^{\tilde f^{m_\beta}_\sharp\beta_l}_{\tilde f^{m_\beta}_\sharp\beta}\Big),\phi
\Big\rangle_{f_\sharp\beta}
\Big\|_{L^2(f_\sharp\beta)}
\\
&\le
\Big\|
g^{\tilde f^{m_\beta}_\sharp\beta_k}_{\tilde f^{m_\beta}_\sharp\beta}
-
U g^{\tilde f^{m_\beta}_\sharp\beta_k}_{\tilde f^{m_\beta}_\sharp\beta}
\Big\|_{L^2(f_\sharp\beta)}
\cdot
\Big|\Big\langle
g^{\tilde f^{m_\beta}_\sharp\beta_l}_{\tilde f^{m_\beta}_\sharp\beta},\phi
\Big\rangle_{f_\sharp\beta}\Big|+
\Big\|U g^{\tilde f^{m_\beta}_\sharp\beta_k}_{\tilde f^{m_\beta}_\sharp\beta}\Big\|_{L^2(f_\sharp\beta)}
\cdot
\Big|\Big\langle
g^{\tilde f^{m_\beta}_\sharp\beta_l}_{\tilde f^{m_\beta}_\sharp\beta}
-
U g^{\tilde f^{m_\beta}_\sharp\beta_l}_{\tilde f^{m_\beta}_\sharp\beta},\phi
\Big\rangle_{f_\sharp\beta}\Big|
\\
&\le
\|J_f(m_\beta)\|_2\delta \|1-\sqrt w\|_{L^2(f_\sharp\beta)}\cdot (\|J_f(m_\beta)\|_2\delta)
+
(\|J_f(m_\beta)\|_2\delta)\cdot \|J_f(m_\beta)\|_2\delta\|1-\sqrt w\|_{L^2(f_\sharp\beta)}
\\
&\le
2(\|J_f(m_\beta)\|_2\delta)^2 \|1-\sqrt w\|_{L^2(f_\sharp\beta)}.
\end{aligned}
\end{equation*}
Hence
\begin{equation*}
\begin{aligned}
\|\Delta^{(3)}\phi\|_{L^2(f_\sharp\beta)} &= \Bigg\|
\sum_{k=1}^K \sum_{l=1}^K (C^\dagger)_{kl}\Big[
g^{\tilde f^{m_\beta}_\sharp\beta_k}_{\tilde f^{m_\beta}_\sharp\beta}
\Big\langle
g^{\tilde f^{m_\beta}_\sharp\beta_l}_{\tilde f^{m_\beta}_\sharp\beta},\phi
\Big\rangle_{f_\sharp\beta}
-
\Big(U g^{\tilde f^{m_\beta}_\sharp\beta_k}_{\tilde f^{m_\beta}_\sharp\beta}\Big)
\Big\langle
\Big(U g^{\tilde f^{m_\beta}_\sharp\beta_l}_{\tilde f^{m_\beta}_\sharp\beta}\Big),\phi
\Big\rangle_{f_\sharp\beta}
\Big]\Bigg\|_{L^2(f_\sharp\beta)}
\\
&\le
2(\|J_f(m_\beta)\|_2\delta)^2 \|1-\sqrt w\|_{L^2(f_\sharp\beta)}
\sum_{k=1}^K \sum_{l=1}^K |(C^\dagger)_{kl}|\\
&\le
2(\|J_f(m_\beta)\|_2\delta)^2 \|1-\sqrt w\|_{L^2(f_\sharp\beta)} K\|C^\dagger\|_F
\\
&\le
2(\|J_f(m_\beta)\|_2\delta)^2 \|1-\sqrt w\|_{L^2(f_\sharp\beta)} K^\frac{3}{2}\|C^\dagger\|_2.
\end{aligned}
\end{equation*}
By \Cref{lem:BC_comparison}, $\|C^\dagger\|_2=O(\delta^{-4})$. This implies
\begin{equation}\label{eq:Delta3_order}
\|\Delta^{(3)}\|_{op}
=
O(\delta^{-2})
\|1-\sqrt w\|_{L^2(f_\sharp\beta)}.
\end{equation}
Combining \eqref{eq:Delta_decomposition}, \eqref{eq:Delta1_order},
\eqref{eq:Delta2_order}, and \eqref{eq:Delta3_order}, we obtain
\begin{equation*}
\|F^\dagger_{N_{f_\sharp \beta}}  - \tilde F^\dagger_{N_{\tilde f^{m_\beta}_\sharp\beta}}\|_{op}
=
O(\delta^{p/2-3})
+
O(\delta^{-2})
\|1-\sqrt w\|_{L^2(f_\sharp\beta)}.
\end{equation*}
It remains to rewrite the last term using the Hellinger distance. By \Cref{def:hellinger_standard}, with
\begin{equation*}
P=\tilde f^{m_\beta}_\sharp\beta,
\qquad
Q=f_\sharp\beta,
\qquad
\lambda = f_\sharp\beta.
\end{equation*}
Then,
\begin{equation*}
\frac{dQ}{d\lambda}(y)=\frac{d(f_\sharp\beta)}{d(f_\sharp\beta)}(y)=1,
\qquad
\frac{dP}{d\lambda}(y)=\frac{d(\tilde f^{m_\beta}_\sharp\beta)}{d(f_\sharp\beta)}(y)=w(y).
\end{equation*}
Substituting these into \eqref{eq:hellinger_standard_lambda} yields
\begin{equation} \label{eq:Hellinger_distance_transition}
H^2 \Big(\tilde f^{m_\beta}_\sharp\beta,\ f_\sharp\beta\Big)
=
\frac12\int_{\mathbb{R}^n}\bigl(\sqrt{w(y)}-\sqrt{1}\bigr)^2 d(f_\sharp\beta)(y)
=
\frac12\|\sqrt{w}-{1}\bigr\|^2_{L^2(f_\sharp\beta)}.
\end{equation}
It follows that 
\begin{equation*}
    \|F^\dagger_{N_{f_\sharp \beta}}  - \tilde F^\dagger_{N_{\tilde f^{m_\beta}_\sharp\beta}}\|_{op}
=
O(\delta^{p/2-3})
+
O(\delta^{-2} H\left(\tilde f^{m_\beta}_\sharp\beta, f_\sharp\beta\right)).\qedhere
\end{equation*}
\end{enumerate}
\end{proof}

\section{Additional results}\label{sec: Appendix B}

\begin{lemma} \label{uniform distribute covariance}
Let $X$ be a random vector in $\mathbb{R}^n$ uniformly distributed over the ball $B(m,\delta)=\{x\in\mathbb{R}^n:\|x-m\|\le \delta\}\subset \mathbb{R}^n$. Then
\begin{equation*}
\mathrm{Cov}(X)
=\mathbb{E}\big[(X-\mathbb{E}[X])(X-\mathbb{E}[X])^\top\big]
=\mathbb{E}\big[(X-m)(X-m)^\top\big]
=\frac{\delta^2}{n+2} I_n .
\end{equation*}
where $I_n$ is the $n \times n $ identity matrix.

\begin{proof}
By symmetry, $\mathbb E[X]=m$, and direct integration in polar coordinates shows that each diagonal entry of the covariance equals to
\begin{equation*}
    \frac{1}{n}
\frac{\int_0^\delta r^{n+1} dr}
{\int_0^\delta r^{n-1} dr}
=
\frac{\delta^2}{n+2},
\end{equation*}
while the off-diagonal entries vanish. The result then follows immediately. 
\end{proof}

\end{lemma}

\begin{lemma}\label{lemma: WM-inner-covariance}
  Under the same setup as \Cref{linear approx of f approx euclidean mahalanobis}, we have that
    \begin{equation}
    \big\langle g^{\tilde f^{m_\beta}_\sharp\alpha}_{\tilde f^{m_\beta}_\sharp\beta}, 
F^\dagger_{N_{\tilde f^{m_\beta}_\sharp\beta}}
\big(g^{\tilde f^{m_\beta}_\sharp\alpha}_{\tilde f^{m_\beta}_\sharp\beta}\big)
\big\rangle_{\tilde f^{m_\beta}_\sharp\beta} = (J_f(m_\beta)(m_\alpha-m_\beta))^\top M^\dagger_{N_{\tilde f^{m_\beta}_\sharp\beta}}(J_f(m_\beta)(m_\alpha-m_\beta)),
\end{equation}
where $  M_{N_{\tilde f^{m_\beta}_\sharp\beta}} = \frac{1}{K}\sum_{i=1}^K
    J_f(m_\beta)(m_{\beta_i}-m_\beta)\big(J_f(m_\beta)(m_{\beta_i}-m_\beta)\big)^\top$.
\end{lemma}

\begin{proof}
Under the affine approximation $\tilde f^{m_\beta}$, $\tilde f^{m_\beta}_\sharp\alpha$ and $\tilde f^{m_\beta}_\sharp\beta$
are Gaussians with the same covariance, and hence the OT displacement is a translation,
\begin{equation}\label{eq:linear displacement}
    g^{\tilde f^{m_\beta}_\sharp\alpha}_{\tilde f^{m_\beta}_\sharp\beta}
=J_f(m_\beta)(m_\alpha-m_\beta).
\end{equation}
Similarly, for each neighbor $\beta_i$,
\begin{equation*}
g^{\tilde f^{m_\beta}_\sharp\beta_i}_{\tilde f^{m_\beta}_\sharp\beta}
=J_f(m_\beta)(m_{\beta_i}-m_\beta).
\end{equation*}
Hence the empirical covariance operator acts on any
$\phi\in L^2(\tilde f^{m_\beta}_\sharp\beta)$ as
\begin{equation*}
F_{N_{\tilde f^{m_\beta}_\sharp\beta}}(\phi)
=
M_{N_{\tilde f^{m_\beta}_\sharp\beta}}
\int_{\mathbb R^n}\phi(y) d\tilde f^{m_\beta}_\sharp\beta(y),
\end{equation*}
where
\begin{equation*}
M_{N_{\tilde f^{m_\beta}_\sharp\beta}}
=
\frac1K\sum_{i=1}^K
J_f(m_\beta)(m_{\beta_i}-m_\beta)
\big(J_f(m_\beta)(m_{\beta_i}-m_\beta)\big)^\top .
\end{equation*}    
Therefore, using the fact that $\tilde f^{m_\beta}_\sharp\beta$ is a probability measure, i.e.\ $ \int_{\mathbb{R}^n} 1   d \tilde f^{m_\beta}_\sharp\beta (y) = 1$,

\begin{align*}
&\big\langle g^{\tilde f^{m_\beta}_\sharp\alpha}_{\tilde f^{m_\beta}_\sharp\beta},
F^\dagger_{N_{\tilde f^{m_\beta}_\sharp\beta}}
\big(g^{\tilde f^{m_\beta}_\sharp\alpha}_{\tilde f^{m_\beta}_\sharp\beta}\big)
\big\rangle_{\tilde f^{m_\beta}_\sharp\beta} =
\big(J_f(m_\beta)(m_\alpha-m_\beta)\big)^\top
M^\dagger_{N_{\tilde f^{m_\beta}_\sharp\beta}}
\big(J_f(m_\beta)(m_\alpha-m_\beta)\big). \qedhere
\end{align*}
\end{proof}

\begin{example}\label{example: euclidean example}
Let $m_\alpha,m_\beta\in\mathbb R$ and define $h=m_\beta-m_\alpha$. Consider the nonlinear map $f:\mathbb R\to\mathbb R$
\begin{equation*}
    f(x)=x+x^3.
\end{equation*}
Let $\tilde f$ denote the first-order approximation of $f$ at
$m_\alpha$,
\begin{equation*}
    \tilde f(x)
    = f(m_\alpha)+f'(m_\alpha)(x-m_\alpha),
    \qquad \text{where }
    f'(x)=1+3x^2.
\end{equation*}
We compare the displacement between the mapped points under $f$ with
that under $\tilde f$. The displacement between the two mapped points under $f$ is simply
\begin{equation*}
    f(m_\beta)-f(m_\alpha)
    =
    f(m_\alpha+h)-f(m_\alpha).
\end{equation*}
Expanding gives
\begin{align*}
    f(m_\alpha+h)-f(m_\alpha)
    =
    (m_\alpha+h)+(m_\alpha+h)^3
    -
    (m_\alpha+m_\alpha^3)  
    =
    (1+3m_\alpha^2)h+3m_\alpha h^2+h^3 .
\end{align*}
On the other hand, the displacement between the two mapped points under $\tilde f$ is
\begin{equation*}
\tilde f(m_\beta)-\tilde f(m_\alpha)
    = f'(m_\alpha)(m_\beta-m_\alpha)
    = (1+3m_\alpha^2)h.
\end{equation*}
Subtracting the two displacements gives
\begin{align*}
\bigl(f(m_\alpha+h)-f(m_\alpha)\bigr)-f'(m_\alpha)h
    = 3m_\alpha h^2 + h^3.
\end{align*}
It follows that the squared Euclidean error is
\begin{align*}
(3m_\alpha h^2 + h^3)^2 
    &= 9m_\alpha^2 h^4 + 6m_\alpha h^5 + h^6,
\end{align*}
which has leading order $O(h^4)$ .
\end{example}
\begin{lemma}\label{sigular value multiplication bound}
Let $A\in\mathbb R^{m\times m}$ be full-rank and let $B\in\mathbb R^{m\times n}$. Then for each
$
i=1,\dots,\rank(B),
$
we have
\begin{equation*}
\sigma_i(AB)\ge \sigma_{\min}(A)\sigma_i(B).
\end{equation*}
In particular,
\begin{equation*}
\sigma_{\min}^+(AB)\ge \sigma_{\min}(A)\sigma_{\min}^+(B).
\end{equation*}
\begin{proof}
    For every $x\in \mathbb R^m$ we have
\begin{equation}\label{eq:A_lower_bound}
\|Ax\|\ge \sigma_{\min}(A)\|x\|.
\end{equation}
Using \eqref{eq:A_lower_bound} with $x=By$, we get
\begin{equation*}
\frac{\|ABy\|}{\|y\|}
\ge
\sigma_{\min}(A)\frac{\|By\|}{\|y\|}.
\end{equation*}
Therefore, for every $S\subseteq \mathbb R^n$ with $\dim S=i$,
\begin{equation*}
\min_{\substack{y\in S\\ y\neq 0}}
\frac{\|ABy\|}{\|y\|}
\ge
\sigma_{\min}(A)
\min_{\substack{y\in S\\ y\neq 0}}
\frac{\|By\|}{\|y\|}.
\end{equation*}
Taking the max over all $i$-dimensional subspaces $S$, we obtain
\begin{equation*}
\sigma_i(AB)
\ge
\sigma_{\min}(A)
\max_{\dim S=i}\ \min_{\substack{y\in S\\ y\neq 0}}
\frac{\|By\|}{\|y\|}
=
\sigma_{\min}(A)\sigma_i(B).
\end{equation*}
Thus
\begin{equation*}
\sigma_i(AB)\ge \sigma_{\min}(A)\sigma_i(B),
\qquad i=1,\dots,\rank(B).
\end{equation*}
Taking $i=\rank(B)$ gives
\begin{equation*}
\sigma_{\min}^+(AB)
=
\sigma_{\rank(B)}(AB)
\ge
\sigma_{\min}(A)\sigma_{\rank(B)}(B)
=
\sigma_{\min}(A)\sigma_{\min}^+(B). \qedhere
\end{equation*}
\end{proof}
\end{lemma}

\begin{lemma}[Weyl's inequality, \cite{Tao2010Eigenvalues,Weyl1912Das_asymptotische}]\label{Weyl's inequality}
Let $A,B\in\mathbb{C}^{n\times n}$ be Hermitian matrices. For any Hermitian matrix $M$, denote its eigenvalues in decreasing order by
\begin{equation*}
\lambda_1(M)\ge \lambda_2(M)\ge \cdots \ge \lambda_n(M).
\end{equation*}
Then for every $k\in\{1,\dots,n\}$,

\begin{equation}\label{eq:weyl_opnorm_bound}
\bigl|\lambda_k(A+B)-\lambda_k(A)\bigr|\ \le\ \|B\|_{\mathrm{op}},
\end{equation}
where the operator norm is
\begin{equation*}
\|B\|_{\mathrm{op}} =\max\bigl(|\lambda_1(B)|, |\lambda_n(B)|\bigr).
\end{equation*}
\end{lemma}

\begin{lemma}\label{lemma:bounds for C}
Let
$m_\beta \in \mathbb R^n$, $\delta>0$, and let $m_{\beta_i}$, $i=1,\dots,K$ be uniformly distributed in
$B(m_\beta,\delta)$. Let $f:\mathbb{R}^n\to\mathbb{R}^n$ satisfy \Cref{assump:regularity}, and assume that $K$, $\delta$ satisfy \Cref{Assumption: empirical sampling}. Define
\begin{equation*}
    V
    =
    \big[
    J_f(m_\beta)(m_{\beta_1}-m_\beta),\dots,
    J_f(m_\beta)(m_{\beta_K}-m_\beta)
    \big]\in\mathbb R^{n\times K}.
\end{equation*}
Let
\begin{equation}\label{eq:C_Gram_linear}
    C=
    \frac{1}{K}
    \big(V^\top V\big)^2.
\end{equation}
Then
\begin{equation*}
    \|C^\dagger\|_2
    \leq
    \frac{(n+2)^2}
    {K\delta^4
    \bigl[\sigma_{\min}(J_f(m_\beta))\bigr]^4}.
\end{equation*}
\begin{proof}
    Let $r=\rank(V)=\rank(C)\le n$, and denote by
$\sigma_{\min}^+(V)=\sigma_r(V)$ the smallest positive singular value of $V$. Since the nonzero eigenvalues of $V^\top V$
are $\{\sigma_i(V)^2\}_{i=1}^r$, \eqref{eq:C_Gram_linear} implies
\begin{equation*}
\sigma_{\min}^+(C)
=
\frac1K\big(\sigma_{\min}^+(V)\big)^4,
\end{equation*}
hence
\begin{equation*}
\|C^\dagger\|_2=\frac{1}{\sigma_{\min}^+(C)}
=\frac{K}{\big(\sigma_{\min}^+(V)\big)^4}.
\end{equation*}
Next, write
\begin{equation*}
V=J_f(m_\beta)H, \qquad H=\big[(m_{\beta_1}-m_\beta),\dots,(m_{\beta_K}-m_\beta)\big]\in\R^{n\times K}.
\end{equation*}
By \Cref{sigular value multiplication bound},
\begin{equation*}
    \sigma_{\min}^+(V) = \sigma_{\min}^+(J_f(m_\beta)H) \geq \sigma_{\min}(J_f(m_\beta))\sigma_{\min}^+(H).
\end{equation*}
Under the uniform sampling assumption of $\{m_{\beta_i}\}_{i=1}^K$ and \Cref{Assumption: empirical sampling},
\begin{align*}
    \frac{1}{K}HH^\top = \frac{\delta^2}{n+2} I \quad \text{so that} \quad
    HH^\top =\frac{K\delta^2}{n+2} I.
\end{align*}
It follows that $\{\sigma_i (H)\}_{i=1}^r = \sqrt{\frac{K\delta^2}{n+2}}$, and therefore
\begin{align*}
\sigma^+_{\min}\bigl(V\bigr)
      \geq\sigma_{\min}\bigl(J_f(m_\beta)\bigr)
         \sqrt{\frac{K\delta^{2}}{n+2}}.
\end{align*}
Thus
\begin{equation*}
\|C^\dagger\|_2
   =\frac{K}{\bigl(\sigma_{\min}^+(V)\bigr)^{4}}
   \leq
   \frac{(n+2)^{2}}
        {K\delta^{4}
         \bigl[\sigma_{\min}\bigl(J_f(m_\beta)\bigr)\bigr]^{4}} . \qedhere
\end{equation*}

\end{proof}
\end{lemma}

\begin{lemma} \label{lemma:bounds for B-C}
Let $\beta = \mathcal{N}(m_\beta, \Sigma)$ be a Gaussian distribution on $\mathbb R^n$, let $\delta>0$ and let $\beta_i=\mathcal N(m_{\beta_i},\Sigma), i=1,\dots,K$, where the means $m_{\beta_i}$ are  uniformly distributed in ${B}(m_\beta,\delta)$. Let $f:\mathbb{R}^n\to\mathbb{R}^n$ satisfy \Cref{assump:regularity}, and assume that $K$, $\delta$ satisfy \Cref{Assumption: empirical sampling}. Let $\tilde{f}^{m_\beta}$ be linear approximation of $f$ at $m_\beta$ \eqref{eq:affine-approx}. Denote by
\begin{equation*}
N_{\tilde f^{m_\beta}_\sharp\beta} = \Big\{g^{\tilde f^{m_\beta}_\sharp\beta_i}_{\tilde f^{m_\beta}_\sharp\beta}\Big\}_{i=1}^K,
\qquad
N_{f_\sharp\beta} = \Big\{g^{f_\sharp \beta_i}_{f_\sharp \beta}\Big\}_{i=1}^K,
\end{equation*}
the sets containing OT displacement functions \eqref{eq:displacement_mu_mui} from $\tilde f^{m_\beta}_\sharp\beta$ to $\tilde f^{m_\beta}_\sharp\beta_i$, and from ${f}_\sharp \beta$ to ${f}_\sharp \beta_i$, respectively. Define
\begin{equation}\label{eq: displacement difference}
\xi_i(y) = g^{f_\sharp \beta_i}_{f_\sharp \beta}(y) -g^{\tilde f^{m_\beta}_\sharp\beta_i}_{\tilde f^{m_\beta}_\sharp\beta} (y).
\end{equation}
Then
\begin{align*}
    &\frac1K\sum_{i=1}^K\Bigg[
\Big\langle g^{f_\sharp\beta_i}_{f_\sharp\beta}, g^{f_\sharp\beta_l}_{f_\sharp\beta}\Big\rangle_{f_\sharp\beta}
\Big\langle g^{f_\sharp\beta_i}_{f_\sharp\beta}, g^{f_\sharp\beta_k}_{f_\sharp\beta}\Big\rangle_{f_\sharp\beta} -
\Big\langle g^{\tilde f^{m_\beta}_\sharp\beta_i}_{\tilde f^{m_\beta}_\sharp\beta}, g^{\tilde f^{m_\beta}_\sharp\beta_l}_{\tilde f^{m_\beta}_\sharp\beta}\Big\rangle_{\tilde f^{m_\beta}_\sharp\beta}
\Big\langle g^{\tilde f^{m_\beta}_\sharp\beta_i}_{\tilde f^{m_\beta}_\sharp\beta}, g^{\tilde f^{m_\beta}_\sharp\beta_k}_{\tilde f^{m_\beta}_\sharp\beta}\Big\rangle_{\tilde f^{m_\beta}_\sharp\beta}
\Bigg]\\
&\leq\frac1K\sum_{i=1}^K
\Bigg(\Big[
\|J_f(m_\beta)\|_2\delta
\big(
\|\xi_i\|_{L^2(f_\sharp\beta)}
+
\|\xi_l\|_{L^2(f_\sharp\beta)}
\big)
+
\|\xi_i\|_{L^2(f_\sharp\beta)}
\|\xi_l\|_{L^2(f_\sharp\beta)}
\Big] \\
&\times
\Big(
\|J_f(m_\beta)\|_2\delta
+
\|\xi_i\|_{L^2(f_\sharp\beta)}
\Big)
\Big(
\|J_f(m_\beta)\|_2\delta
+
\|\xi_k\|_{L^2(f_\sharp\beta)}
\Big)\\
&
+
\|J_f(m_\beta)\|_2^2\delta^2
\Big[
\|J_f(m_\beta)\|_2\delta
\big(
\|\xi_i\|_{L^2(f_\sharp\beta)}
+
\|\xi_k\|_{L^2(f_\sharp\beta)}
\big)
+
\|\xi_i\|_{L^2(f_\sharp\beta)}
\|\xi_k\|_{L^2(f_\sharp\beta)}
\Big]\Bigg).
\end{align*}
\begin{proof}
Since $\tilde f^{m_\beta}$ is affine, $g^{\tilde f^{m_\beta}_\sharp\beta_j}_{\tilde f^{m_\beta}_\sharp\beta}(y)
=
J_f(m_\beta)(m_{\beta_j}-m_\beta)$ are constant functions for each $j=1,\dots,K$.
Therefore, the linearized inner products may be evaluated
in $L^2(f_\sharp\beta)$:
\begin{equation*}
    \Big\langle g^{\tilde f^{m_\beta}_\sharp\beta_i}_{\tilde f^{m_\beta}_\sharp\beta},\ g^{\tilde f^{m_\beta}_\sharp\beta_l}_{\tilde f^{m_\beta}_\sharp\beta}\Big\rangle_{\tilde f^{m_\beta}_\sharp\beta} = \Big\langle g^{\tilde f^{m_\beta}_\sharp\beta_i}_{\tilde f^{m_\beta}_\sharp\beta},\ g^{\tilde f^{m_\beta}_\sharp\beta_l}_{\tilde f^{m_\beta}_\sharp\beta}\Big\rangle_{f_\sharp\beta}.
\end{equation*}
Thus,
\begin{align*}
&\frac1K\sum_{i=1}^K\Bigg[
\Big\langle g^{f_\sharp\beta_i}_{f_\sharp\beta}, g^{f_\sharp\beta_l}_{f_\sharp\beta}\Big\rangle_{f_\sharp\beta}
\Big\langle g^{f_\sharp\beta_i}_{f_\sharp\beta}, g^{f_\sharp\beta_k}_{f_\sharp\beta}\Big\rangle_{f_\sharp\beta} -
\Big\langle g^{\tilde f^{m_\beta}_\sharp\beta_i}_{\tilde f^{m_\beta}_\sharp\beta}, g^{\tilde f^{m_\beta}_\sharp\beta_l}_{\tilde f^{m_\beta}_\sharp\beta}\Big\rangle_{\tilde f^{m_\beta}_\sharp\beta}
\Big\langle g^{\tilde f^{m_\beta}_\sharp\beta_i}_{\tilde f^{m_\beta}_\sharp\beta}, g^{\tilde f^{m_\beta}_\sharp\beta_k}_{\tilde f^{m_\beta}_\sharp\beta}\Big\rangle_{\tilde f^{m_\beta}_\sharp\beta}
\Bigg]\\
=&\frac1K\sum_{i=1}^K\Bigg[
\Big\langle g^{f_\sharp\beta_i}_{f_\sharp\beta}, g^{f_\sharp\beta_l}_{f_\sharp\beta}\Big\rangle_{f_\sharp\beta}
\Big\langle g^{f_\sharp\beta_i}_{f_\sharp\beta}, g^{f_\sharp\beta_k}_{f_\sharp\beta}\Big\rangle_{f_\sharp\beta} -
\Big\langle g^{\tilde f^{m_\beta}_\sharp\beta_i}_{\tilde f^{m_\beta}_\sharp\beta}, g^{\tilde f^{m_\beta}_\sharp\beta_l}_{\tilde f^{m_\beta}_\sharp\beta}\Big\rangle_{f_\sharp\beta}
\Big\langle g^{\tilde f^{m_\beta}_\sharp\beta_i}_{\tilde f^{m_\beta}_\sharp\beta}, g^{\tilde f^{m_\beta}_\sharp\beta_k}_{\tilde f^{m_\beta}_\sharp\beta}\Big\rangle_{f_\sharp\beta}
\Bigg].
\end{align*}
Applying the identity $ab-cd=(a-c)b+c(b-d)$, we obtain
\begin{equation}\label{eq:D_expansion}
\begin{aligned}
    &\Big\langle g^{f_\sharp\beta_i}_{f_\sharp\beta},\ g^{f_\sharp\beta_l}_{f_\sharp\beta}\Big\rangle_{f_\sharp\beta}
\Big\langle g^{f_\sharp\beta_i}_{f_\sharp\beta},\ g^{f_\sharp\beta_k}_{f_\sharp\beta}\Big\rangle_{f_\sharp\beta}
-
\Big\langle g^{\tilde f^{m_\beta}_\sharp\beta_i}_{\tilde f^{m_\beta}_\sharp\beta},\ g^{\tilde f^{m_\beta}_\sharp\beta_l}_{\tilde f^{m_\beta}_\sharp\beta}\Big\rangle_{f_\sharp\beta}
\Big\langle g^{\tilde f^{m_\beta}_\sharp\beta_i}_{\tilde f^{m_\beta}_\sharp\beta},\ g^{\tilde f^{m_\beta}_\sharp\beta_k}_{\tilde f^{m_\beta}_\sharp\beta}\Big\rangle_{f_\sharp\beta}\\
=&\Big(
\Big\langle g^{f_\sharp\beta_i}_{f_\sharp\beta},\ g^{f_\sharp\beta_l}_{f_\sharp\beta}\Big\rangle_{f_\sharp\beta}
-
\Big\langle g^{\tilde f^{m_\beta}_\sharp\beta_i}_{\tilde f^{m_\beta}_\sharp\beta},\ g^{\tilde f^{m_\beta}_\sharp\beta_l}_{\tilde f^{m_\beta}_\sharp\beta}\Big\rangle_{f_\sharp\beta}
\Big)
\Big\langle g^{f_\sharp\beta_i}_{f_\sharp\beta},\ g^{f_\sharp\beta_k}_{f_\sharp\beta}\Big\rangle_{f_\sharp\beta}\\
+&\Big\langle g^{\tilde f^{m_\beta}_\sharp\beta_i}_{\tilde f^{m_\beta}_\sharp\beta},\ g^{\tilde f^{m_\beta}_\sharp\beta_l}_{\tilde f^{m_\beta}_\sharp\beta}\Big\rangle_{f_\sharp\beta}
\Big(
\Big\langle g^{f_\sharp\beta_i}_{f_\sharp\beta},\ g^{f_\sharp\beta_k}_{f_\sharp\beta}\Big\rangle_{f_\sharp\beta}
-
\Big\langle g^{\tilde f^{m_\beta}_\sharp\beta_i}_{\tilde f^{m_\beta}_\sharp\beta},\ g^{\tilde f^{m_\beta}_\sharp\beta_k}_{\tilde f^{m_\beta}_\sharp\beta}\Big\rangle_{f_\sharp\beta}
\Big).
\end{aligned}
\end{equation}
Using \eqref{eq: displacement difference}, we have
\begin{align*}
\Big\langle g^{f_\sharp\beta_i}_{f_\sharp\beta},\ g^{f_\sharp\beta_l}_{f_\sharp\beta}\Big\rangle_{f_\sharp\beta}
-
\Big\langle g^{\tilde f^{m_\beta}_\sharp\beta_i}_{\tilde f^{m_\beta}_\sharp\beta},\ g^{\tilde f^{m_\beta}_\sharp\beta_l}_{\tilde f^{m_\beta}_\sharp\beta}\Big\rangle_{f_\sharp\beta}
&=
\big\langle \xi_i,\ g^{\tilde f^{m_\beta}_\sharp\beta_l}_{\tilde f^{m_\beta}_\sharp\beta}\big\rangle_{f_\sharp\beta}
+\big\langle g^{\tilde f^{m_\beta}_\sharp\beta_i}_{\tilde f^{m_\beta}_\sharp\beta},\ \xi_l\big\rangle_{f_\sharp\beta}
+\langle \xi_i,\xi_l\rangle_{f_\sharp\beta}.
\end{align*}
By Cauchy--Schwarz,
\begin{align*}
&\Big|
\Big\langle g^{f_\sharp\beta_i}_{f_\sharp\beta},\ g^{f_\sharp\beta_l}_{f_\sharp\beta}\Big\rangle_{f_\sharp\beta}
-
\Big\langle g^{\tilde f^{m_\beta}_\sharp\beta_i}_{\tilde f^{m_\beta}_\sharp\beta},\ g^{\tilde f^{m_\beta}_\sharp\beta_l}_{\tilde f^{m_\beta}_\sharp\beta}\Big\rangle_{f_\sharp\beta}
\Big|\\
&\le
\|\xi_i\|_{L^2(f_\sharp\beta)} \Big\|g^{\tilde f^{m_\beta}_\sharp\beta_l}_{\tilde f^{m_\beta}_\sharp\beta}\Big\|_{L^2(f_\sharp\beta)}
+
\Big\|g^{\tilde f^{m_\beta}_\sharp\beta_i}_{\tilde f^{m_\beta}_\sharp\beta}\Big\|_{L^2(f_\sharp\beta)} \|\xi_l\|_{L^2(f_\sharp\beta)}
+\|\xi_i\|_{L^2(f_\sharp\beta)} \|\xi_l\|_{L^2(f_\sharp\beta)}.
\end{align*}
We have
\begin{equation*}
\Big\|g^{\tilde f^{m_\beta}_\sharp\beta_j}_{\tilde f^{m_\beta}_\sharp\beta}\Big\|_{L^2(f_\sharp\beta)}
\le
\|J_f(m_\beta)\|_2 \delta.
\end{equation*}
Moreover,
\begin{align*}
\Big\|g^{f_\sharp\beta_i}_{f_\sharp\beta}\Big\|_{L^2(f_\sharp\beta)}
&\le
\Big\|g^{\tilde f^{m_\beta}_\sharp\beta_i}_{\tilde f^{m_\beta}_\sharp\beta}\Big\|_{L^2(f_\sharp\beta)}
+
\|\xi_i\|_{L^2(f_\sharp\beta)} \\
&\le
\|J_f(m_\beta)\|_2 \delta + \|\xi_i\|_{L^2(f_\sharp\beta)}.
\end{align*}
Hence
\begin{equation*}
\Big|\Big\langle g^{f_\sharp\beta_i}_{f_\sharp\beta},\ g^{f_\sharp\beta_k}_{f_\sharp\beta}\Big\rangle_{f_\sharp\beta}\Big|
\leq
\Big(
\|J_f(m_\beta)\|_2\delta
+
\|\xi_i\|_{L^2(f_\sharp\beta)}
\Big)
\Big(
\|J_f(m_\beta)\|_2\delta
+
\|\xi_k\|_{L^2(f_\sharp\beta)}
\Big).
\end{equation*}
Similarly,
\begin{equation*}
\Big|
\Big\langle g^{\tilde f^{m_\beta}_\sharp\beta_i}_{\tilde f^{m_\beta}_\sharp\beta},\ g^{\tilde f^{m_\beta}_\sharp\beta_l}_{\tilde f^{m_\beta}_\sharp\beta}\Big\rangle_{f_\sharp\beta}
\Big|
\le
\|J_f(m_\beta)\|_2^2 \delta^2.
\end{equation*}
Combining these bounds into \eqref{eq:D_expansion},
\begin{align*}
&\Big| \Big\langle g^{f_\sharp\beta_i}_{f_\sharp\beta},\ g^{f_\sharp\beta_l}_{f_\sharp\beta}\Big\rangle_{f_\sharp\beta}
\Big\langle g^{f_\sharp\beta_i}_{f_\sharp\beta},\ g^{f_\sharp\beta_k}_{f_\sharp\beta}\Big\rangle_{f_\sharp\beta}
-
\Big\langle g^{\tilde f^{m_\beta}_\sharp\beta_i}_{\tilde f^{m_\beta}_\sharp\beta},\ g^{\tilde f^{m_\beta}_\sharp\beta_l}_{\tilde f^{m_\beta}_\sharp\beta}\Big\rangle_{f_\sharp\beta}
\Big\langle g^{\tilde f^{m_\beta}_\sharp\beta_i}_{\tilde f^{m_\beta}_\sharp\beta},\ g^{\tilde f^{m_\beta}_\sharp\beta_k}_{\tilde f^{m_\beta}_\sharp\beta}\Big\rangle _{f_\sharp\beta}\Big| \\
&\leq
\Big[
\|J_f(m_\beta)\|_2\delta
\big(
\|\xi_i\|_{L^2(f_\sharp\beta)}
+
\|\xi_l\|_{L^2(f_\sharp\beta)}
\big)
+
\|\xi_i\|_{L^2(f_\sharp\beta)}
\|\xi_l\|_{L^2(f_\sharp\beta)}
\Big] \\
&\times
\Big(
\|J_f(m_\beta)\|_2\delta
+
\|\xi_i\|_{L^2(f_\sharp\beta)}
\Big)
\Big(
\|J_f(m_\beta)\|_2\delta
+
\|\xi_k\|_{L^2(f_\sharp\beta)}
\Big)\\
&
+
\|J_f(m_\beta)\|_2^2\delta^2
\Big[
\|J_f(m_\beta)\|_2\delta
\big(
\|\xi_i\|_{L^2(f_\sharp\beta)}
+
\|\xi_k\|_{L^2(f_\sharp\beta)}
\big)
+
\|\xi_i\|_{L^2(f_\sharp\beta)}
\|\xi_k\|_{L^2(f_\sharp\beta)}
\Big].
\end{align*}
Taking the average over $i=1,\dots,K$ gives the desired result.
\end{proof}
\end{lemma}

\begin{lemma}\label{Lemma:operater_defference_explicit}
Let $\beta = \mathcal{N}(m_\beta, \Sigma)$ be a Gaussian distribution on $\mathbb R^n$, let $\delta>0$, and let $\beta_i=\mathcal N(m_{\beta_i},\Sigma), i=1,\dots,K$, where the means $m_{\beta_i}$ are uniformly distributed in ${B}(m_\beta,\delta)$. Let $f:\mathbb{R}^n\to\mathbb{R}^n$ satisfy \Cref{assump:regularity}, and assume that $K$, $\delta$ satisfy \Cref{Assumption: empirical sampling}. Let $\tilde{f}^{m_\beta}$ be the linear approximation of $f$ at $m_\beta$ \eqref{eq:affine-approx}. Denote by
\begin{equation*}
N_{\tilde f^{m_\beta}_\sharp\beta} = \Big\{g^{\tilde f^{m_\beta}_\sharp\beta_i}_{\tilde f^{m_\beta}_\sharp\beta}\Big\}_{i=1}^K,
\qquad
N_{f_\sharp\beta} = \Big\{g^{f_\sharp \beta_i}_{f_\sharp \beta}\Big\}_{i=1}^K,
\end{equation*}
the sets containing OT displacement functions \eqref{eq:displacement_mu_mui} from $\tilde f^{m_\beta}_\sharp\beta$ to $\tilde f^{m_\beta}_\sharp\beta_i$, and from ${f}_\sharp \beta$ to ${f}_\sharp \beta_i$, respectively. Let $F^\dagger_{N_{f_\sharp\beta}}$, $F^\dagger_{N_{\tilde f^{m_\beta}_\sharp\beta}}$
be the pseudo-inverse empirical covariance operators \eqref{eq:Fdagger_Nalpha} defined by $N_{f_\sharp\beta}$ and $N_{\tilde f^{m_\beta}_\sharp\beta}$, respectively. Let
$U$
be the operator defined in \eqref{eq:U_def} and let
\begin{equation*}
\tilde F^\dagger_{N_{\tilde f^{m_\beta}_\sharp\beta}}
=
UF^\dagger_{N_{\tilde f^{m_\beta}_\sharp\beta}}U^{-1}.
\end{equation*}
Then 
\begin{equation}\label{eq:operator_diff_explicit}
\begin{aligned}
F^\dagger_{N_{f_\sharp \beta}} -\tilde F^\dagger_{N_{\tilde f^{m_\beta}_\sharp\beta}}  
 &=
\sum_{k=1}^K \sum_{l=1}^K \Big((B^\dagger)_{kl}-(C^\dagger)_{kl}\Big) 
g^{\tilde f^{m_\beta}_\sharp\beta_k}_{\tilde f^{m_\beta}_\sharp\beta} 
\Big\langle
g^{\tilde f^{m_\beta}_\sharp\beta_l}_{\tilde f^{m_\beta}_\sharp\beta},\cdot
\Big\rangle_{f_\sharp\beta}
\\
&+
\sum_{k=1}^K \sum_{l=1}^K (B^\dagger)_{kl}\Bigg[
g^{\tilde f^{m_\beta}_\sharp\beta_k}_{\tilde f^{m_\beta}_\sharp\beta} 
\big\langle \xi_l,\cdot\big\rangle_{f_\sharp\beta}
+
\xi_k 
\Big\langle
g^{\tilde f^{m_\beta}_\sharp\beta_l}_{\tilde f^{m_\beta}_\sharp\beta},\cdot
\Big\rangle_{f_\sharp\beta}
+
\xi_k \big\langle \xi_l,\cdot\big\rangle_{f_\sharp\beta}
\Bigg]
\\
&+
\sum_{k=1}^K \sum_{l=1}^K (C^\dagger)_{kl}\Bigg[
g^{\tilde f^{m_\beta}_\sharp\beta_k}_{\tilde f^{m_\beta}_\sharp\beta} 
\Big\langle
g^{\tilde f^{m_\beta}_\sharp\beta_l}_{\tilde f^{m_\beta}_\sharp\beta},\cdot
\Big\rangle_{f_\sharp\beta}
-
\Big(U g^{\tilde f^{m_\beta}_\sharp\beta_k}_{\tilde f^{m_\beta}_\sharp\beta}\Big) 
\Big\langle
\Big(U g^{\tilde f^{m_\beta}_\sharp\beta_l}_{\tilde f^{m_\beta}_\sharp\beta}\Big),\cdot
\Big\rangle_{f_\sharp\beta}
\Bigg],
\end{aligned}
\end{equation}
where $\xi_k=g^{f_\sharp \beta_k}_{f_\sharp \beta} - g^{\tilde f^{m_\beta}_\sharp\beta_k}_{\tilde f^{m_\beta}_\sharp\beta} $ and $B$ and $C$ are defined in \Cref{lem:BC_comparison}.
\begin{proof}
Recall that by definition
\begin{align}\label{eq:Fdagger_nonlinear_expansion_start}
F^\dagger_{N_{f_\sharp \beta}}&= \sum_{k=1}^K \sum_{l=1}^K (B^\dagger)_{kl} \left( g^{f_\sharp \beta_k}_{f_\sharp \beta} \right) \left\langle g^{f_\sharp \beta_l}_{f_\sharp \beta}, \cdot \right\rangle_{f_\sharp \beta}.
\end{align}
Using $\xi_k=g^{f_\sharp \beta_k}_{f_\sharp \beta} - g^{\tilde f^{m_\beta}_\sharp\beta_k}_{\tilde f^{m_\beta}_\sharp\beta} $,
we expand \eqref{eq:Fdagger_nonlinear_expansion_start} to obtain
\begin{equation*}
\begin{aligned}
F^\dagger_{N_{f_\sharp \beta}} = \sum_{k=1}^K \sum_{l=1}^K (B^\dagger)_{kl}&\left[ g^{\tilde f^{m_\beta}_\sharp\beta_k}_{\tilde f^{m_\beta}_\sharp\beta} \left\langle g^{\tilde f^{m_\beta}_\sharp\beta_l}_{\tilde f^{m_\beta}_\sharp\beta}, \cdot \right\rangle_{f_\sharp \beta} + g^{\tilde f^{m_\beta}_\sharp\beta_k}_{\tilde f^{m_\beta}_\sharp\beta} \left\langle \xi_l, \cdot \right\rangle_{f_\sharp \beta} \right. \left. + \xi_k \left\langle g^{\tilde f^{m_\beta}_\sharp\beta_l}_{\tilde f^{m_\beta}_\sharp\beta}, \cdot \right\rangle_{f_\sharp \beta} + \xi_k\left\langle \xi_l, \cdot \right\rangle_{f_\sharp \beta} \right].
\end{aligned}
\end{equation*}
Also, by definition
\begin{align}\label{eq:linear peuso inverse operator}
F^\dagger_{N_{\tilde f^{m_\beta}_\sharp\beta}} &= \sum_{k=1}^K \sum_{l=1}^K (C^\dagger)_{kl}  \ g^{\tilde f^{m_\beta}_\sharp\beta_k}_{\tilde f^{m_\beta}_\sharp\beta}  \  \left\langle g^{\tilde f^{m_\beta}_\sharp\beta_l}_{\tilde f^{m_\beta}_\sharp\beta}, \cdot \right\rangle_{\tilde f^{m_\beta}_\sharp\beta}.
\end{align}
Since this operator acts on $L^2(\tilde f^{m_\beta}_\sharp\beta)$, we conjugate it to
$L^2(f_\sharp\beta)$ by the operator $U$ defined in \eqref{eq:U_def}:
\begin{align*}
\tilde F^\dagger_{N_{\tilde f^{m_\beta}_\sharp\beta}}
&=UF^\dagger_{N_{\tilde f^{m_\beta}_\sharp\beta}} U^{-1} =
\sum_{k=1}^K \sum_{l=1}^K (C^\dagger)_{kl} 
\Big(U g^{\tilde f^{m_\beta}_\sharp\beta_k}_{\tilde f^{m_\beta}_\sharp\beta}\Big) 
\left\langle \Big(U g^{\tilde f^{m_\beta}_\sharp\beta_l}_{\tilde f^{m_\beta}_\sharp\beta}\Big), \cdot \right\rangle_{f_\sharp \beta}.
\end{align*}
Now we can decompose $F^\dagger_{N_{f_\sharp \beta}}  - \tilde F^\dagger_{N_{\tilde f^{m_\beta}_\sharp\beta}}$ as
\begin{align*}
F^\dagger_{N_{f_\sharp \beta}}  - \tilde F^\dagger_{N_{\tilde f^{m_\beta}_\sharp\beta}} &= \sum_{k=1}^K \sum_{l=1}^K (B^\dagger)_{kl}\left[ g^{\tilde f^{m_\beta}_\sharp\beta_k}_{\tilde f^{m_\beta}_\sharp\beta} \left\langle g^{\tilde f^{m_\beta}_\sharp\beta_l}_{\tilde f^{m_\beta}_\sharp\beta}, \cdot \right\rangle_{f_\sharp \beta} + g^{\tilde f^{m_\beta}_\sharp\beta_k}_{\tilde f^{m_\beta}_\sharp\beta} \left\langle \xi_l, \cdot \right\rangle_{f_\sharp \beta} \right. \left. + \xi_k \left\langle g^{\tilde f^{m_\beta}_\sharp\beta_l}_{\tilde f^{m_\beta}_\sharp\beta}, \cdot \right\rangle_{f_\sharp \beta} + \xi_k\left\langle \xi_l, \cdot \right\rangle_{f_\sharp \beta} \right]\\
& - \sum_{k=1}^K \sum_{l=1}^K (C^\dagger)_{kl}\left[ Ug^{\tilde f^{m_\beta}_\sharp\beta_k}_{\tilde f^{m_\beta}_\sharp\beta} \left\langle Ug^{\tilde f^{m_\beta}_\sharp\beta_l}_{\tilde f^{m_\beta}_\sharp\beta}, \cdot \right\rangle_{f_\sharp \beta}  \right]. 
\end{align*}
Adding and subtracting an intermediate term \eqref{eq:intermediate term}, we obtain
\begin{align}
F^\dagger_{N_{f_\sharp \beta}}  - \tilde F^\dagger_{N_{\tilde f^{m_\beta}_\sharp\beta}} &= \sum_{k=1}^K \sum_{l=1}^K (B^\dagger)_{kl}\left[ g^{\tilde f^{m_\beta}_\sharp\beta_k}_{\tilde f^{m_\beta}_\sharp\beta} \left\langle g^{\tilde f^{m_\beta}_\sharp\beta_l}_{\tilde f^{m_\beta}_\sharp\beta}, \cdot \right\rangle_{f_\sharp \beta} + g^{\tilde f^{m_\beta}_\sharp\beta_k}_{\tilde f^{m_\beta}_\sharp\beta} \left\langle \xi_l, \cdot \right\rangle_{f_\sharp \beta} \right. \left. + \xi_k \left\langle g^{\tilde f^{m_\beta}_\sharp\beta_l}_{\tilde f^{m_\beta}_\sharp\beta}, \cdot \right\rangle_{f_\sharp \beta} + \xi_k\left\langle \xi_l, \cdot \right\rangle_{f_\sharp \beta} \right] \notag\\
 &+\left[- \sum_{k=1}^K \sum_{l=1}^K (C^\dagger)_{kl} 
g^{\tilde f^{m_\beta}_\sharp\beta_k}_{\tilde f^{m_\beta}_\sharp\beta} 
\Big\langle
g^{\tilde f^{m_\beta}_\sharp\beta_l}_{\tilde f^{m_\beta}_\sharp\beta},\cdot
\Big\rangle_{f_\sharp\beta}+
\sum_{k=1}^K \sum_{l=1}^K (C^\dagger)_{kl} 
g^{\tilde f^{m_\beta}_\sharp\beta_k}_{\tilde f^{m_\beta}_\sharp\beta} 
\Big\langle
g^{\tilde f^{m_\beta}_\sharp\beta_l}_{\tilde f^{m_\beta}_\sharp\beta},\cdot
\Big\rangle_{f_\sharp\beta}\right]\label{eq:intermediate term}\\
& - \sum_{k=1}^K \sum_{l=1}^K (C^\dagger)_{kl}\left[ Ug^{\tilde f^{m_\beta}_\sharp\beta_k}_{\tilde f^{m_\beta}_\sharp\beta} \left\langle Ug^{\tilde f^{m_\beta}_\sharp\beta_l}_{\tilde f^{m_\beta}_\sharp\beta}, \cdot \right\rangle_{f_\sharp \beta}  \right] \notag. 
\end{align}
Reorganizing, we obtain
\begin{align*}
F^\dagger_{N_{f_\sharp \beta}}  - \tilde F^\dagger_{N_{\tilde f^{m_\beta}_\sharp\beta}}
&=
\sum_{k=1}^K \sum_{l=1}^K \Big((B^\dagger)_{kl}-(C^\dagger)_{kl}\Big) 
g^{\tilde f^{m_\beta}_\sharp\beta_k}_{\tilde f^{m_\beta}_\sharp\beta} 
\Big\langle
g^{\tilde f^{m_\beta}_\sharp\beta_l}_{\tilde f^{m_\beta}_\sharp\beta},\cdot
\Big\rangle_{f_\sharp\beta}
\\
&+
\sum_{k=1}^K \sum_{l=1}^K (B^\dagger)_{kl}\Bigg[
g^{\tilde f^{m_\beta}_\sharp\beta_k}_{\tilde f^{m_\beta}_\sharp\beta} 
\big\langle \xi_l,\cdot\big\rangle_{f_\sharp\beta}
+
\xi_k 
\Big\langle
g^{\tilde f^{m_\beta}_\sharp\beta_l}_{\tilde f^{m_\beta}_\sharp\beta},\cdot
\Big\rangle_{f_\sharp\beta}
+
\xi_k \big\langle \xi_l,\cdot\big\rangle_{f_\sharp\beta}
\Bigg]
\\
&+
\sum_{k=1}^K \sum_{l=1}^K (C^\dagger)_{kl}\Bigg[
g^{\tilde f^{m_\beta}_\sharp\beta_k}_{\tilde f^{m_\beta}_\sharp\beta} 
\Big\langle
g^{\tilde f^{m_\beta}_\sharp\beta_l}_{\tilde f^{m_\beta}_\sharp\beta},\cdot
\Big\rangle_{f_\sharp\beta}
-
\Big(U g^{\tilde f^{m_\beta}_\sharp\beta_k}_{\tilde f^{m_\beta}_\sharp\beta}\Big) 
\Big\langle
\Big(U g^{\tilde f^{m_\beta}_\sharp\beta_l}_{\tilde f^{m_\beta}_\sharp\beta}\Big),\cdot
\Big\rangle_{f_\sharp\beta}
\Bigg]. \qedhere
\end{align*}
\end{proof}
\end{lemma}

\begin{definition}[Hellinger distance, Definition~6.35 of \cite{Stuart2010inverse}]\label{def:hellinger_standard} 
Let $P$ and $Q$ be probability measures on $\mathbb{R}^n$.
Assume there exists a probability measure $\lambda$ on $\mathbb{R}^n$ such that $P$ and $Q$ are absolutely continuous with respect to $\lambda$.
Write
\begin{equation*}
dP = p d\lambda,
\qquad
dQ = q d\lambda,
\end{equation*}
where $p=\frac{dP}{d\lambda}$ and $q=\frac{dQ}{d\lambda}$ are Radon-Nikodym derivatives.
The Hellinger distance between $P$ and $Q$ is defined by
\begin{equation}\label{eq:hellinger_standard_lambda}
H^2(P,Q)
=
\frac12\int_{\mathbb{R}^n}\bigl(\sqrt{p(y)}-\sqrt{q(y)}\bigr)^2 d\lambda(y).
\end{equation}
This definition does not depend on the choice of $\lambda$ as long as Radon--Nikodym derivatives are well defined.
\end{definition}

\begin{lemma}\label{lemma:Radon-Nikodym well define}
Let $\beta=\mathcal{N}(m_\beta,\Sigma)$ on $\mathbb{R}^n$, where $\beta$ is non-degenerated $(\Sigma \text{ positive definite)}$.
Let $f:\mathbb{R}^n\to\mathbb{R}^n$ be a smooth ($C^\infty$) diffeomorphism, denote by $\tilde f^{m_\beta}$ its affine approximation at $m_{\beta}$ \eqref{eq:affine-approx}. Then $\tilde f^{m_\beta}_\sharp\beta$ is absolutely continuous with respect to $f_\sharp\beta$. Moreover, the Radon-Nikodym derivative
\begin{equation*}
    w=\frac{d(\tilde f^{m_\beta}_\sharp\beta)}{d(f_\sharp\beta)}
\end{equation*}
is strictly positive.

\begin{proof}
Since $\beta$ is non-degenerate, it has a density $\rho_\beta (x)$, which is strictly positive $\forall x\in\mathbb{R}^n$.
Because $f$ is a $C^\infty$ diffeomorphism, $f^{-1}$ is $C^\infty$ as well and we have
\begin{equation*}
    \bigl|\det J_{f^{-1}}(y)\bigr|>0
\qquad \forall y\in\mathbb R^n.
\end{equation*}
For any Borel set $A\subset\mathbb{R}^n$, by change of variables $y=f(x)$, $dx=\bigl|\det J_{f^{-1}}(y)\bigr| dy$,
\begin{align*}
(f_\sharp\beta)(A)
&=\beta\bigl(f^{-1}(A)\bigr) \\
&=\int_{f^{-1}(A)} \rho_\beta(x)dx\\
&=
\int_A \rho_\beta(f^{-1}(y))
\bigl|\det J_{f^{-1}}(y)\bigr|dy,
\end{align*}
which is strictly positive. Moreover,  $\tilde f^{m_\beta}$ is an invertible affine map, hence a smooth diffeomorphism.
Repeating the same change of variables method (with $y=\tilde f^{m_\beta}(x)$) yields that
$\tilde f^{m_\beta}_\sharp\beta$ also has a strictly positive density on $\mathbb{R}^n$.

Assume that $A\subset\mathbb{R}^n$ is a Borel set such that $(f_\sharp\beta)(A)=0$. Since $f_\sharp\beta$ has a strictly positive density, this implies that $\int_A dy=0$. It follows that $(\tilde f^{m_\beta}_\sharp\beta)(A)=0$, and hence $\tilde f^{m_\beta}_\sharp\beta$ is absolutely continuous with respect to $f_\sharp\beta$. Since both $f_\sharp\beta$ and $\tilde f^{m_\beta}_\sharp\beta$ have strictly positive densities,
\begin{equation*}
    w=\frac{d(\tilde f^{m_\beta}_\sharp\beta)}{d(f_\sharp\beta)}
\end{equation*}
is strictly positive.
\end{proof}
\end{lemma}

\begin{lemma}\label{lemma:main_theorem_lemma}
Let $\alpha = \mathcal{N}(m_\alpha, \Sigma), \beta = \mathcal{N}(m_\beta, \Sigma)$ be Gaussian distributions on $\mathbb R^n$. For $\delta>0$, let $\beta_i=\mathcal N(m_{\beta_i},\Sigma), i=1,\dots,K,$ where the means $m_{\beta_i}$ are uniformly distributed in ${B}(m_\beta,\delta)$.
 Let $f:\mathbb{R}^n\to\mathbb{R}^n$ satisfy \Cref{assump:regularity}, and assume that $K$, $\delta$ satisfy \Cref{Assumption: empirical sampling}. Let $\tilde f^{m_\beta}$ be the linear approximation of $f$ at $m_\beta$ \eqref{eq:affine-approx}. Denote by
\begin{equation*}
N_{\tilde f^{m_\beta}_\sharp\beta} = \Big\{g^{\tilde f^{m_\beta}_\sharp\beta_i}_{\tilde f^{m_\beta}_\sharp\beta}\Big\}_{i=1}^K,
\qquad
N_{f_\sharp\beta} = \Big\{g^{f_\sharp \beta_i}_{f_\sharp \beta}\Big\}_{i=1}^K,
\end{equation*}
the sets containing OT displacement functions \eqref{eq:displacement_mu_mui} from $\tilde f^{m_\beta}_\sharp\beta$ to $\tilde f^{m_\beta}_\sharp\beta_i$ and from ${f}_\sharp \beta$ to ${f}_\sharp \beta_i$, respectively. Let $F^\dagger_{N_{f_\sharp\beta}}$, $F^\dagger_{N_{\tilde f^{m_\beta}_\sharp\beta}}$
be the pseudo-inverse empirical covariance operators \eqref{eq:Fdagger_Nalpha} defined by $N_{f_\sharp\beta}$, and $N_{\tilde f^{m_\beta}_\sharp\beta}$, respectively.
Suppose that there exists $p\geq 2$ such that \Cref{displacement function assumption} holds for $\beta$ with $r_\beta \geq \max\{\|m_\alpha - m_\beta\|, \delta\}$. If $p=2$, suppose in addition that  \Cref{assuption:dispacement_local} holds for $\beta$. Then
\begin{align*}
&\big\langle g^{f_\sharp\alpha}_{f_\sharp\beta}, F^\dagger_{N_{f_\sharp\beta}}\big(g^{f_\sharp\alpha}_{f_\sharp\beta}\big)\big\rangle_{f_\sharp\beta}  - \big\langle g^{\tilde f^{m_\beta}_\sharp\alpha}_{\tilde f^{m_\beta}_\sharp\beta},  F^\dagger_{N_{\tilde f^{m_\beta}_\sharp\beta}} \big(g^{\tilde f^{m_\beta}_\sharp\alpha}_{\tilde f^{m_\beta}_\sharp\beta}\big) \big\rangle_{\tilde f^{m_\beta}_\sharp\beta}\\
& = O(\delta^{p/2-3}\|m_\beta-m_\alpha\|^2)
+
O \left(\frac{\|m_\beta-m_\alpha\|^{2}}{\delta^2} H(\tilde f^{m_\beta}_\sharp\beta,\ f_\sharp\beta)\right),
\end{align*}
 where $H$ denotes the Hellinger distance (\Cref{def:hellinger_standard}).

\begin{proof}
The proof compares the nonlinear Wasserstein Mahalanobis distance with the
corresponding distance obtained after replacing $f$ by its linearization:
\begin{equation*}
\left\langle
g^{f_\sharp\alpha}_{f_\sharp\beta},
F^\dagger_{N_{f_\sharp\beta}}
\left(g^{f_\sharp\alpha}_{f_\sharp\beta}\right)
\right\rangle_{f_\sharp\beta}
\quad
\text{and}
\quad
\left\langle
g^{\tilde f^{m_\beta}_\sharp\alpha}_{\tilde f^{m_\beta}_\sharp\beta},
F^\dagger_{N_{\tilde f^{m_\beta}_\sharp\beta}}
\left(
g^{\tilde f^{m_\beta}_\sharp\alpha}_{\tilde f^{m_\beta}_\sharp\beta}
\right)
\right\rangle_{\tilde f^{m_\beta}_\sharp\beta}.
\end{equation*}
The difference in this comparison comes from two sources: the displacement difference
between $g^{f_\sharp\alpha}_{f_\sharp\beta}$ and
$g^{\tilde f^{m_\beta}_\sharp\alpha}_{\tilde f^{m_\beta}_\sharp\beta}$, and the
operator difference between $F^\dagger_{N_{f_\sharp\beta}}$ and $F^\dagger_{N_{\tilde f^{m_\beta}_\sharp\beta}}$, which were already controlled in \Cref{displacement function assumption} and by \Cref{operator differece lemma}. The result can therefore be derived directly as follows.

Define the displacement error
\begin{equation*}
\epsilon_\beta
=
g^{f_\sharp\alpha}_{f_\sharp\beta}
-
g^{\tilde f^{m_\beta}_\sharp\alpha}_{\tilde f^{m_\beta}_\sharp\beta}.
\end{equation*}
Since the linearized displacement $g^{\tilde{f}^{m_\beta}_\sharp\alpha}_{\tilde{f}^{m_\beta}_\sharp\beta}$ is constant, we treat it as the same constant function in $L^2(f_\sharp\beta)$. Using \Cref{displacement function assumption} we obtain
$
   \|\epsilon_\beta\|_{L^2(f_\sharp \beta)}^2 \leq M,
$ 
where $M$ has order $O(\|m_\beta - m_\alpha  \|^p)$. By \Cref{operator differece lemma}, we can write
\begin{equation*}
F^\dagger_{N_{f_\sharp\beta}}
=
\tilde F^\dagger_{N_{\tilde f^{m_\beta}_\sharp\beta}}
+
\Delta_\beta,
\end{equation*}
where
\begin{equation*}
\|\Delta_\beta\|_{op}
=
O(\delta^{p/2-3})
+
O(\delta^{-2} 
H(\tilde f^{m_\beta}_\sharp\beta,f_\sharp\beta)).
\end{equation*}
Using $g^{f_\sharp\alpha}_{f_\sharp\beta}
=
g^{\tilde f^{m_\beta}_\sharp\alpha}_{\tilde f^{m_\beta}_\sharp\beta}
+
\epsilon_\beta$ and $\tilde F^\dagger_{N_{\tilde f^{m_\beta}_\sharp\beta}}
+
\Delta_\beta$, we expand
\begin{align}\nonumber
 &\langle g^{f_\sharp \alpha}_{f_\sharp \beta}, F^\dagger_{N_{f_\sharp \beta}}(g^{f_\sharp \alpha}_{f_\sharp \beta})\rangle_{f_\sharp \beta}
= \langle g^{\tilde{f}^{m_\beta}_\sharp\alpha}_{\tilde{f}^{m_\beta}_\sharp\beta} + \epsilon_\beta, (\tilde F^\dagger_{N_{\tilde f^{m_\beta}_\sharp\beta}}+\Delta_\beta)(g^{\tilde{f}^{m_\beta}_\sharp\alpha}_{\tilde{f}^{m_\beta}_\sharp\beta} + \epsilon_\beta) \rangle_{f_\sharp \beta} \\ \label{eq: nonlinear inner product expansion}
&= \langle g^{\tilde{f}^{m_\beta}_\sharp\alpha}_{\tilde{f}^{m_\beta}_\sharp\beta}, \tilde F^\dagger_{N_{\tilde f^{m_\beta}_\sharp\beta}}(g^{\tilde{f}^{m_\beta}_\sharp\alpha}_{\tilde{f}^{m_\beta}_\sharp\beta})\rangle_{f_\sharp \beta} 
+ \langle g^{\tilde{f}^{m_\beta}_\sharp\alpha}_{\tilde{f}^{m_\beta}_\sharp\beta},\tilde F^\dagger_{N_{\tilde f^{m_\beta}_\sharp\beta}}(\epsilon_\beta)\rangle_{f_\sharp \beta}
+\langle \epsilon_\beta,\tilde F^\dagger_{N_{\tilde f^{m_\beta}_\sharp\beta}}(g^{\tilde{f}^{m_\beta}_\sharp\alpha}_{\tilde{f}^{m_\beta}_\sharp\beta})\rangle_{f_\sharp \beta} \\ \nonumber
&+ \langle \epsilon_\beta, \tilde F^\dagger_{N_{\tilde f^{m_\beta}_\sharp\beta}}(\epsilon_\beta)\rangle_{f_\sharp \beta}
+ \langle g^{\tilde{f}^{m_\beta}_\sharp\alpha}_{\tilde{f}^{m_\beta}_\sharp\beta},\Delta_\beta(g^{\tilde{f}^{m_\beta}_\sharp\alpha}_{\tilde{f}^{m_\beta}_\sharp\beta})\rangle_{f_\sharp \beta} + \langle \epsilon_\beta,\Delta_\beta(\epsilon_\beta)\rangle_{f_\sharp \beta}\\ \nonumber
&+ \langle \epsilon_\beta,\Delta_\beta(g^{\tilde{f}^{m_\beta}_\sharp\alpha}_{\tilde{f}^{m_\beta}_\sharp\beta})\rangle_{f_\sharp \beta}
+\langle g^{\tilde{f}^{m_\beta}_\sharp\alpha}_{\tilde{f}^{m_\beta}_\sharp\beta},\Delta_\beta(\epsilon_\beta)\rangle_{f_\sharp \beta}    
\end{align}
The quantity we want to compare with \eqref{eq: nonlinear inner product expansion} is 
$
    \langle
g^{\tilde f^{m_\beta}_\sharp\alpha}_{\tilde f^{m_\beta}_\sharp\beta},
F^\dagger_{N_{\tilde f^{m_\beta}_\sharp\beta}}
(
g^{\tilde f^{m_\beta}_\sharp\alpha}_{\tilde f^{m_\beta}_\sharp\beta}
)
\rangle_{\tilde f^{m_\beta}_\sharp\beta}.
$
By the isometry identity of $U$, we have 
\begin{equation}\label{eq:linear inner product equality}
\big\langle g^{\tilde f^{m_\beta}_\sharp\alpha}_{\tilde f^{m_\beta}_\sharp\beta}, 
F^\dagger_{N_{\tilde f^{m_\beta}_\sharp\beta}}
\big(g^{\tilde f^{m_\beta}_\sharp\alpha}_{\tilde f^{m_\beta}_\sharp\beta}\big)
\big\rangle_{\tilde f^{m_\beta}_\sharp\beta} = \big\langle U g^{\tilde f^{m_\beta}_\sharp\alpha}_{\tilde f^{m_\beta}_\sharp\beta},\ 
\tilde F^\dagger_{N_{\tilde f^{m_\beta}_\sharp\beta}}
\big(U g^{\tilde f^{m_\beta}_\sharp\alpha}_{\tilde f^{m_\beta}_\sharp\beta}\big)
\big\rangle_{f_\sharp\beta}.
\end{equation}  
Subtracting the LHS in \eqref{eq:linear inner product equality} from the LHS in \eqref{eq: nonlinear inner product expansion} and the RHS in \eqref{eq:linear inner product equality} from the RHS \eqref{eq: nonlinear inner product expansion}, we can write 

\begin{align*}
    \big\langle g^{f_\sharp\alpha}_{f_\sharp\beta}, F^\dagger_{N_{f_\sharp\beta}}\big(g^{f_\sharp\alpha}_{f_\sharp\beta}\big)\big\rangle_{f_\sharp\beta}  - \big\langle g^{\tilde f^{m_\beta}_\sharp\alpha}_{\tilde f^{m_\beta}_\sharp\beta},  F^\dagger_{N_{\tilde f^{m_\beta}_\sharp\beta}} \big(g^{\tilde f^{m_\beta}_\sharp\alpha}_{\tilde f^{m_\beta}_\sharp\beta}\big) \big\rangle_{\tilde f^{m_\beta}_\sharp\beta} = R_{\beta,1} + R_{\beta,2},
\end{align*}
where
\begin{equation}\label{eq:R1 diff}
    R_{\beta,1} = \big\langle g^{\tilde{f}^{m_\beta}_\sharp\alpha}_{\tilde{f}^{m_\beta}_\sharp\beta}, \tilde F^\dagger_{N_{\tilde f^{m_\beta}_\sharp\beta}}(g^{\tilde{f}^{m_\beta}_\sharp\alpha}_{\tilde{f}^{m_\beta}_\sharp\beta})\big\rangle_{f_\sharp \beta}  - \big\langle U g^{\tilde f^{m_\beta}_\sharp\alpha}_{\tilde f^{m_\beta}_\sharp\beta},\ 
\tilde F^\dagger_{N_{\tilde f^{m_\beta}_\sharp\beta}}
\big(U g^{\tilde f^{m_\beta}_\sharp\alpha}_{\tilde f^{m_\beta}_\sharp\beta}\big)
\big\rangle_{f_\sharp\beta},
\end{equation}
and
\begin{align}
    R_{\beta,2} &= \langle g^{\tilde{f}^{m_\beta}_\sharp\alpha}_{\tilde{f}^{m_\beta}_\sharp\beta},\tilde F^\dagger_{N_{\tilde f^{m_\beta}_\sharp\beta}}(\epsilon_\beta)\rangle_{f_\sharp \beta}
+\langle \epsilon_\beta,\tilde F^\dagger_{N_{\tilde f^{m_\beta}_\sharp\beta}}(g^{\tilde{f}^{m_\beta}_\sharp\alpha}_{\tilde{f}^{m_\beta}_\sharp\beta})\rangle_{f_\sharp \beta}\label{inner product expansion1} \\
&+ \langle \epsilon_\beta, \tilde F^\dagger_{N_{\tilde f^{m_\beta}_\sharp\beta}}(\epsilon_\beta)\rangle_{f_\sharp \beta}
+ \langle g^{\tilde{f}^{m_\beta}_\sharp\alpha}_{\tilde{f}^{m_\beta}_\sharp\beta},\Delta_\beta(g^{\tilde{f}^{m_\beta}_\sharp\alpha}_{\tilde{f}^{m_\beta}_\sharp\beta})\rangle_{f_\sharp \beta} \label{inner product expansion2} \\
&+ \langle \epsilon_\beta,\Delta_\beta(\epsilon_\beta)\rangle_{f_\sharp \beta} + \langle \epsilon_\beta,\Delta_\beta(g^{\tilde{f}^{m_\beta}_\sharp\alpha}_{\tilde{f}^{m_\beta}_\sharp\beta})\rangle_{f_\sharp \beta}
+\langle g^{\tilde{f}^{m_\beta}_\sharp\alpha}_{\tilde{f}^{m_\beta}_\sharp\beta},\Delta_\beta(\epsilon_\beta)\rangle_{f_\sharp \beta}. \label{inner product expansion3}
\end{align}

Before we bound $R_{\beta,1}$ and $R_{\beta,2}$, we first bound each of the following terms:  
$
g^{\tilde{f}^{m_\beta}_\sharp\alpha}_{\tilde{f}^{m_\beta}_\sharp\beta},
\tilde F^\dagger_{N_{\tilde f^{m_\beta}_\sharp\beta}},
\text{ and }
\Delta_\beta .
$
Recall that
\begin{equation*}
g^{\tilde{f}^{m_\beta}_\sharp\alpha}_{\tilde{f}^{m_\beta}_\sharp\beta}  =   J_f(m_\beta)(  m_\alpha - m_\beta  ).
\end{equation*}
Therefore
\begin{equation}\label{eq:Bound_item1}
\bigl\|g^{\tilde{f}^{m_\beta}_\sharp\alpha}_{\tilde{f}^{m_\beta}_\sharp\beta}
\bigr\|_{L^2(f_\sharp\beta)}
\le
\|J_f(m_\beta)\|_2\|m_\alpha-m_\beta\|.
\end{equation}
By definition, for any 
$\phi\in L^2(\tilde f^{m_\beta}_\sharp\beta)$, with $\|\phi\|_{L^2(\tilde f^{m_\beta}_\sharp\beta)}  =1$, we have
\begin{align*}
\left\|
F^\dagger_{N_{\tilde f^{m_\beta}_\sharp\beta}}\phi
\right\|_{L^2(\tilde f^{m_\beta}_\sharp\beta)}
&=
\left\|
\sum_{k=1}^K \sum_{l=1}^K
(C^\dagger)_{kl}
g^{\tilde f^{m_\beta}_\sharp\beta_k}_{\tilde f^{m_\beta}_\sharp\beta}
\left\langle
g^{\tilde f^{m_\beta}_\sharp\beta_l}_{\tilde f^{m_\beta}_\sharp\beta},
\phi
\right\rangle_{\tilde f^{m_\beta}_\sharp\beta}
\right\|_{L^2(\tilde f^{m_\beta}_\sharp\beta)} \\
&\le
\sum_{k=1}^K \sum_{l=1}^K
|(C^\dagger)_{kl}|
\left\|
g^{\tilde f^{m_\beta}_\sharp\beta_k}_{\tilde f^{m_\beta}_\sharp\beta}
\right\|_{L^2(\tilde f^{m_\beta}_\sharp\beta)}
\left\|
g^{\tilde f^{m_\beta}_\sharp\beta_l}_{\tilde f^{m_\beta}_\sharp\beta}
\right\|_{L^2(\tilde f^{m_\beta}_\sharp\beta)}.
\end{align*}
Since
\begin{equation*}
g^{\tilde f^{m_\beta}_\sharp\beta_i}_{\tilde f^{m_\beta}_\sharp\beta}
=
J_f(m_\beta)(m_{\beta_i}-m_\beta),
\end{equation*}
and $\|m_{\beta_i}-m_\beta\|\le \delta$, we have
\begin{equation*}
\left\|
g^{\tilde f^{m_\beta}_\sharp\beta_i}_{\tilde f^{m_\beta}_\sharp\beta}
\right\|
\le
\|J_f(m_\beta)\|_2 \delta .
\end{equation*}
Therefore,
\begin{align*}
\left\|
F^\dagger_{N_{\tilde f^{m_\beta}_\sharp\beta}}\phi
\right\|_{L^2(\tilde f^{m_\beta}_\sharp\beta)}
&\le
\|J_f(m_\beta)\|_2^2\delta^2
\left(
\sum_{k=1}^K \sum_{l=1}^K |(C^\dagger)_{kl}|
\right)
\\
&\le
\|J_f(m_\beta)\|_2^2\delta^2
K\|C^\dagger\|_F \\
&\le
K^{3/2}
\|J_f(m_\beta)\|_2^2\delta^2
\|C^\dagger\|_2
.
\end{align*}
Using \Cref{lemma:bounds for C}, we obtain
\begin{align*}
\left\|
F^\dagger_{N_{\tilde f^{m_\beta}_\sharp\beta}}\phi
\right\|_{L^2(\tilde f^{m_\beta}_\sharp\beta)}
&\le
K^{3/2}
\|J_f(m_\beta)\|_2^2\delta^2
\frac{(n+2)^2}
{K\delta^4[\sigma_{\min}(J_f(m_\beta))]^4}
 \\
&=
\frac{K^{1/2}(n+2)^2\|J_f(m_\beta)\|_2^2}
{\delta^2[\sigma_{\min}(J_f(m_\beta))]^4}
.
\end{align*}
By \eqref{eq:op_norm_invariance} we get the following.
\begin{equation}\label{eq:Bound_item2}
\bigl\|
\tilde F^\dagger_{N_{\tilde f^{m_\beta}_\sharp\beta}}
\bigr\|_{\mathrm{op}}
\le
\frac{K^{1/2}(n+2)^2\|J_f(m_\beta)\|_2^2}
{\delta^2[\sigma_{\min}(J_f(m_\beta))]^4}.
\end{equation}
From \Cref{operator differece lemma},
\begin{equation}\label{eq:Bound_item3}
    \|\Delta_\beta\|_{op} \leq \kappa
\end{equation}
where $\kappa$ has order $O(\delta^{p/2-3})
+
O\big(\delta^{-2}  H(\tilde f^{m_\beta}_\sharp\beta,\ f_\sharp\beta)\big)$. 

We now bound the three groups of terms in 
\eqref{inner product expansion1}, \eqref{inner product expansion2}, and 
\eqref{inner product expansion3} in $R_{\beta,2}$ separately using \eqref{eq:Bound_item1} \eqref{eq:Bound_item2}, \eqref{eq:Bound_item3}.  First, the two terms in \eqref{inner product expansion1} are bounded by
\begin{align*}
\bigl\langle g^{\tilde{f}^{m_\beta}_\sharp\alpha}_{\tilde{f}^{m_\beta}_\sharp\beta},\tilde F^\dagger_{N_{\tilde f^{m_\beta}_\sharp\beta}}(\epsilon_\beta)
\bigr\rangle_{f_\sharp \beta}
+
\bigl\langle \epsilon_\beta,\tilde F^\dagger_{N_{\tilde f^{m_\beta}_\sharp\beta}}(g^{\tilde{f}^{m_\beta}_\sharp\alpha}_{\tilde{f}^{m_\beta}_\sharp\beta})
\bigr\rangle_{f_\sharp \beta}
&\le 2\|\tilde F^\dagger_{N_{\tilde f^{m_\beta}_\sharp\beta}}\|_{op} 
      \|g^{\tilde{f}^{m_\beta}_\sharp\alpha}_{\tilde{f}^{m_\beta}_\sharp\beta}\|_{L^2(f_\sharp \beta)} 
      \|\epsilon_\beta\|_{L^2(f_\sharp \beta)}\\
&\le \frac{2K^{1/2}(n+2)^2\|J_f(m_\beta)\|_2^3}
{\delta^2[\sigma_{\min}(J_f(m_\beta))]^4} 
      \|m_\alpha - m_\beta\|\sqrt{M}.
\end{align*}
Next, the first term in \eqref{inner product expansion2} satisfies
\begin{align*}   
\bigl\langle \epsilon_\beta,\tilde F^\dagger_{N_{\tilde f^{m_\beta}_\sharp\beta}}(\epsilon_\beta)
\bigr\rangle_{f_\sharp \beta}
&\le \|\tilde F^\dagger_{N_{\tilde f^{m_\beta}_\sharp\beta}}\|_{op} 
      \|\epsilon_\beta\|_{L^2(f_\sharp \beta)}^{2}
=\frac{K^{1/2}(n+2)^2\|J_f(m_\beta)\|_2^2}
{\delta^2[\sigma_{\min}(J_f(m_\beta))]^4}
      M
\end{align*}
The second term in \eqref{inner product expansion2} is bounded by
\begin{align*}
\bigl\langle g^{\tilde{f}^{m_\beta}_\sharp\alpha}_{\tilde{f}^{m_\beta}_\sharp\beta},\Delta_\beta(g^{\tilde{f}^{m_\beta}_\sharp\alpha}_{\tilde{f}^{m_\beta}_\sharp\beta})
\bigr\rangle_{f_\sharp \beta}
&\le \kappa \|g^{\tilde{f}^{m_\beta}_\sharp\alpha}_{\tilde{f}^{m_\beta}_\sharp\beta}\|_{L^2(f_\sharp \beta)}^{2}
\le \kappa \|J_f(m_\beta)\|_{2}^{2} \|m_\alpha - m_\beta\|^{2}.
\end{align*}
Finally, we bound the three terms in \eqref{inner product expansion3}. The first term in \eqref{inner product expansion3} is bounded by
\begin{align*}
\bigl\langle \epsilon_\beta,\Delta_\beta(\epsilon_\beta)
\bigr\rangle_{f_\sharp \beta}
&\le \kappa M.
\end{align*}
The remaining two terms in \eqref{inner product expansion3} satisfy
\begin{align*}
\bigl\langle g^{\tilde{f}^{m_\beta}_\sharp\alpha}_{\tilde{f}^{m_\beta}_\sharp\beta},\Delta_\beta(\epsilon_\beta)
\bigr\rangle_{f_\sharp \beta}
+
\bigl\langle \epsilon_\beta,\Delta_\beta(g^{\tilde{f}^{m_\beta}_\sharp\alpha}_{\tilde{f}^{m_\beta}_\sharp\beta})
\bigr\rangle_{f_\sharp \beta}
&\leq 2\kappa  
      \|g^{\tilde{f}^{m_\beta}_\sharp\alpha}_{\tilde{f}^{m_\beta}_\sharp\beta}\|_{L^2(f_\sharp \beta)} 
      \|\epsilon_\beta\|_{L^2(f_\sharp \beta)}\\
&\le 2\kappa \|J_f(m_\beta)\|_{2} 
         \|m_\alpha - m_\beta\| \sqrt{M}.
\end{align*}
Combining the bounds from 
\eqref{inner product expansion1}, \eqref{inner product expansion2}, and 
\eqref{inner product expansion3}, we obtain
\begin{align*}
    |R_{\beta,2}| 
      & \le 
\frac{K^{1/2}(n+2)^2\|J_f(m_\beta)\|_2^2}
{\delta^2[\sigma_{\min}(J_f(m_\beta))]^4}
      M
 + 
\frac{2K^{1/2}(n+2)^2\|J_f(m_\beta)\|_2^3}
{\delta^2[\sigma_{\min}(J_f(m_\beta))]^4} 
      \|m_\alpha - m_\beta\|\sqrt{M}\\
 &+ 
\kappa \Bigl(
      M
     +2\|J_f(m_\beta)\|_{2}\|m_\alpha-m_\beta\|\sqrt{M}
     +\|J_f(m_\beta)\|_{2}^{2}\|m_\alpha-m_\beta\|^{2}
 \Bigr),
\end{align*}
which has order 

\begin{equation*}
\begin{aligned}
|R_{\beta,2}|
&=
O \left(\delta^{-2} \|m_\beta-m_\alpha\|^{p}\right)
+
O\left(\delta^{-2} \|m_\beta-m_\alpha\|^{1+p/2}\right)
\\
&+
O \left(\delta^{p/2-3} \|m_\beta-m_\alpha\|^{p}\right)
+
O \left(\delta^{-2} \|m_\beta-m_\alpha\|^{p}  H(\tilde f^{m_\beta}_\sharp\beta,\ f_\sharp\beta)\right)
\\
&+
O \left(\delta^{p/2-3} \|m_\beta-m_\alpha\|^{1+p/2}\right)
+
O \left(\delta^{-2} \|m_\beta-m_\alpha\|^{1+p/2}  H(\tilde f^{m_\beta}_\sharp\beta,\ f_\sharp\beta)\right)
\\
&+
O \left(\delta^{p/2-3} \|m_\beta-m_\alpha\|^{2}\right)
+
O \left(\delta^{-2} \|m_\beta-m_\alpha\|^{2}  H(\tilde f^{m_\beta}_\sharp\beta,\ f_\sharp\beta)\right).
\end{aligned}
\end{equation*}
For $p\ge2$, the leading order terms are $O(\delta^{p/2-3}\|m_\beta-m_\alpha\|^2)
+
O \left(\frac{\|m_\beta-m_\alpha\|^{2}}{\delta^2} H(\tilde f^{m_\beta}_\sharp\beta,\ f_\sharp\beta)\right)$.\footnote{When $p=2$, all terms yield equal order. When $p>2$, let
$h=\|m_\beta-m_\alpha\|$. The Hellinger-related terms satisfy
$h^p\le h ^{1+p/2}\le h ^2$, so the dominant Hellinger-related term is
$O(\delta^{-2}h^2H(\tilde f^{m_\beta}_\sharp\beta,f_\sharp\beta))$.
For the non-Hellinger terms, $\delta^{p/2-3}h^2$ dominates.}

Now we bound $|R_{\beta,1}|$ \eqref{eq:R1 diff}. Since $\bigl\|Ug^{\tilde f^{m_\beta}_\sharp\alpha}_{\tilde f^{m_\beta}_\sharp\beta}\bigr\|_{L^2(f_\sharp\beta)}
=
\bigl\|g^{\tilde f^{m_\beta}_\sharp\alpha}_{\tilde f^{m_\beta}_\sharp\beta}\bigr\|_{L^2(\tilde f^{m_\beta}_\sharp\beta)}=
\bigl\|g^{\tilde f^{m_\beta}_\sharp\alpha}_{\tilde f^{m_\beta}_\sharp\beta}\bigr\|_{L^2(f_\sharp\beta)}$, we have
\begin{align*}
\label{eq:transport_gap_bound_simplified}
|R_{\beta,1}| &= \Big|
\big\langle 
g^{\tilde f^{m_\beta}_\sharp\alpha}_{\tilde f^{m_\beta}_\sharp\beta},\ 
\tilde F^\dagger_{N_{\tilde f^{m_\beta}_\sharp\beta}}
\big(g^{\tilde f^{m_\beta}_\sharp\alpha}_{\tilde f^{m_\beta}_\sharp\beta}\big)
\big\rangle_{f_\sharp\beta}
-
\big\langle U g^{\tilde f^{m_\beta}_\sharp\alpha}_{\tilde f^{m_\beta}_\sharp\beta},\ 
\tilde F^\dagger_{N_{\tilde f^{m_\beta}_\sharp\beta}}
\big(U g^{\tilde f^{m_\beta}_\sharp\alpha}_{\tilde f^{m_\beta}_\sharp\beta}\big)
\big\rangle_{f_\sharp\beta}
\Big| \\
&\le
\bigl\|\tilde F^\dagger_{N_{\tilde f^{m_\beta}_\sharp\beta}}\bigr\|_{op} 
\Big(
\bigl\|g^{\tilde f^{m_\beta}_\sharp\alpha}_{\tilde f^{m_\beta}_\sharp\beta}\bigr\|_{L^2(f_\sharp\beta)}
+
\bigl\|U g^{\tilde f^{m_\beta}_\sharp\alpha}_{\tilde f^{m_\beta}_\sharp\beta}\bigr\|_{L^2(f_\sharp\beta)}
\Big) 
\bigl\|U g^{\tilde f^{m_\beta}_\sharp\alpha}_{\tilde f^{m_\beta}_\sharp\beta}
-
g^{\tilde f^{m_\beta}_\sharp\alpha}_{\tilde f^{m_\beta}_\sharp\beta}
\bigr\|_{L^2(f_\sharp\beta)}\\
&\le
\bigl\|\tilde F^\dagger_{N_{\tilde f^{m_\beta}_\sharp\beta}}\bigr\|_{op} 
\Big(
2 \|g^{\tilde f^{m_\beta}_\sharp\alpha}_{\tilde f^{m_\beta}_\sharp\beta}\bigr\|_{L^2(f_\sharp\beta)}\Big)
\bigl\|U g^{\tilde f^{m_\beta}_\sharp\alpha}_{\tilde f^{m_\beta}_\sharp\beta}
-
g^{\tilde f^{m_\beta}_\sharp\alpha}_{\tilde f^{m_\beta}_\sharp\beta}
\bigr\|_{L^2(f_\sharp\beta)}.
\end{align*}
By definition we jave
$
U g^{\tilde f^{m_\beta}_\sharp\alpha}_{\tilde f^{m_\beta}_\sharp\beta}
-
g^{\tilde f^{m_\beta}_\sharp\alpha}_{\tilde f^{m_\beta}_\sharp\beta}
=
g^{\tilde f^{m_\beta}_\sharp\alpha}_{\tilde f^{m_\beta}_\sharp\beta}\big(\sqrt{w}-1\big).
$
Taking the $L^2$ norm and using \eqref{eq:Hellinger_distance_transition} from \Cref{operator differece lemma}, we obtain the following.
\begin{align*}
\bigl\|U g^{\tilde f^{m_\beta}_\sharp\alpha}_{\tilde f^{m_\beta}_\sharp\beta}
-
g^{\tilde f^{m_\beta}_\sharp\alpha}_{\tilde f^{m_\beta}_\sharp\beta}\bigr\|_{L^2(f_\sharp\beta)}
&=
\Big(\int_{\R^n}
\bigl\|g^{\tilde f^{m_\beta}_\sharp\alpha}_{\tilde f^{m_\beta}_\sharp\beta}\bigr\|^2
(\sqrt{w(y)}-1)^2 d(f_\sharp\beta)(y)\Big)^{1/2}\\
&=
\bigr\|g^{\tilde f^{m_\beta}_\sharp\alpha}_{\tilde f^{m_\beta}_\sharp\beta}\bigr\|
\Big(\int_{\R^n}(\sqrt{w(y)}-1)^2 d(f_\sharp\beta)(y)\Big)^{1/2}\\
&\le \|J_f(m_\beta)\| \|m_\beta-m_\alpha\| \sqrt{2} H(\tilde f^{m_\beta}_\sharp\beta,\ f_\sharp\beta).
\end{align*}
Using \eqref{eq:Bound_item1} \eqref{eq:Bound_item2}, it follows that $|R_{\beta,1}|$ has order $O\big(\frac{\|m_\alpha - m_\beta\|^2}{\delta^2} H(\tilde f^{m_\beta}_\sharp\beta,\ f_\sharp\beta)\big).$
Thus,
\begin{align*}
&\big\langle g^{f_\sharp\alpha}_{f_\sharp\beta}, F^\dagger_{N_{f_\sharp\beta}}\big(g^{f_\sharp\alpha}_{f_\sharp\beta}\big)\big\rangle_{f_\sharp\beta}  - \big\langle g^{\tilde f^{m_\beta}_\sharp\alpha}_{\tilde f^{m_\beta}_\sharp\beta},  F^\dagger_{N_{\tilde f^{m_\beta}_\sharp\beta}} \big(g^{\tilde f^{m_\beta}_\sharp\alpha}_{\tilde f^{m_\beta}_\sharp\beta}\big) \big\rangle_{\tilde f^{m_\beta}_\sharp\beta}\\
&= O(\delta^{p/2-3}\|m_\beta-m_\alpha\|^2)
+
O \left(\frac{\|m_\beta-m_\alpha\|^{2}}{\delta^2}  H(\tilde f^{m_\beta}_\sharp\beta,\ f_\sharp\beta)\right).\qedhere
\end{align*}
\end{proof}

\end{lemma}

\section{Algorithms}
\begin{algorithm}[H]
\caption{We compute the discrete OT displacement vector from the exact optimal transport plan using its barycentric projection.}
\begin{algorithmic}[1]\label{Algo:OT_displacement}
\Statex \textbf{Input:} Two empirical point clouds
\begin{equation*}
X=\{x_r\}_{r=1}^N \subset \mathbb{R}^d,
\qquad
Y=\{y_s\}_{s=1}^N \subset \mathbb{R}^d.
\end{equation*}
\STATE Set the source and target weights
\begin{equation*}
a=\frac{1}{N},
\qquad
b=\frac{1}{N}.
\end{equation*}
\STATE Form the Euclidean cost matrix
\begin{equation*}
C_{rs}=\|x_r-y_s\|^2,
\qquad
r,s=1,\dots,N.
\end{equation*}

\STATE Compute the exact OT plan
\begin{equation*}
\Gamma=(\gamma_{rs})\in\mathbb{R}^{N\times N}
\end{equation*}
using \texttt{ot.emd(a,b,C)} \cite{python2021pot}.

\STATE Compute the row sums
\begin{equation*}
m_r=\sum_{s=1}^N \gamma_{rs},
\qquad
r=1,\dots,N.
\end{equation*}

\STATE Compute the barycentric image of each source point
\begin{equation*}
\widetilde y_r
=
\frac{\sum_{s=1}^N \gamma_{rs} y_s}{m_r},
\qquad
r=1,\dots,N.
\end{equation*}

\STATE Compute the pointwise displacements
\begin{equation*}
d_r=\widetilde y_r-x_r,
\qquad
r=1,\dots,N.
\end{equation*}

\STATE Stack all pointwise displacements into a single vector
\begin{equation*}
g_X^Y=\mathrm{vec}(d_1,\dots,d_N)\in\mathbb{R}^{dN}.
\end{equation*}

\Statex \textbf{Output:} OT displacement from $X$ to $Y$: $g_X^Y\in\mathbb{R}^{dN}$.
\end{algorithmic}
\end{algorithm}

\begin{algorithm}[H]
\caption{Construction of the empirical covariance operator and its pseudo-inverse at a fixed index $i$.}
\begin{algorithmic}[1]\label{algo:Pesudo_Inverse_Operater}
\Statex \textbf{Input:} 
A center point cloud
\begin{equation*}
U_i=\{u_r\}_{r=1}^N \subset \mathbb{R}^d,
\end{equation*}
and a collection of nearby clouds
\begin{equation*}
{U_i}^{(1)},{U_i}^{(2)},\dots,{U_i}^{(m)} \subset \mathbb{R}^d,
\end{equation*}
each containing $N$ sample points.

\STATE For each $k=1,\dots,m$, compute the discrete OT displacement vector (\Cref{Algo:OT_displacement})
\begin{equation*}
g_{U_i}^{U_i^{(k)}} \in \mathbb{R}^{dN}
\end{equation*}
from $U_i$ to ${U_i}^{(k)}$.

\STATE Stack these displacement vectors into the matrix
\begin{equation*}
X_i=
\begin{bmatrix}
 g_{U_i}^{U_i^{(1)}} & g_{U_i}^{U_i^{(2)}} & \cdots & g_{U_i}^{U_i^{(m)}} 
\end{bmatrix}
\in \mathbb{R}^{dN\times m}.
\end{equation*}

\STATE Form the empirical covariance operator
\begin{equation*}
F_i=\frac{1}{m}X_iX_i^\top \in \mathbb{R}^{dN\times dN}.
\end{equation*}

\STATE Compute the singular value decomposition
\begin{equation*}
F_i = U \Sigma V^\top,
\end{equation*}
where
\begin{equation*}
\Sigma=\mathrm{diag}(\sigma_1,\sigma_2,\dots).
\end{equation*}

\STATE Construct the truncated reciprocal singular values by keeping only the first $d$ singular values
\begin{equation*}
\Sigma_d^\dagger=\mathrm{diag} \left(\sigma_1^{-1},\dots,\sigma_d^{-1},0,\dots,0\right).
\end{equation*}

\STATE Define the pseudo-inverse of the empirical covariance operator by
\begin{equation*}
F_i^\dagger = V \Sigma_d^\dagger U^\top.
\end{equation*}

\Statex \textbf{Output:} The empirical covariance operator pseudo-inverse $F_i^\dagger \in \mathbb{R}^{dN\times dN}$.

\end{algorithmic}
\end{algorithm}
\end{document}